\documentclass[authoryear]{elsarticle}

\usepackage{lipsum}
\makeatletter
\def\ps@pprintTitle{%
 \let\@oddhead\@empty
 \let\@evenhead\@empty
 \def\@oddfoot{}%
 \let\@evenfoot\@oddfoot}
\makeatother

\usepackage[T1]{fontenc}
\usepackage[utf8]{inputenc}
\usepackage[english]{babel}

\usepackage{hyphenat}
\usepackage{amssymb}
\usepackage{array}
\usepackage{amsmath}
\usepackage{amsthm}
\usepackage{xcolor}
\usepackage[hidelinks,colorlinks=true,linkcolor=blue,citecolor=blue]{hyperref}
\usepackage{sectsty}
\sectionfont{\fontsize{12}{12}\centering}
\subsectionfont{\fontsize{10}{10}\centering}
\usepackage{color}
\usepackage{graphicx}
\usepackage{mathtools}

\usepackage{enumerate}
\usepackage{fancyhdr}
\usepackage{yfonts}
\DeclareFontFamily{U}{matha}{\hyphenchar\font45}
\DeclareFontShape{U}{matha}{m}{n}{
  <-6> matha5 <6-7> matha6 <7-8> matha7
  <8-9> matha8 <9-10> matha9
  <10-12> matha10 <12-> matha12
  }{}

\DeclareSymbolFont{matha}{U}{matha}{m}{n}
\DeclareMathSymbol{\Lt}{3}{matha}{"CE}\usepackage{bbm}
\usepackage{mathabx}
\usepackage{geometry}
\newtheorem{theorem}{Theorem}[section]
\newtheorem{hypothesis}{Hypothesis}[section]

\newtheorem{observation}{Observation}[section]
\newtheorem{corollary}{Corollary}[theorem]
\newtheorem{definition}{Definition}[section]
\newtheorem{example}{Example}[section]
\newtheorem{lemma}{Lemma}[section]
\newtheorem{proposition}{Proposition}[section]
\usepackage{imakeidx}
\begin{document}

\begin{frontmatter}

\title{Abstract Integration in Net Convergence Structures}

\author[1]{Alexandre Reggiolli Teixeira}
\ead{riccimath123@gmail.com, a163407@dac.unicamp.br}
\affiliation[1]{organization={Departamento de Estatística, Universidade Estadual de Campinas},
            addressline={Rua Sérgio Buarque de Holanda, 651},
            postcode={13083-970},
            city={Campinas-SP},
            country={Brazil}}

\begin{abstract}
   In this article, we propose a general theory of integration of the Riemann and Lebesgue types with respect to arbitrary measures and functions, connected by a continuous bilinear product, with values in abstract vector spaces endowed with a convergence structure given by nets. This covers both the topological and order based convergences in the literature. We then show that this integral satisfies most of the common properties of the objects that comprises integration theory. By establishing a generalized notion of summability on vector spaces with a convergence structure and an integral built upon countable partitions of the base space, we then stablish some uniform, monotone and dominated convergence theorems for the refereed integrals, as well as a non-topological or order based Henstock Lemma and a general convergence theorem based on the notion of conjugated lattice seminorms. An application of these theorems is made to prove various equivalences concerning the Lebesgue, for which we give a brief survey, Saks and Riemann type integrals in partially ordered and topological vector spaces presented in the literature, for which we also make a thorough review. We finish the article with a possible way of classifying general integration procedures  defined in abstract convergence structures, and pose some open problems based on them. 
   
\end{abstract}


\end{frontmatter}

\tableofcontents

\makeatletter
\renewcommand\@biblabel[1]{[#1]}
\makeatother

\section{Introduction, Historical Context and Some Notation.}

Since the pioneering works on integration theory of \cite{pettisoriginalarticle}, \cite{bochneroriginalarticle}, \cite{birkhofforiginalarticle}, \cite{bartle1956general} and others, in the beginning to mid of the last century, a myriad of different integration procedures, with varying generality with respect to algebraic, set and convergence theoretic notions, have been proposed, as can be gauge from the vast bibliography in the classical work of \cite{henstockmagnum}. In almost all cases, these works are constructed on the basis of algebraic and convergence operations on a fixed vector (group/semigroup) space, with the approximation procedure necessary for each integration operation being taken on a fixed topology or order structure. Then, elementary properties of these integrals, such as additivity, preservation of order (isotonocity), subset integration, and others, are obtained using the special approximation procedures - Riemann sums or Lebesgue-type limit arguments - in each case. Further, and more complicated, continuity properties of these are usually constructed by restricting the space of integrable functions or set functions involved, with a special use of convergence sequences of functions or other objects. These are the basic elements of integration theory outlined by \cite{dinculeanustochastic} in his vast treatise on bilinear vector integration, but as can be seen from the recent surveys of \cite{boccutoxenofon} and \cite{boccutovrabelovariecan}, the partially ordered spaces construction is no different. 

In contrast to these early (and modern) trends, we propose to work with general archetypes of integration procedures, such as Lebesgue and Riemann ones, to find general patterns in the definition/approximation procedures of each of these types (of integration) in such a way that the properties cited in the previous paragraphs are natural consequences, in a very general context, of simple and unifying principles having to do with the process itself, and not deriving from simplifying assumptions related to topology or measure theory. 

The first step in this program is to notice that many of the most elementary of these properties, such as additivity of integration or isotonicity in the presence of partial orders, or even the definition of the integration procedure itself, are not dependent on any type of topological or order based convergence, being common to all abstract limiting procedures. In this sense, the most general context of integration theory is one where we have an algebraic object, a vector space for example, endowed with an operation which can be called "limit", and which satisfies some mild restrictions related to the everyday use of such objects. This is the context of general convergence structures, on algebraic objects, in the sense of \cite{doleckiconvergence}, and recently studied by \cite{vanderwaldlocallysolid} and \cite{convergencestructuresvanderwalt}, but already present in some aspects in \cite{Higgs}.

By using these ideas and the fact that the important element in every integration procedure is the way we set approximating nets of varying objects (such as sets and functions or/and combinations of them), the main proposals of integration procedures in this work, which are called the (abstract) net Riemann, abstract Lebesgue and Saks integrals, contained in Definitions \ref{netriemannpartiallyordered}, \ref{lebesgueintegral} and \ref{saksintegral}, are constructed in full generality with respect to their characteristic approximation procedures, generalizing earlier work on the Lebesgue case by \cite{boccutoxenofon}, \cite{sambuciniconvergence}, \cite{ballveeconvergence}. Also, the net Riemann, our main focus, and Saks integrals are, as far as we know, new is this generality, the first one being based on more restrictive constructions of \cite{szaznetintegralconvergence} and \cite{fleischersemigroup}.

In this context, we prove general additivity and isotonicity of the net Riemann integral in Theorem \ref{propertiesnetriemannintegral}, as well as properties having to do with subset integration in Theorem \ref{subsetintegration}, in which only convergence related properties are used (i.e, no topology or order), and integration of simple functions, all new in this generalized framework, and which are established in a way such that most, or all, known integration procedures satisfies them. We also develop, using elements of duality of convergence vector spaces as in \cite{butzmann}, a very general and complete Pettis integration procedure, and prove that it has a Riemmanian representation in terms of the net Riemann integral in Theorem \ref{pettisriemann}, thus generalizing recent results in the stochastic integration literature as in \cite{gelfandstochastic}, and the classical results of \cite{sionsemigroup}. The same is done for the Lebesgue and Saks constructions in Sections \ref{pavlakosintegralsection}, \ref{topologicallebesgueintegralssection} and \ref{saksintegralsection}. 

As a byproduct of such general definitions and properties, we show on Section \ref{examples} that most of the Riemann type constructions of integrals in the literature, such as the classical Henstock-Kurzweil integral in division spaces of \cite{henstockmagnum} and the refinement integrals of \cite{kolmogorovuntersuchungen} and \cite{sionsemigroup}, are special cases of our constructions. On the same section, we also provide a review of these integration procedures in a very general context, and make a list of well-known and also non-classical integration procedures that are contained in our framework, thus showing the (initial, as we shall see) potential of such ideas to unify and generalize integration methods contained in the literature.  

Nevertheless these general constructions, some examples presented in the classical literature of integration theory exemplified by \cite{kolmogorovuntersuchungen}, \cite{goguadzebook}, \cite{deleanu} and \cite{tulceaensemble}, are not contained in them, and necessitate integrals based on infinite Riemann sums. Using the initial attempt of \cite{Higgs} as an inspiration, we develop, in Section \ref{sectionsummabilitykolmogorov}, an abstract summability (infinite series) theory on vector spaces equipped with net convergence structures of the sort cited above, and show not only that the classical dichotomy of unconditional/conditional summability and properties of these series are valid in our abstract work, but also that a Riemann type-integral based on infinite sums, called the $S^{*}$-integral, can be defined in this setting with its usual properties, and which contains the others cited. Relations with the net Riemann integral are also proved.  

To obtain further properties of all these abstract integration procedures, and connect them with more commonly used ones, we also have aimed to prove general continuity of these integrals in terms of different convergence results. Following the usual schemes in the literature, we first separate these theorems in two types: the ones based on order, and the other on topology. For the first case, we obtain a sweeping abstraction of the Lebesgue measure Kurzweil integral convergence theorem of \cite{monteirouniform}, contained in our Theorem \ref{uniformconvergencetheoremriemann}, showing that it is valid for a very general version of the net Riemann integral based on non-negative set functions with values in Riesz Spaces. Similar, we obtain an outside of null sets convergence result for the $S^{*}$, and a restricted version of the net Riemann, integrals that are not contained in the previous literature, and allows us to connect these integration procedures with the Lebesgue construction of \cite{pavlakosintegration} and others presented in Section \ref{pavlakosintegralsection}, which also presents a wide list of integration procedures of such (Lebesgue) type that can be compared to ours in the partially ordered case. In the same context of Riesz spaces, we also obtain very general monotone and dominated convergence theorems for the net Riemann integrals, getting explicit restrictions of our general integration procedure that allows us to obtain these theorems. Similarly, in the topological case, we prove a type of (Egorov) quasi-uniform convergence theorem in the context of the net Riemann integral, generalizing earlier work of \cite{sionsemigroup}, \cite{trombettalimit} and \cite{millingtonproduct}, and comment on further results contained in the literature. 

In the same backdrop of continuity of abstract integration, we prove two further results that are of special interest: the Henstock lemma in Theorem \ref{henstocklemma} and the general convergence theorem in Theorem \ref{latticenormconvergencetheorem}. Both of these are proved for functions and measures with values in S-metric, or lattice normed, spaces, a generalization of topological and partially ordered Riesz spaces with general convergences. These are, as far as we can tell, not present in any other work of integration theory, and manifest the fact that convergence theorems can be obtained even if we have no topology or order involved in our integration procedure, and contains seeds for further and more general convergence theorems of this non topological/order type. 

Finally, we make connections between the main integral archetypes presented, and specific examples in the literature. In fact, we solve a series of open problems in the comparison of integration procedures that were left open in the literature, from which we highlight the following:

\begin{enumerate} [(i)]
    \item In the AMS review (MR0450503) of \cite{sionsemigroup} book, Z. Lipecki cited the open problem of comparing Sion integral with the non-locally convex construction of \cite{Turpin2}. In Section \ref{bdsintegralsection}, we show that the more general integral of \cite{drewlabuda} is a special case of Sion integral, thus answering Lipecki's question.
    \item In a similar context, \cite{pavlakosintegration} and \cite{popescurickart} both left open the question of comparing Lebesgue integration on partially ordered spaces with Riemannian constructions in the same setting. We do this comparison in Section \ref{pavlakosintegralsection} using the convergence theorems cited above. In particular, we prove that both integrals are a special case of the $S^{*}$ integral in the context of infinite summability, which uses in the proof Theorem \ref{convergenceoutsidenullset}, not obtained previously in the literature.
    \item In his book \cite{sionsemigroup}, M. Sion left open the question of comparing a version of his truncation/refinement integral with the $S^{*}$ integral. We show in Section \ref{sectionsummabilitykolmogorov} that in the topological and partially ordered cases, Sion's procedure is more general, and contains the refereed infinite Riemann sums integral. 
    \item Lastly, \cite{massesemigroup} left open the question of comparing his semigroup submeasure integral with \cite{sionsemigroup} integration procedure. We solve this problem partially in Theorem \ref{sionandmasseintegrals}, showing that in the presence of a control measure in a metric linear space, the two integrals are equivalent. 

\end{enumerate}

We finish the work by showing that our general definition of the Saks integration procedure contains not only the net Riemann integral, and therefore all of the integration theories of above, but also incorporates non-additive integration procedures such as the ones in \cite{boccutoriecanconcave}. A list of open problems follow, which shows that further work on our general integration procedure can lead to interesting new ideas not only on measure and integration, but also on topology and related areas. 

We now pass to some formal comments on references and notation.

For information about the topological, algebraic and measure-theoretic concepts that will be used in the context of integration theory in this article, we cite, for the reader convenience, \cite{mchsaneorder}, \cite{fremlintopologicalrieszspaces}, \cite{vulikh}, \cite{zaanenluxemburg1}, \cite{zaanen2}, \cite{gahler1}, \cite{gahler2}, \cite{boccutorieszabstract}, \cite{boccutosambucini2023}, \cite{boccutocandeloro2003a} and, specially for integration, the more recent publications of
\cite{lvaluedintegrationvanderwalt}, \cite{orderintegrals} and \cite{jeujiang}.

In in this work, we will, in general, denote spaces of any type with uppercase Latin letters $X,Y, Z$, or $\mathcal{X}$, $\mathcal{Y}$, $\mathcal{Z}$ when there is duplicity. As a rule, all partial orders in the refereed spaces, when they have one, will be denoted by $\leq$, following the custom of classical and recent works on spaces equipped with partial orders, as \cite{vulikh}, \cite{zaanenluxemburg1}, \cite{popovriesz}, and \cite{prerieszspaces}, as we will not work with more than one partial order, in a fixed space, at a time (therefore, minimizing the possibility of confusion). Differentiation between order notations will be made only when necessary. As usual, we use the notation $\max(x,-x) = x \vee -x = |x|, x^{+} = \max(x,0), x^{-} = \max(-x,0)$ for the, in order, absolute value, positive and negative parts of an element $x$ of a Riesz space (vector lattice). It will also be supposed that every partially ordered space below is Archimedean, and that every topological space is Hausdorff, and we only cite explicitly these properties again in case we need to stress their presence - for more details on these notions, see \cite{zaanenluxemburg1} and \cite{kelley}. Finally, all algebraic operations shall be supposed to be commutative, and all vector spaces will be assume to be real (i.e, vector spaces with respect to the field $\mathbb{R}$).

With regards to notation, for a general set $A$, following \cite{bourbaki2013general}, we denote by $\mathcal{F}(A)$ the collection of all its finite subsets (which is a directed set by inclusion $\subseteq$). We also have the following set-theoretic definition (see, for example, \cite{rao2018measure}):

\begin{definition} [Paved Space]

We say that a pair $(\Omega, \mathcal{H})$ consisting of an arbitrary (non-empty) set and a collection of its subsets is a Paved Space if $\emptyset \in \mathcal{H}$.
    
\end{definition}

For the measure and integration results below, we shall need a variety of different convergence notions used in the literature, sometimes not equivalently defined in all works, and even not equivalent in general. In this case, comparison of our results with the literature should be done considering in detail the definitions below. We usually adopt the most general one to avoid problems of compatibility, and point out equivalences (and differences) wherever possible. As the topological notions are fairly unified (see, for example, the classics by \cite{bourbaki2013general} and \cite{kelley}), we focus on the partially ordered spaces case, for which convergences are more numerous in the literature. 

\subsection{Partially Ordered Spaces and Order Convergence.}

Considering the context above, we begin by the following definition given in \cite{pavlakosintegration} and \cite{prerieszspaces}, which originally is due to  \cite{vulikh}, called $(o_{3})$-convergence:

\begin{definition} [$(o_{3})$- Convergence of Nets] \label{pavlakosorderconvergence}

Let $X$ be a partially ordered set and $I$, $C$, and $D$ directed sets. We say that a net $\{x_{i}\}_{i \in I}$ in $X$ $(o_{3})$-converges to $x$ in $X$, which we denote by:

\[(o_{3})-lim_{i}x_{i} = x,\]

if there are directed sets $C,D$ and an increasing net $\{z_{c}\}_{c \in C}$ and a decreasing net $\{y_{d}\}_{d \in D}$ in $X$ such that:

\begin{enumerate} [(i)]
    \item $\sup\{ z_{c}:c \in C\} = x = \inf\{y_{d}:d \in D\}$ 
    \item For every $(c,d) \in C \times D$ there exist $i^{*} \in I$ such that:

    \[z_{c} \leq x_{i} \leq y_{d}, \ \text{whenever} \ i \geq i^{*}\]
\end{enumerate}

\end{definition}

Sometimes we may get a simplification by obtaining increasing and decreasing nets in the definition above with same indexing directed set, and then get the following definition of \cite{convergencestructuresvanderwalt} and \cite{vanderwaldlocallysolid} (see also \cite{prerieszspaces}):

\begin{definition} [Order Convergence of Nets] \label{usualorderconvergence}

Let $X$ be a partially ordered set and $I$ a directed set. We say that a net $\{x_{i}\}_{i \in I}$ in $X$ order converges, or $(o)$-converges, to $x$ in $X$, which we denote by:

\[(o)-lim_{i}x_{i} = x,\]

if there is a directed set $C$ and an increasing net $\{z_{c}\}_{c \in C}$ and a decreasing net $\{y_{c}\}_{c \in C}$ in $X$ such that:

\begin{enumerate} [(i)]
    \item $\sup\{ z_{c}:c \in C\} = x = \inf\{y_{c}:c \in C\}$ 
    \item For every $(c,c) \in C \times C$ there exist $i^{*} \in I$ such that:

    \[z_{c} \leq x_{i} \leq y_{c}, \ \text{whenever} \ i \geq i^{*}\]
\end{enumerate}

\end{definition}

For the rest of this work, when we say that a net is bounded in a partially ordered space, we mean:

\begin{definition} [Bounded Nets and Absolute value of Functions]

Let $X$ be a Riesz Space and $\{x_{i}\}_{i \in I}$  a net in $X$. We say that this net is bounded if and only if there exists an $y \in X$ such that:

\[|x_{i}| \leq y,\]

for every $i \in I$.

Also, for a function $f: \Omega \rightarrow X$, where $\Omega$ is a non-empty set, we shall denote by $|f|: \Omega \rightarrow X^{+}$ the function given by:

\[|f|(\omega) = |f(\omega)|,\]

for each $\omega \in \Omega$.
    
\end{definition}

The following characterization of order convergent nets will be useful later (see \cite{monteirouniform} and \cite{mcgillitnegrationinvectorlattices}):

\begin{proposition} \label{orderconvergenceboundednets}
    Let $X$ be a Riesz Space. Then, a net $\{x_{i}\}_{i \in I}$ is $(o)$-convergent to $x \in X$ if and only if it is bounded and:

    \[\limsup_{i}x_{i}= \bigwedge_{i \in I} \bigvee_{i^{'} \geq i}x_{i^{'}} = \bigvee_{i \in I} \bigwedge_{i^{'} \geq i} x_{i^{'}} = \liminf_{i}x_{i}.\]

\end{proposition}

For Riesz Spaces, we have the following result of \cite{pavlakosintegration} and \cite{papangelouorder}, which is also given in \cite{convergencestructuresvanderwalt} and \cite{vanderwaldlocallysolid} for Riesz spaces:

\begin{proposition}
Let $X$ be a Riesz Space and $(I, \alpha)$ a directed set. A  net $\{x_{i}\}_{i \in I}$ is  $(o)$-convergent in $X$ if and only if it is $o_{3}$-convergent, which occurs if and only if there is a decreasing net $\{y_{d}\}_{d \in D}$ in $X$, with $D$ a directed set, such that $y_{d} \downarrow 0$ and for every $d \in D$ there exists $i^{*} \in I$ such that:

\[|x_{i}-x| \leq y_{d} \ ,\]

for every $i \geq i^{*}$.
\end{proposition}

Most of the time we will work with objects less general than lattice ordered groups, e.g. Riesz spaces, in such a way that all of the previous convergence definitions, including the one given in the proposition above, are equivalent. Also, we can show, at least for sequences, that the following result is true (see \cite[Section 3.7, p. 184]{prerieszspaces}):

\begin{proposition}
    Let $X$ be a Riesz Space. A sequence $\{x_{n}\}_{n \in \mathbb{N}}$ in $X$ is $(o_{3}),o$-convergent to $x$ in $X$ if and only if there is a non-increasing sequence $\{u_{n}\}_{n \in \mathbb{N}}$ in $X$ such that $u_{n} \downarrow 0$ and:

    \[|x_{n}-x| \leq u_{n},\]

    for every $n \in \mathbb{N}$.
\end{proposition}

The sequential nature of the result above is essential, as "tail" type convergence results for nets are not usually not true. For more about this subject, see Section 3.7 of \cite{prerieszspaces}.

For various proofs in the analytical theory of partially ordered spaces, one hard to work with property of the previous convergences is that $\varepsilon$-type arguments have no analogue. This complicates some proofs related to Cauchy-type conditions and countable sequence and subsequence arguments. To remediate this fact, a different type of convergence, usually spoken as a type of "regulator" or $(D)$-convergence was created, and plays a major role in classical texts, such as  \cite{zaanen2} and \cite{zaanenluxemburg1}, as well as in the vast literature in integration and convergence theory in partially ordered spaces, as given by the surveys of \cite{boccutoxenofon}, \cite{boccutovrabelovariecan}, and the recent works of \cite{convergencestructuresvanderwalt} and \cite{vanderwaldlocallysolid}, and their collaborators.

In this case we introduce the  following structure associated with $(D)$-convergence, as in \cite{boccutovrabelovariecan}:

\begin{definition} [$(D)$-sequence or Regulator]

Let $X$ be a Riesz space. We say that a double sequence $\{a_{ij}\}_{(i,j) \in \mathbb{N} \times \mathbb{N}}$\footnote{From now on, except when strictly needed, we will represent a regulator only as $\{a_{ij}\}_{ij}$ or $\{a_{ij}\}$.} in $X$ is a $(D)$-sequence, or regulator, if $a_{ij} \geq a_{ij+1}$ for every $i,n \in \mathbb{N}$ and:

\[ \bigwedge_{j=1}^{\infty} a_{ij} =0, \ \forall i \in \mathbb{N}\]
    
\end{definition}

It follows that every $(D)$-sequence is non-negative: i.e, for every $i,j$, $a_{ij} \geq 0$.

Then, we have the following definition:

\begin{definition} [(D)-convergence] \label{dconvergence}

Let $X$ be a Riesz Space and $I$ a directed set. We say that a net $\{x_{i}\}_{i \in I}$ in $X$ is $(D)$-convergent to an element $x \in X$ if there exists a $(D)$-sequence $\{a_{ij}\}$ in $X$ such that for every $\varphi \in \mathbb{N}^{\mathbb{N}}$ there exists $i_{0} \in I$ such that:

\[|x_{i}-x| \leq \bigvee_{j=1}^{\infty}a_{j\varphi(j)},\]

for every $i \geq i_{0}$.

\end{definition}

In general, we have the following relations between this convergence and the previously defined ones (see Chapter 2 of \cite{boccutovrabelovariecan}):

\begin{proposition} \label{(D)ando}
    Let $X$ be an Riesz Space. Then, a $(o)$-convergent sequence $\{x_{n}\}_{n \in \mathbb{N}}$ in $X$ is $(D)$-convergent.  
\end{proposition}

In general, the reverse statement of convergence is not true - see, for example, \cite{zaanenluxemburg1}.  The result for nets is also not, in general, true - see again Chapter 2 of \cite{boccutovrabelovariecan}. 

As indicated earlier, $(D)$-convergence can give, in some cases, a substitute for $\varepsilon \setminus 2^{n}$-type arguments in terms of (one version of) the Fremlin Lemma (see \cite{integralmeasureandordering}), for which we need the next notions, beginning with the essential concept from the integration theory of \cite{matthesdaniell} and \cite{wrightextension}.

\begin{definition} [Weak $\sigma$-Distributivity]

Let $X$ be a Riesz Space. We say that $X$ is weakly $\sigma$-distributive if and only if:

\[\bigwedge_{\varphi \in \mathbb{N}^{\mathbb{N}}}(\bigvee_{i=1}^{\infty}a_{i\varphi(i)}) = 0,\]

for every $(D)$-sequence $\{a_{ij}\}$ in $X$.

\end{definition}

As an interesting consequence (which we shall use later in the context of integration) of this property, we have (see \cite[Proposition 2.22, p. 31]{boccutovrabelovariecan}):

\begin{proposition} \label{equivalenceoand(D)convergence}
    Let $X$ be an Archimedean  and weakly $\sigma$-distributive Riesz Space. Then, a $(D)$-convergent sequence $\{x_{n}\}_{n \in \mathbb{N}}$ in $X$ is $(o)$-convergent.  
\end{proposition}

The last definition we shall need for the Fremlin Lemma is the next classical concept, which can be found in most books about partially ordered spaces (see, for example, \cite{vulikh} or \cite{zaanenluxemburg1}):

\begin{definition} [Dedekind Completeness] \label{dedekindcomplete}

Let $X$ be a partially ordered space. We say that $X$ is Dedekind complete if every upper bounded set has a supremum. Equivalently, if every increasing bounded net has an $(o)$-limit (given by its supremum). 
    
\end{definition}

For further properties relating to Dedekind completeness, see Chapter 4 of \cite{zaanenluxemburg1}.

We now have the main lemma we alluded to, and which finishes this section:

\begin{theorem} [Fremlin Lemma]

Let $X$ be a Dedekind $\sigma$-complete weakly $\sigma$-distributive Riesz space. If $\{a_{nij}\}_{nij}$ is a triple sequence of elements of $X$ such that, for each $n \in \mathbb{N}$, $\{a_{nij}\}_{ij}$ is a $(D)$-sequence, then there is a double sequence $\{b_{ij
}\}_{ij}$ of elements of $X$ such that, for all $L \in X^{+}\setminus \{0\}$ and $\varphi \in \mathbb{N}^{\mathbb{N}}$, 

\[L \wedge (\sum_{n=1}^{k} \bigvee_{i=1}^{\infty} a_{ni\varphi(i+n)})\leq \bigvee_{j=1}^{\infty}(L \wedge b_{j\varphi(j)}),\]

for all $k \in \mathbb{N}$.

Moreover, $\{S \wedge b_{ij}\}_{ij}$ is a $(D)$-sequence for all $L \in X^{+}\setminus \{0\}$.

\end{theorem}

A detailed proof of this important result can be found in \cite{monteirodissertação}.

\subsection{Lattice Normed Spaces.}

For the development of our integration theory, besides the partially ordered and topological cases, we shall also need some basic elements of convergence in spaces equipped with abstract distance functions that are valued in (a collection of) partially ordered vector spaces. This idea has been considered in various forms in the literature, and we shall drawn mainly from \cite{kusraevdominated} and \cite{kusraevmalyugin}, but generalizing the construction to cover non-locally convex and other types of spaces. A similar notion in the topological literature is the category of S-metric spaces, studied mainly by \cite{reichel1977distanzfunktionen} and \cite{reichel2006nearness}.

We begin with the following definition:

\begin{definition} [Lattice Norm]

Let $X$ be a (real) vector space and $\mathcal{X}$ a Riesz space and $\mathcal{X}^{+}$ its positive cone. Then, a mapping $||\cdot||: X \rightarrow \mathcal{X}^{+}$ is a lattice norm of $X$ with values in $\mathcal{X}$ if it satisfies the following for every $x,y \in X$:

\begin{enumerate}[(i)]
    \item $||x|| = 0 \iff x =0$
    \item $||x+y|| \leq ||x||+||y||$
\end{enumerate}

\end{definition}

With this, we now define the main structure associated to a lattice norm:

\begin{definition} [Lattice Normed Space]

We say that a triple $(X, \{||\cdot||_{l \in L}, \{\mathcal{X}\}_{l \in L})$, where $L$ is an index set, $||\cdot||_{l}: X \rightarrow \mathcal{X}^{+}_{l}$ a lattice norm for every $l \in L$ and $\mathcal{X}_{l}$ a Riesz space for every $l \in L$, is a lattice normed space, or only $X$ if no confusion arises.
\end{definition}

As in Chapter 2 of \cite{kusraevdominated}, this construction wields a natural notion of convergence in $X$:

\begin{definition} [Convergence in Lattice Normed Spaces]

We say that a net $\{x_{i}\}_{i}$ of elements of a lattice normed space $(X, \{||\cdot||_{l \in L}, \{\mathcal{X}\}_{l \in L})$ is convergent to an element $x \in X$ if and only if for each $l \in L$:

\[||x - x_{i}||_{l} \rightarrow 0,\]

where this convergence is $(o)$ or $(D)$-convergence in $\mathcal{X}_{l}$\footnote{Any other net convergence structure is admitted as well, see the next section.}.

We then say that the lattice norm $||\cdot||_{l}$ is continuous if:

\[||x||_{l} = \lim_{I} ||x_{i}||_{l},\]

with $\lim$ being taken in the sense above. This convergence will also be said to be complete if any Cauchy net for this convergence is convergent in $X$\footnote{See also the next section for more details about Cauchy concepts in general convergences.}.
    
\end{definition}

For the fact that these two constructions generalize topological as well as Riesz space convergences, we refer the reader to \cite{kusraevdominated} and the paper \cite{reichel2006nearness}.

We shall now pass to a discussion of a general structure that unifies all the presented convergence concepts so far. 

\subsection{Net Convergence Structures in General Spaces.}

To unify the different notions of order-based and topological convergences above, we shall need a concept that abstracts the notion of convergence independently of the presence of a topology or order in the space in question. Various notions of this idea exists, with varying generality, in the works \cite{gahler1}, \cite{gahler2}, \cite{dolecki}, \cite{doleckiconvergence}, \cite{Higgs} and others. Usually, this notion of "convergence structure" is presented in the language of filters, which are more adapted to topology. As our focus is integration, nets are more intuitive and useful. Therefore, we shall use the recent contributions of \cite{vanderwaldlocallysolid} and \cite{convergencestructuresvanderwalt}, which developed this theory in the context of nets. 

In this sense, following \cite{convergencestructuresvanderwalt}, we fix the space $\mathcal{N}(X)$ as the set of all admissible nets in an abstract (non-empty) space $X$. The concept of an admissible net is of a technical nature (based on the axioms of ZFC, as explained in \cite{convergencestructuresvanderwalt}), which we are not going to emphasize much - it is always possible to find, for each net, an admissible net such that the two have equal limit in the sense of the next definition. So, we shall work only with this class (i.e, identify nets with its admissible class) and comment no further.

In the rest of this work $I$ and $L$ will, in general, denote general directed (upwards or downwards) sets that are to serve as the domain of the nets. We shall not refrain to use other (Latin) indexes when these two are already in use in the same context.

We begin with the following preliminary concept (see Page 4 of \cite{convergencestructuresvanderwalt}):

\begin{definition} [Quasi-Subnet]

Given two nets $\{x_{i}\}_{i \in I}$ and $\{y_{i^{'}}\}_{I^{'}}$, we say that the latter is a quasi-subnet of the former if for every $i_{0} \in I$ there exists $i^{'}_{0} \in I^{'}$ such that:

\[\{y_{i^{'}}\}_{i^{'}\geq i^{'}_{0}} \subseteq \{x_{i}\}_{i \geq i_{0}}\]

\end{definition}

Then, we have the following main definition from \cite{convergencestructuresvanderwalt}. 

\begin{definition} [Net Convergence Structure] \label{netconvergencestructure}

A net convergence structure on $X$ is a function

\[\eta: X \rightarrow 2^{\mathcal{N}(X)},\]

which is defined by the relation $\{x_{i}\}_{i \in I} \in \eta(x)$, written sometimes as the limit notation $\{x_{i}\}_{i \in I} \rightarrow x$,  $\lim_{(I, \eta)}x_{i} = x$ or $\lim_{i}x_{i} = x$, such that $x$ is called the limit of the net, and:

\begin{enumerate} [(i)]
    \item If a net $\{x_{i}\}_{i \in I}$ is constant and equals to a $x \in X$, then $ \{x_{i}\}_{i \in I} \in \eta(x)$ and its limit is $x$;
    \item If a net converges to $x$, every quasi-subnet of it converges to $x$;
    \item Let $\{z_{i}\}_{i \in I} $ be a net such that $z_{i} \in \{x_{i},y_{i}\}$, for every $i \in I$, with nets $\{x_{i}\}_{i \in I}$ and $\{y_{i}\}_{i \in I}$ such that $\{x_{i}\}_{i \in I} \rightarrow x$ and $\{y_{i}\}_{i \in I} \rightarrow x$. Then $\{z_{i}\}_{i \in I} \rightarrow x$.
\end{enumerate}

We call $(X, \eta)$, or only $X$ if the convergence is implicit, a net convergence space.
    
\end{definition}

If the space in question has a partial order, we also have the following analogous construction of \cite{vanderwaldlocallysolid}, which is a generalization of the various order convergences above, as well as of  relative uniform, $(uo), (un)$ and $(uaw)$ convergences:

\begin{definition} [Locally Solid or Partially Ordered Net Convergence Structure]

Given a net convergence space $(X, \eta)$ where $X$ is a Riesz space, we say that $\eta$ is locally solid if, besides the axioms of Definition \ref{netconvergencestructure}, it also satisfies:

\begin{enumerate}
    \item If $\{y_{i}\}_{i \in I}$ converges to $y$ in $X$ and $\{x_{i}\}_{i \in I}$ is a net such that:

    \[|x_{i}| \leq |y_{i}|,\]

    for every $i \in I$, then $\{x_{i}\}_{i \in I}$ converges to $y$.
    \item If $\{y_{l}\}_{l \in L}$ converges to $y$ in $X$ and $\{x_{i}\}_{i \in I}$ is a net such that for every $l_{0} \in I$ there exists $i_{0} \in L$ such that for every $i \geq i_{0}$ there exists $l \geq l_{0}$ with $|x_{i}| \leq |y_{l}|$, then $\{x_{i}\}_{i \in I}$ converges to $y$. 
\end{enumerate}

We shall also say that $(X,\eta)$ is, in this case, a partially ordered net convergence structure.
    
\end{definition}

We now have an analogous definition of the Hausdorff separation property:

\begin{definition} [Hausdorff Net Convergence Spaces]
    We say that a net convergence space $(X, \eta)$ is Hausdorff if each net has only one limit. 
\end{definition}

We now add the following hypothesis for the rest of this work:

\begin{hypothesis} [Hausdorff Hypothesis]

From now, all the net convergence spaces in this work shall be supposed, or proved, to be Hausdorff. 
    
\end{hypothesis}

For further comparison with other convergences, the following result of \cite{convergencestructuresvanderwalt} is of interest:

\begin{proposition}
    The Axiom 13.2.1 (ii) is equivalent to the following two conditions:

    \begin{enumerate} 
        \item If a net converges to $x$, then so does each of its subnets and,
        \item If $\{x_{i}\}_{i \in I} $ converges to $x$ and $\{y_{i^{'}}\}_{i^{'} \in I^{'}}$ is tail equivalent to $\{x_{i}\}_{i \in I}$, then $\{y_{i^{'}}\}_{i^{'} \in I^{'}}$ converges to $x$.
    \end{enumerate}
\end{proposition}

Thus, this result shows that this general convergence concept includes the net convergence definition of \cite{Higgs}. It can also be seen that it includes the net convergences of \cite{boccutoxenofon}, based on axioms of pages 306-307 from it, and \cite{boccutocandeloro2003a}, as well as the other convergences cited above, which also have filter analogues.

\begin{definition} [Cauchy Nets and Complete Net Convergence Structures]

Let $(X, \eta)$ be a net convergence space. We say that a net $\{x_{i}\}_{i \in I}$ is Cauchy if its associated difference

\[\{x_{i}-x_{j}\}_{(i,j) \in I \times I},\]

is convergent to $0$ in the net convergence structure $\eta$.

If each Cauchy net in $X$ associated to $\eta$ is convergent\footnote{Notice that every convergent sequence is Cauchy in an arbitrary net convergence structure. For details, see \cite{convergencestructuresvanderwalt}.}, the convergence structure $\eta$ is called Cauchy or, equivalently, $X$ is called complete (in reference to the net convergence structure $\eta$).
\end{definition}

Continuity is an important concept, and it can be generalized to this case (see \cite{convergencestructuresvanderwalt} and \cite{vanderwaldlocallysolid}):

\begin{definition} [Continuous Maps Between Convergent Spaces]

Let $X$ and $Y$ be net convergence spaces with net convergence structures $\eta_{1}$ and $\eta_{2}$ respectively. We say that a map $f: X \rightarrow Y$ is continuous if and only if:

\[\{x_{i}\}_{i \in I} \xrightarrow{\eta_{1}} x \implies \{(f(x_{i})\}_{i \in I} \xrightarrow{\eta_{2}} f(x), \]

for every net convergent net $\{x_{i}\}_{i \in I}$ in $X$.
\end{definition}

With regards to this last definition, we shall add to the above the following two hypothesis for the rest of this work:

\begin{hypothesis} [Hypothesis and Definition - Continuity of Addition for Vector Spaces and Product Convergence Structure] \label{productconvergence}

Let $(X, \eta)$ be a net convergence space with $X$ being a vector space. For the rest of this work, we shall suppose that the operation of addition and multiplication by real numbers for each fixed $r \in \mathbb{R}$ in $X$ are continuous. That is:

\[
\setlength\arraycolsep{0pt}
+\colon \begin{array}[t]{ >{\displaystyle}r >{{}}c<{{}}  >{\displaystyle}l } 
          X \times X &\to& X \\ 
          (x,y) &\mapsto& x+y 
         \end{array}
\]

\[
\setlength\arraycolsep{0pt}
\cdot_{r}\colon \begin{array}[t]{ >{\displaystyle}r >{{}}c<{{}}  >{\displaystyle}l } 
          X &\to& X \\ 
          x &\mapsto& r 
         \end{array}
\]    

are both continuous, where $X \times X$ is equipped with what we call the product convergence structure, that is: a net $\{x_{i}\}_{i \in I}$ of the form $x_{i} = (x_{i}^{1},x_{i}^{2})$ is convergence to $x = (x_{1},x_{2})$ in $X$ if and only if the marginal nets $\{x_{i}^{1}\}_{i \in I}$ and $\{x_{i}^{2}\}_{i \in I}$ converge to $x_{1}$ and $x_{2}$ in $X$ with respect to $\eta$ respectively. The product convergence structure in an infinite (possibly uncountable) product is defined in the analogous (i.e, convergence in each projection) way, with addition and multiplication by (real) constants understood in the coordinate-to-coordinate sense.

This is related to the concept of linear convergence structure as in \cite{convergencestructuresvanderwalt}.
\end{hypothesis}

\begin{hypothesis} [Hypothesis - Locally Solid Convergence]

Let $(X, \eta)$ be a net convergence space with $X$ being a Riesz space. For the rest of this work, we shall suppose that the the convergence structure $\eta$ in $X$ is locally solid.

\end{hypothesis}

To connect these concepts with the ones in the previous section, we note the next result due to \cite{mcgillitnegrationinvectorlattices}:

\begin{proposition} [Order Convergence is Cauchy]
    In a Dedekind complete Riesz space $X$, $(o)$-convergence is Cauchy. 
\end{proposition}

Similarly, we have:

\begin{proposition} [(D)-convergence is Cauchy]
        In a Riesz space $X$ that is weakly $\sigma$-distributive and Dedekind complete, (D)-convergence is Cauchy. 

        In other words, a net $\{x_{i}\}_{i \in I}$ in $X$ is $(D)$-convergent if and only if there exists a non-negative bounded double sequence $\{a_{nk}\}_{nk}  \subseteq X$ such that $a_{nk} \downarrow 0$ ($k \rightarrow \infty, \forall n \in \mathbb{N}$) and for every $\varphi \in \mathbb{N}^{\mathbb{N
        }}$ $\exists i_{0} \in I$ such that:

        \[|x_{i_{1}}-x_{i_{2}}| \leq \bigvee_{n=1}^{\infty}a_{i\varphi(i)},\]

        for every $i_{1},i_{2} \geq i_{0}$.
\end{proposition}

As we have not found in the literature a proof for this result for general nets, we supply a proof inspired in the result for integrals in \cite{riecanvrabelovaoperator}. 

\begin{proof}
    The necessity is clear.

    For the sufficiency, let $\{a_{nk}\}_{nk}  \subseteq X$ be a non-negative bounded double sequence  such that $a_{nk} \downarrow 0$ ($k \rightarrow \infty, \forall n \in \mathbb{N}$) and for every $\varphi \in \mathbb{N}^{\mathbb{N
        }}$ $\exists i_{0} \in I$ such that:

        \[|x_{i_{1}}-x_{i_{2}}| \leq \bigvee_{n=1}^{\infty}a_{n\varphi(n)},\]

        for every $i_{1},i_{2} \geq i_{0}$.

    Define the following index set:

    \[J = \{i \in I: \exists \varphi \in \mathbb{N}^{\mathbb{N}}, i=i(\varphi)\}.\]

    By definition and construction, for $i^{*} \in J$, the set

    \[\{x_{i}: i \geq i^{*}\},\]

    is a bounded set, as for $i \geq i^{*}$,

    \[x_{i} \leq \bigvee_{n=1}^{\infty}a_{nk}+|x_{i^{*}}|,\]

    which shows the boundedness.

    As $X$ is Dedekind complete, there exists the following quantities, following again the notation above:

    \[a_{i^{*}} = \inf_{i \geq i^{*}}x_{i},\]

    \[b_{i^{*}}= \inf_{i \geq i^{*}}x_{i}.\]

    Put, then, for arbitrary $i_{1},i_{2} \in J$, $i^{*} = \max(i_{1},i_{2})$ (note that it satisfies the construction above). By construction, $\{x_{i}: i \geq i^{*}\}$ is bounded and:

    \[a_{i_{1}} = \inf_{i \geq i_{1}}x_{i} \leq \inf_{i \geq i^{*}}x_{i} \leq x_{i} \leq \sup_{i \geq i^{*}}x_{i} \leq \sup_{i \geq i_{2}}x_{i} = b_{i_{2}},\]

    Therefore, 

    \[\sup_{i \in I}a_{i} \leq \inf_{i \in I}b_{i}.\]

    Consequently, there exists $x \in X$ such that

    \[a_{i} \leq y \leq b_{i},\]

    for every $i \in J$. Now, let $\varphi \in \mathbb{N}^{\mathbb{N}}$ arbitrary. By construction, there exists $i^{*}=i^{*}(\varphi) \in I$ such that

    \[x_{i_{1}} \leq x_{i_{2}} + \bigvee_{n=1}^{\infty}a_{n\varphi(n)},\]

    for every $i_{1},i_{2} \geq i^{*}$. Fix such $i_{2}$. Then, taking the supremum with respect to $i_{1}$,

    \[b_{i^{*}(\varphi)} \leq x_{i_{2}} + \bigvee_{n=1}^{\infty}a_{n\varphi(n)}.\]

    As a consequence of the weak $\sigma$-distributivity, $\bigwedge_{\varphi \in \mathbb{N}^{\mathbb{N}}}\bigvee_{n=1}^{\infty}a_{n\varphi(n)} = 0$, and therefore

    \[\bigwedge_{\varphi \in \mathbb{N}^{\mathbb{N}}}b_{i^{*}(\varphi)} - \bigvee_{\varphi \in \mathbb{N}^{\mathbb{N}}}a_{i^{*}(\varphi)} = \bigwedge_{\varphi \in \mathbb{N}^{\mathbb{N}}}(b_{i^{*}(\varphi)}-a_{i^{*}(\varphi)}) = 0.\]

    On that account, 

    \[x = \bigwedge_{\varphi \in \mathbb{N}^{\mathbb{N}}}b_{i^{*}(\varphi)} = \bigvee_{\varphi \in \mathbb{N}^{\mathbb{N}}}a_{i^{*}(\varphi)}.\]

    As a consequence, for every $i \geq i^{*}(\varphi)$,

    \[x_{i} - x \leq b_{i^{*}(\varphi)} - a_{i^{*}(\varphi)}  \leq \bigvee_{n=1}^{\infty}a_{n\varphi(n)}.\]

    Similarly for the opposing inequality, in such a way that we obtain, finally, 

    \[|x_{i} - x| \leq \bigvee_{n=1}^{\infty}a_{n\varphi(n)},\]

    which is the convergence of $x_{i}$ to $x$, with respect to $(D)$-convergence, in $X$.

\end{proof}

This gives all we need to define the main integration procedure of this work, which we now do. 

\section{The Abstract Net Riemann Integral.} \label{sectionnetriemannintegral}

The convergence ideas above can give rise to a very general integration procedure when combined with an algebraic element of some sort. Basically, this combination of an algebraic operation and a convergence concept encapsulates the notion that defining a Riemann type integral is a matter of having a notion of well defined sums (or products) and an approximation procedure. That is, fixing $(X, \eta, \star)$ a net convergence structure with an operation $\star$ compatible with the convergence (in some sense), we may consider integration (of a function against a measure) as a limit of a net/filter of formal sums of the form:

\[(f, \mu) \rightarrow \sum_{i \in I_{\gamma}}f(\tau)\mu(\sigma),\]

where $f$ and $\mu$ have values in spaces that can be connected with $X$ by some map (not necessarily linear).

We now follow this intuition to define a very general integral that has some interesting properties, and which was considered, in a less general setting and with a narrower definition in \cite{szaznetintegralconvergence} and \cite{fleischerchange}. After defining our general integration procedure, we shown that most of the integrals in the literature are special cases of ours, and we also generalize specific convergence/additivity theorems of these integrals to the present abstraction. We present most of the results in the language of nets, but the dual filter notation is also be useful. 

As we will be dealing with different types of integrals in this work, we identify them by an superscript consisting of an upper case letter with a direct summary of the name, e.g. $\int^R$ for the Riemann integral. The only exception will be the main integral of this section, which will be denoted without superscript. 

In general, we will be integrating measures and functions with values in different spaces in conjunction to net convergence structures in each of these. Therefore, a concatenation structure that connects the two to a third space equipped with a convergence structure of the cited type, with the additional requirement of continuity of the algebraic operations, is needed. More precisely, we suppose the existence of a "bilinear product" as in \cite{pavlakosintegration}, \cite{dinculeanu2014vector}, \cite{dinculeanustochastic} and other classics of vector integration (see also \cite{boccutorieszabstract}, \cite{boccutovrabelovariecan}, \cite{filtermultiplicative} and \cite{pavlakosrepresentation}):

\begin{definition} [Bilinear Product in Net Convergence Spaces] \label{bilinearproduct}

let $(X,\eta_{1})$, $(Y,\eta_{2})$ and $(Z,\eta_{3})$ be three vector spaces with respective net convergence structures $\eta_{i}$, $i=1,2,3$.

A map

\[\bullet: X \times Y \rightarrow Z,\]

which we denote by juxtaposition, such that it is:

\begin{enumerate} [(i)]
    \item Bilinear in each argument,
    \item And such that the partial maps induced by fixing each argument:

    \[\bullet_{y}: X \rightarrow Z\]

    \[\bullet_{x}: Y \rightarrow Z\]

    are continuous in the associated net convergence structure of the spaces in each case, including multiplication by (fixed) real constants. 
\end{enumerate}

is called a bilinear product.

If $X,Y,Z$ are Riesz spaces, we also assume that the bilinear product satisfies:

(iii) It is isotone on the positive cones of $X$ and $Y$, $X^{+}$ and $Y^{+}$ respectively.

(iv) If $\{x_{n}\}_{n \in \mathbb{N}}$ and $\{y_{n}\}_{n \in \mathbb{N}}$ are non-negative sequences in $X$ and $Y$ respectively such that their supremum exists (in $X$ and $Y$ respectively), then:

\[\sup_{n \in \mathbb{N}}(x_{n} \bullet y_{i}) = (\sup_{n \in \mathbb{N}}x_{n}) \bullet y_{i}\]

\[\sup_{n \in \mathbb{N}}(x_{i} \bullet y_{n}) = x_{i} \bullet (\sup_{n \in \mathbb{N}} y_{n}),\]

for each fixed $i \in \mathbb{N}$.

\end{definition}

For the measure-theoretic structures which the integration theory in the present work will be reduced to, and were first defined by \cite{szazdefiningnets} and \cite{szaznetintegralconvergence} in a less general setting, we use the next construction:

\begin{definition} [Nets For Integration]

Let $(\Omega, \mathcal{H})$ be a fixed paved space. A net:

\[\mathfrak{N} = \{(\sigma_{\gamma},\tau_{\gamma})\}_{\gamma \in \Gamma},\]

where $\Gamma$ is a directed set and such that:

\[\sigma_{\gamma} = \{\sigma_{\gamma_{i}}\}_{i \in I_{\gamma}},\]

\[\tau_{\gamma} = \{\tau_{\gamma_{i}}\}_{i \in I_{\gamma}},\]

are finite families in $\mathcal{H}$ and $\Omega$ respectively, is called a Net for Integration or, in another terminology, a Fleischer-Szabó-Száz net.  
    
\end{definition}

Given these two definitions, we have the main structural object of the present section, defined in analogy to Definition 1.5 of \cite{szaznetintegralconvergence}:

\begin{definition} [Integration Structure] \label{integrationstructure}

An ordered triple:

\[((\Omega, \mathcal{H}), \mathfrak{N},((X,\eta_{1}), (Y,\eta_{2}), (Z,\eta_{3}); \bullet)),\]

consisting of 

\begin{enumerate} [(i)]
    \item $(\Omega, \mathcal{H})$, a paved space;
    \item $\mathfrak{N}$, a Fleischer-Szabó-Száz net defined on (i);
    \item $(X,\eta_{1}), (Y,\eta_{2}) ,(Z,\eta_{3})$, vector spaces with (partially ordered, if Riesz spaces) net convergence structures connected by a bilinear product $\bullet$;
\end{enumerate}

will be called an integration structure. If $X,Y,Z$ are Riesz spaces, we shall call the integration structure partially ordered. 
\end{definition}

With these definitions, we may state the main object of our integration theory:

\begin{definition} [Net Riemann Integral] \label{netriemannpartiallyordered}

Let $f: \Omega \rightarrow X$ be a function and $\mu: \mathcal{H} \rightarrow Y$ be an arbitrary set function, both given with respect to an integration structure

\[((\Omega, \mathcal{H}), \mathfrak{N},((X,\eta_{1}), (Y,\eta_{2}), (Z,\eta_{3}); \bullet)).\]

Then, if for each $A \in \mathcal{H}$ we have a Fleischer-Szabó-Száz net 

\[\mathfrak{N}_{A} = \{(\sigma_{\gamma},\tau_{\gamma})\}_{\gamma \in \Gamma_{A}},\] 

and

\[S^{A}_{\gamma}(f,\mu) = \sum_{i \in I_{\gamma}}f(\tau_{\gamma_{i}})\mu(\sigma_{\gamma_{i}}),\]

an abstract Riemann sum in $Z$ defined on $\mathfrak{N}_{A}$, such that the following limit exists, 

\[\int_{A} f d\mu = \lim_{(\Gamma_{A}, \eta_{3})} S^{A}_{\gamma}(f,\mu), \]

it will be called the net Riemann integral of $f$ with respect to $\mu$ on $A$, or only the  net Riemann integral of $f$ if the set function is fixed and the limit exists for every $A \in \mathcal{H}$.

\end{definition}

When the context is clear, we shall omit the index $A$ in the Riemann sum $S^{A}_{\gamma}$ for simplicity. 

The first relevant comment about this definition is that it is a strict extension of the integral of \cite{szaznetintegralconvergence}, allowing integration of domains different than the whole domain $\Omega$, incorporating insights from \cite{fleischerinterchange}, which we also extend, as he imposed further conditions on the integration structure that are not needed in defining the integral. 

We also note that this integral is a special case of the following more general construction, inspired by \cite{mcshaneriemann}:

\begin{definition} [Abstract Net Riemann Integral] \label{abstractnetriemannpartiallyordered}

Let $U: X \times \mathcal{H} \rightarrow Z$ be a function for $Z$ a vector space equipped with a net convergence structure, and, for each $A \in \mathcal{H}$ and Fleischer-Szabó-Száz net 

\[\mathfrak{N}_{A} = \{(\sigma_{\gamma},\tau_{\gamma})\}_{\gamma \in \Gamma_{A}},\] 

put

\[S^{U,A}_{\gamma} = \sum_{i \in I_{\gamma}}U(\tau_{\gamma_{i}},\sigma_{\gamma_{i}}),\]

an abstract Riemann sum in $Z$ defined on $\mathfrak{N}_{A}$, with respect to $U$. If the following limit exists, 

\[\int_{A} U = \lim_{(\Gamma_{A}, \eta_{3})} S^{U,A}_{\gamma}, \]

it will be called the abstract net Riemann integral of $U$ on $A$, again omitting $A$ and $U$ in the Riemann sum when it causes no confusion.

\end{definition}

We now have the next immediate observation.

\begin{observation} [Net Riemann Integral from the Abstract Net Riemann Integral]
    The net Riemann integral of $f$ with respect to $\mu$ appears as a special case if we pick $U(x, A) = x\mu(A)$, for $A \in \mathcal{A}$, $x$ in the image of $f$ for some $\omega \in \Omega$, and consider the bilinear product of the associated integration structure. 
\end{observation}

We these constructions in mind, we now prove some general properties of such integrals, focusing on the net Riemann integral. As the first property is quite generic, we prove it for the abstract net Riemann integral. For the next result, we shall need some definitions of convergence structures in product spaces. For that,  see \cite{convergencestructuresvanderwalt} and the comments in Hypothesis \ref{productconvergence}. Then, we have, as an immediate consequence of Definition \ref{abstractnetriemannpartiallyordered}:

\begin{theorem} 

let $\{(Z_{l}, \eta_{l})\}_{l \in L}$ be a collection of net convergence spaces and $Z$ its product equipped with a product convergence structure $\eta$. Then, a function:

\[U: X \times \mathcal{H} \rightarrow Z,\]

is abstract net Riemann integrable on $A \in \mathcal{H}$ if and only if each $U^{l}: X \times \mathcal{H} \rightarrow Z_{l}$ is net Riemann integrable on $A$ and:

\[\int_{A}U = z \iff \int_{A}U^{l} = z_{l},\]

where $z \in Z$, $z_{l} = \pi_{l}(z) \in Z_{l}$. 
    
\end{theorem}

Similarly, from this last result:

\begin{theorem} [Homomorphic Properties of the abstract net Riemann integral]

Let  $\varphi: Z \rightarrow Z^{'}$ be a continuous homomorphism between the vector spaces $Z,Z^{'}$ such that $(Z, \eta)$ and $(Z^{'}, \eta^{'})$ are net convergence spaces. Suppose that $U: X \times \Omega \rightarrow Z$ is abstract net Riemann integrable on $A \in \mathcal{H}$ with respect to a (fixed) net convergence structure. Then, $\varphi \circ U: X \times \Omega \rightarrow Z^{'}$ is abstract net Riemann integrable on $A$ and

\[\int_{A} \varphi \circ U = \varphi (\int_{A}U)\]
\end{theorem}

\begin{proof}
    Without loss of generality, we prove this result for $\Omega$ denoting its directed set by $\Gamma$. In this case, Let $S^{U}_{\gamma}$ the corresponding Riemann sum of $U$ for $\gamma \in \Gamma$. As $\varphi$ is a (vector space) homomorphism, we have that:

    \[S_{\gamma}^{\varphi \circ U} = \sum_{i \in I_{\gamma}}\varphi(U(\tau_{\gamma_{i}},\sigma_{\gamma_{i}}))= \varphi(S_{\gamma}^{U}),\]

    for each $\gamma \in \Gamma$. As the net $\{S^{U}_{\gamma}\}_{\gamma \in \Gamma}$ has a well defined limit in $Z$ and $\varphi$ is continuous, the right side has a limit in $Z^{'}$. By the last equality, we then have:

    \[\int_{\Omega} \varphi(U) = \lim_{(\Gamma, \eta)}S^{\varphi \circ U}_{\gamma} = \lim_{(\Gamma, \eta)}\varphi(S^{U}_{\gamma}) = \varphi(\int_{\Omega}U),\]

    which is the desired result.
\end{proof}

From these two general results, and with the choice in the observation above, we get:

\begin{corollary} [Linearity of the net Riemann integral] \label{linearityofnetriemannintegral}

If $f_{1},f_{2}$ are net Riemann integrable functions on each $A \in \mathcal{H}$, and with values in a (real) vector space $X$, then $af_{1}+bf_{2}$ is a net Riemann integrable function on each $A \in \mathcal{H}$ for each real constants $a,b \in \mathbb{R}$ and:

\[\int_{A}(af_{1}+bf_{2})d\mu = a\int_{A}f_{1}d\mu+b\int_{A}f_{2}d\mu,\]

for each $A \in \mathcal{H}$.
\end{corollary}

To give further properties of the integral in Definition \ref{netriemannpartiallyordered}, we denote by $L_{NR}(\mu)$ the space of all functions $f: \Omega \rightarrow X$ that are net Riemann integrable, on each $A \in \mathcal{H}$, with respect to the (fixed) set function $\mu: \Omega \rightarrow Y$ in the context of a (fixed) integration structure. 

\begin{theorem} [Further Properties of the net Riemann integral] \label{propertiesnetriemannintegral}

The net Riemann integral has the following properties compatible with a partially ordered integration structure in Riesz spaces $X, Y, Z$:

\begin{enumerate} [(i)]
    \item Let $f \in L_{NR}(\mu)$ such that $f \geq 0$ and $\mu \geq 0$. Then, 

    \[\int_{A}f d\mu \geq 0,\]

    for each $A \in \mathcal{H}$, i.e, the integral is a positive map on the space of integrable functions.
    \item Let $\mu \geq 0 $ and $f,g \in L_{NR}(\mu)$ such that:

    \[g \geq f.\]

    Then,

    \[\int_{A}gd\mu \geq \int_{A}fd\mu,\]

     for each $A \in \mathcal{H}$, i.e, the integral is an isotone map on $L_{NR}(\mu)$ equipped with the partial ordering inherited from $X$.

    \item Let $\mu \geq 0$ and $f \in L_{NR}(\mu)$ such that $|f| \in L_{NR}(\mu)$, then:

    \[|\int_{A}fd\mu| \leq \int_{A}|f|d\mu,\]

     for each $A \in \mathcal{H}$.
\end{enumerate}
    
\end{theorem}

\begin{proof}

    (i) Let $f \in L_{NR}(\mu)$ such that it, and $\mu$, are non-negative (in $X$ and $Y$ respectively). Let the Riemann sum of $f$ in the fixed integration structure given by $S_{\gamma}(f,\mu)$. By the isotone property of the product and the non-negativity of both $f$ and $\mu$, it follows that:

    \[S_{\gamma}(f,\mu) = \sum_{i \in I_{\gamma}}f(\tau_{\gamma_{i}})\mu(\sigma_{\gamma_{i}}) \geq 0, \]

    i.e, the sum is non-negative in $Z$ for every $I_{\gamma}$ and $\gamma \in \Gamma$. Therefore, by taking the limit with respect to the partially ordered net convergence structure $\eta_{3}$, and  using that it is locally solid,  it follows that:

    \[\lim_{(\Gamma, \eta_{3})} S_{\gamma}(f,\mu) \geq 0,\]

    which is equivalent to

    \[\int_{\Omega}fd\mu \geq 0.\]

    (ii) Let $f, g \in L_{NR}(\mu)$ with corresponding Riemann sums in the fixed convergent structure $S_{\gamma}(f,\mu), S_{\gamma}(g,\mu)$. As the product is isotone, we have that, for every $\gamma \in \Gamma$ and $i \in I_{\gamma}$, 

    \[f(\tau_{\gamma_{i}})\mu(\sigma_{\gamma_{i}}) \geq  g(\tau_{\gamma_{i}})\mu(\sigma_{\gamma_{i}}).\]

    As the bilinear product is isotone, and $Z$ is a Riesz space (therefore, the sum is an isotone map in $Z$ partial ordering), it follows that:

    \[ \sum_{i \in I_{\gamma}}g(\tau_{\gamma_{i}})\mu(\sigma_{\gamma_{i}}) \geq  \sum_{i \in I_{\gamma}}f(\tau_{\gamma_{i}})\mu(\sigma_{\gamma_{i}}),\]

    for every $\gamma \in \Gamma$ and $i \in I_{\gamma}$. 

    Therefore, by taking the limit with respect to the partially ordered net convergence structure $\eta_{3}$, and using that it is locally solid, it follows that:

    \[\lim_{(\Gamma, \eta_{3})} S_{\gamma}(g,\mu) \geq \lim_{(\Gamma, \eta_{3})} S_{\gamma}(f,\mu),\]

    which is the same as

    \[\int_{\Omega}gd\mu \geq \int_{\Omega}fd\mu\]

    (iii) Follows from the same previous techniques and the use of axioms $(i)$ and $(ii)$ of the partially ordered net convergence structure of $Z$.

\end{proof}

These results give the basic properties of the abstract integral defined in this section that are valid in full generality. We now aim to prove some specific results concerning hereditary (i.e subset) integration, that is: if a function $f$ is Net Riemann integrable on $B \in \mathcal{H}$ and $A \in \mathcal{H}$ is such that $A \subseteq B$, does it implies that $f$ is integrable on $A$ ? As the example of the Riemann-Stieltjes integral of \cite{mcshaneriemann} shows, this cannot be positively answered in general. To obtain such result, we impose further conditions on the Szábo-Száz net that gives the integration procedure (which are satisfied for general integration procedures such as \cite{mcshaneriemann} and \cite{sionsemigroup}). 

\begin{theorem} [Subset Integration] \label{subsetintegration}

Suppose that $Z$ is complete in its net convergence structure. Fix $A,B \in \mathcal{H}$ such that $A \subseteq B$ and suppose that either $\mathcal{H}$ is closed under taking complements of its members or that $A^{c} \in \mathcal{H}$ . Consider the following hypotheses:

\begin{enumerate} [(i)]
    \item For each $\gamma_{B}, \gamma_{B}^{'} \in \Gamma(B)$, there exists $\gamma_{A}, \gamma^{'}_{A} \in \Gamma(A)$ such that, for each $\gamma^{1}_{A}, \gamma^{2}_{A} \in \Gamma(A)$ that satisfies $\gamma^{1}_{A} \geq \gamma_{A}, \gamma^{2}_{A} \geq \gamma^{'}_{A}$, we can find $\gamma_{A^{c}} \in \Gamma(A^{c})$ in such a way that there exists $\gamma^{1}_{B}, \gamma^{2}_{B} \in \Gamma(B)$, such that $\gamma^{1}_{B} \geq \gamma_{B}, \gamma^{2}_{B} \geq \gamma_{B}$,  with corresponding Fleischer-Szabó-Száz nets

  \[\gamma^{1}_{B} \rightarrow ((\sigma_{\gamma^{1}_{A}} \cup \sigma_{\gamma_{A^{c}}}), (\tau_{\gamma^{1}_{A}} \cup \tau_{\gamma_{A^{c}}})),\]

  \[\gamma^{2}_{B} \rightarrow  ((\sigma_{\gamma^{2}_{A}} \cup \sigma_{\gamma_{A^{c}}}), (\tau_{\gamma^{2}_{A}} \cup \tau_{\gamma_{A^{c}}})),\]

  \item For each $\gamma_{B} \in \Gamma(B)$ there exists $\gamma_{A} \in \Gamma(A)$ and $\gamma_{A^{c}} \in \Gamma(A^{c})$ such that for each $\gamma^{'} \in \Gamma(A)$ and $\gamma^{''} \in \Gamma(A^{c})$, there exists $\gamma \in \Gamma(B)$ of the form

  \[\gamma_{B} \rightarrow ((\sigma_{\gamma^{'}} \cup \sigma_{\gamma_{''}}), (\tau_{\gamma^{'}} \cup \tau_{\gamma_{''}})),\]

  and $\gamma \geq \gamma_{B}$.

\end{enumerate}

Then, if $f$ is net Riemann integrable on $B$, it is net Riemann Integrable on $A$ and $A^{c}$ (the complement relative to $B$) and:

    \[\int_{B}fd\mu = \int_{A}fd\mu + \int_{A^{c}}fd\mu\]

\end{theorem}

\begin{proof}
    As $f$ is net Riemann integrable on $B$, the net of Riemann sums differences formed on $\Gamma(B) \times \Gamma(B)$, with values in $Z$, 

    \[\{\sum_{i \in I_{\gamma^{'}}}f(\tau_{\gamma^{'}_{i}})\mu(\sigma_{\gamma^{'}_{i}}) - \sum_{i \in I_{\gamma}}f(\tau_{\gamma_{i}})\mu(\sigma_{\gamma_{i}})\}_{(\gamma, \gamma^{'}) \in \Gamma(B) \times \Gamma(B) },\]

    is Cauchy in the net convergence structure of $Z$. Fix $\gamma_{B}, \gamma^{'}_{B} \in \Gamma(B)$. Now, by the first hypothesis (i), there exists $\gamma_{A}, \gamma^{'}_{A} \in \Gamma(A)$ such that, for each $\gamma^{1}_{A}, \gamma^{2}_{A} \in \Gamma(A)$ that satisfies $\gamma^{1}_{A} \geq \gamma_{A}, \gamma^{2}_{A} \geq \gamma^{'}_{A}$, we can find $\gamma_{A^{c}} \in \Gamma(A^{c})$ in such a way that there exists $\gamma^{1}_{B}, \gamma^{2}_{B} \in \Gamma(B)$, such that $\gamma^{1}_{B} \geq \gamma_{B}, \gamma^{2}_{B} \geq \gamma_{B}$,  with corresponding Fleischer-Szabó-Száz nets

  \[\gamma^{1}_{B} \rightarrow ((\sigma_{\gamma^{1}_{A}} \cup \sigma_{\gamma_{A^{c}}}), (\tau_{\gamma^{1}_{A}} \cup \tau_{\gamma_{A^{c}}})),\]

  \[\gamma^{2}_{B} \rightarrow  ((\sigma_{\gamma^{2}_{A}} \cup \sigma_{\gamma_{A^{c}}}), (\tau_{\gamma^{2}_{A}} \cup \tau_{\gamma_{A^{c}}})),\]

  Which gives the following decomposition of Riemann sums:

  \[\sum_{i \in I_{\gamma^{1}_{B}}}f(\tau_{\gamma^{1,B}_{i}})\mu(\sigma_{\gamma^{1,B}_{i}}) = \sum_{j \in I_{\gamma^{1}_{A}}}f(\tau_{\gamma^{1,A}_{j}})\mu(\sigma_{\gamma^{1,A}_{j}}) + \sum_{k \in I_{\gamma_{A^{c}}}}f(\tau_{\gamma^{A^{c}}_{k}})\mu(\sigma_{\gamma^{A^{c}}_{k}}),\]

  \[\sum_{i \in I_{\gamma^{2}_{B}}}f(\tau_{\gamma^{2,B}_{i}})\mu(\sigma_{\gamma^{2,B}_{i}}) = \sum_{j \in I_{\gamma^{2}_{A}}}f(\tau_{\gamma^{2,A}_{j}})\mu(\sigma_{\gamma^{2,A}_{j}}) + \sum_{k \in I_{\gamma_{A^{c}}}}f(\tau_{\gamma^{A^{c}}_{k}})\mu(\sigma_{\gamma^{A^{c}}_{k}}).\]

  Therefore, $\{\sum_{i \in I_{\gamma^{1}_{A}}}f(\tau_{\gamma^{1,A}_{i}})\mu(\sigma_{\gamma^{1,A}_{i}}) - \sum_{i \in I_{\gamma^{2}_{A}}}f(\tau_{\gamma^{2,A}_{i}})\mu(\sigma_{\gamma^{2,A}_{i}})\}_{(\gamma^{1}_{A}, \gamma^{2}_{A}) \geq (\gamma_{A}, \gamma^{'}_{A}) \in \Gamma(A) \times \Gamma(A)}$ is contained in $\{\sum_{i \in I_{\gamma_{B}^{1}}}f(\tau_{\gamma^{1,B}_{i}})\mu(\sigma_{\gamma^{1,B}_{i}}) - \sum_{i \in I_{\gamma^{2}_{B}}}f(\tau_{\gamma^{2,B}_{i}})\mu(\sigma_{\gamma^{2,B}_{i}})\}_{(\gamma^{*}_{B}, \gamma^{**}_{B}) \geq (\gamma_{B}, \gamma^{'}_{B}) \in \Gamma(B) \times \Gamma(B) }$, which shows that the net of differences of Riemann sums in $\Gamma(A) \times \Gamma(A)$ is a quasi-subnet of net of differences of Riemann sums in $\Gamma(B) \times \Gamma(B)$ . By the first paragraph (completeness of the convergence structure of $Z$) and axiom (ii) of Definition \ref{netconvergencestructure}, we conclude that $f$ is integrable on $A$. A similar argument gives that $f$ is integrable on $A^{c}$. 

  Similarly, by (ii), it follows $\{\sum_{j \in I_{\gamma_{A}}}f(\tau_{\gamma^{A}_{j}})\mu(\sigma_{\gamma^{A}_{j}}) + \sum_{k \in I_{\gamma_{A^{c}}}}f(\tau_{\gamma^{A^{c}}_{k}})\mu(\sigma_{\gamma^{A^{c}}_{k}})\}_{(\gamma_{A},\gamma_{A^{c}}) \in \Gamma(A) \times \Gamma(A^{c})}$ is a quasi-subnet of $\{\sum_{j \in I_{\gamma_{B}}}f(\tau_{\gamma^{B}_{j}})\mu(\sigma_{\gamma^{B}_{j}})\}_{\gamma_{B} \in \Gamma(B)}$, in such a way that by axiom (ii) of Definition \ref{netconvergencestructure}, they convergence to the same limit, $\int_{B}fd\mu$. Now, by the double indexing technique used in \cite[Remark 2.5, p. 7]{convergencestructuresvanderwalt}, we may regard the Riemann sums of $f$ in $A$ and $A^{c}$ as convergent nets indexed in $\Gamma(A) \times \Gamma(A^{c})$, converging to $\int_{A}fd\mu$ and $\int_{A^{c}}fd\mu$ respectively. Combining this results, we have:

  \[\int_{B}fd\mu = \lim_{(\Gamma(A) \times \Gamma(A^{c}), \eta_{3})}(\sum_{j \in I_{\gamma_{A}}}f(\tau_{\gamma^{A}_{j}})\mu(\sigma_{\gamma^{A}_{j}}) + \sum_{k \in I_{\gamma_{A^{c}}}}f(\tau_{\gamma^{A^{c}}_{k}})\mu(\sigma_{\gamma^{A^{c}}_{k}}))\]
  
  \[= \lim_{(\Gamma(A) \times \Gamma(A^{c}), \eta_{3})}\sum_{j \in I_{\gamma_{A}}}f(\tau_{\gamma^{A}_{j}})\mu(\sigma_{\gamma^{A}_{j}}) + \lim_{(\Gamma(A) \times \Gamma(A^{c}), \eta_{3})}\sum_{k \in I_{\gamma_{A^{c}}}}f(\tau_{\gamma^{A^{c}}_{k}})\mu(\sigma_{\gamma^{A^{c}}_{k}})= \int_{A}fd\mu + \int_{A^{c}}fd\mu\]

\end{proof}

Therefore, from this last result, we get the following additivity property of the net Riemann integral:

\begin{corollary} [Additivity of the Integral] \label{additivitynetriemannintegral}

Let $f$ be an element of $L_{NR}(\mu)$. Then, under the hypotheses of Theorem \ref{subsetintegration}, the set function:

\[A \in \mathcal{H} \rightarrow \eta_{f}(A) =\int_{A}fd\mu,\]

with values in $Z$, is an (finitely) additive set function defined on $\mathcal{H}$.
    
\end{corollary}

Specific applications of these results in the topological case can be found in the integrals of \cite[Theorem 4.2, p. 18]{sionsemigroup} and \cite{millingtonruns}, and in the partially ordered case (in the Henstock-Kurzweil case) of \cite[Proposition 5.2.10, p. 79]{integralmeasureandordering}. One other possibility to obtain the same result is to ask, as did \cite{fleischerinterchange}, that the directed sets $\Gamma(A \cup B)$ and $\Gamma(A)$ (and $\Gamma(B)$), for $A,B$ disjoint, are isomorphic, in the sense of always being possible to obtain a partition from the union of disjoint sets to each one individually, and vice versa. The previously cited integrals satisfy this property, and it implies the conditions in Theorem \ref{subsetintegration}. Therefore, our result is more general.

Corollary \ref{additivitynetriemannintegral} notwithstanding, the behavior of the set function induced by a net Riemann integrable functions is not necessarily simple. An interesting property would be that the function $f = \mathbbm{1}_{A}$, for $A \in \mathcal{H}$, is integrable with respect to $\mu$ and:

\[\nu_{f}(A) = \mu(A).\]

In general, that is not true for integrals such as the Henstock-Kurzweil integral for vector-valued (specially operator-valued) measures (see the discussion in Chapter 5 of \cite{integralmeasureandordering}), except under regularity conditions of the Radon type. As our integral contains this one as a special case (which well be shown shortly), the net Riemann integral will not have this property. Having this problem in mind, \cite{szaznetintegralconvergence} isolated the integration of indicator functions as a formal property and defined the next concept:

\begin{definition} [Regular Integrators]

We say that a set function $\mu: \mathcal{H} \rightarrow Y$ in the context of the net Riemann integral in Definition \ref{netriemannpartiallyordered} is a regular integrator if:

\[\int_{A} \mathbbm{1}_{\Omega}d\mu = \mu(A),\]

and,

\[\int_{\Omega} \mathbbm{1}_{A}d\mu = \mu(A)\]

for each $A \in \mathcal{H}$.
\end{definition}

It follows directly from the linearity properties of the net Riemann integral in Corollary \ref{additivitynetriemannintegral} that simple functions (to be defined next) are then net Riemann Integrable and the usual additivity property holds, that is:

\begin{proposition} [Integration of Simple Functions] \label{integrationofsimplefunctions}

Let $f: \Omega \rightarrow X$ be a $\mathcal{H}$-simple function, i.e:

\[ f = \sum_{i=1}^{n}x_{i}\mathbbm{1}_{A_{i}},\]

with $\{x_{i}\}_{i=1}^{n}$ a finite sequence of elements of $X$ and $\{A_{i}\}_{i=1}^{n}$ a finite (not necessarily disjoint) sequence of elements of $\mathcal{H}$. Then, if $\mu: \mathcal{H} \rightarrow Y$ is a regular integrator, $f$ is net Riemann Integrable and:

\[\int_{\Omega}fd\mu = \sum_{i=1}^{n}x_{i}\mu(A_{i})\]
\end{proposition}

The next strengthening of these two last properties, hereditary/additivity of integration and regular integrators, would be to prove that, if $\mu$ is $\sigma$-additive in some appropriate sense, then the set function induced by a net Riemann integrable function $f$, given by $\nu_{f}$, is $\sigma$-additive in $\mathcal{H}$. As the case of the Henstock-Kurzweil integral of a Banach-valued function with respect to a scalar $\sigma$-additive measure shows (see \cite[Proposition 5.4.1, p. 159]{schwabikgouju} or \cite{serieshenstock}), that is not true in general. This is related to the fact that we may construct integrals by the procedure of Definition \ref{netriemannpartiallyordered} that includes refinements and gauge integrals using finite partitions, and therefore based only on strictly finite Riemann sums, and that are also of non-absolute type. That is in contrast with integrals such as \cite{rickartintegralarticle}, \cite{kolmogorovuntersuchungen} and \cite{sionsemigroup} integrals, which are based on the refinement and truncation order in countable partitions/coverings of the base space, having the $\sigma$-additivity of the integral-induced measure naturally (a consequence of the unconditional summability of the approximating sums involved in the limit process, inducing an absolute integral). As we will see in the section about a generalization of \cite{kolmogorovuntersuchungen} $S^{*}$ integral to net convergence structures, obtaining a property of this type ($\sigma$-additive measure implies $\sigma$-additive integral-induced measure) is possible even without imposing a topology or order in the image space $Z$.

We now pass to a brief discussion of the properties of the space $L_{NR}(\mu)$. An immediate consequence of  Corollary \ref{additivitynetriemannintegral} and Proposition \ref{integrationofsimplefunctions} is the vector space structure of this space:

\begin{proposition} [$L_{NR}(\mu)$ is a Vector Space]

Let $L_{NR}(\mu)$ be the space of net Riemann integrable functions with respect to $\mu$ in a fixed integration structure. Then, $(L_{NR}(\mu),+)$ is a semigroup additively written with respect to the $+$ operation of $X$, and $(L_{NR}(\mu), +, (\mathbb{R}, \cdot))$ is a vector space with pointwise product by real numbers \footnote{As $X$ is a vector space and the bilinear product is compatible with pointwise multiplication in $X$ by real numbers.}. If $\mu$ is a regular integrator, $L_{NR}(\mu)$ contains the space of $X$-valued simple functions as a vector subspace. 

In the case of $X,Y,Z$ have a fixed partially ordered net convergence structure, $L_{RM}(\mu)$ is a partially ordered vector space with the pointwise order for functions: that is, $f,g \in L_{RM}(\mu)$ are such that $f \leq g$ if and only if $f(\omega) \leq g(\omega)$ (in $X$) for each $\omega \in \Omega$.

\end{proposition}

At this level of generality, not much more interesting can be said about this space. For example, we were unable to prove continuity properties of the integration functional for $L_{NR}(\mu)$ without imposing some more conditions, not necessarily of topological or order-convergence properties, on the convergence structures of $X, Y$ and $Z$ (see the next sections). The action of (real) linear functions on $L_{NR}(\mu)$, related to a type of Pettis integral, can be elucidated by some analogue constructions as in topological/partially ordered vector spaces, but we will comment about this aspect in the last part of the present section.  

In the further case of $X,Y,Z$ on a fixed partially ordered net convergence structure, two interesting questions should be asked about $L_{NR}(\mu)$:

\begin{enumerate}
    \item Is it, or can it be made, a Lattice ?
    \item In general, if $f: \Omega \rightarrow X$ is a function such that $|f| \leq |g|$ in the pointwise partial order (as in the definition above), and $g \in L_{NR}(\mu)$, does it implies that $f \in L_{RM}(\mu)$ ? In other words, is $L_{NR}(\mu)$ solid in the space of all functions of the type $f: \Omega \rightarrow X$ ?
\end{enumerate}

As is expected, the Henstock-Kurzweil integral provides a counterexample to these claims for the net Riemann integral. As this integral is non-absolute, if $X = Z$ a Banach lattice and $Y = \mathbb{R}_{+}$ and $\mu$ the Lebesgue measure, its space of integrable functions, equipped with the Alexiewicz norm, is not solid in the sense above (see, for example, \cite{federsonpeculiar}). Also, by the same reason, it is not a lattice. Some results in the lattice direction seem possible by using a modification of the lattice-integration results of \cite{hayesstieltjes} to the present case, but we leave the details for further work.

We now discuss the Pettis integral aspect of the net Riemann integral, the last property in the general context of net convergence structures. For this, we mainly use the material in Chapters 3-5 of \cite{butzmann}, but adapted to nets. Related material (in the filter language) can be found in \cite{lucyshynintegralconvergencespaces} and \cite{lucysshynthesis}.

We begin with the following definition:

\begin{definition} [Dual of a Net Convergence Space]
    Let $(X, \eta)$ be a topological vector space with a net convergence structure $\eta$. Then, the space of all linear continuous real valued functionals $\psi: X \rightarrow \mathbb{R}$, where $\mathbb{R}$ is equipped with its canonical (euclidean) topology, is called the continuous dual of $X$, which we denote by $\mathcal{L}(X)$.
\end{definition}

We assume, tacitly, that $\mathcal{L}(X)$ is non-empty in each statement involving it. Its now easy to see that:

\begin{proposition} [Continuous Convergence]
    Let $\mathcal{L}(X)$ be the continuous dual of the net convergence space $(X, \eta)$, where $X$ is a topological vector space. Consider the next definition of convergence, which we denote by $\eta_{\mathcal{L}(X)}$: a net $\{\psi_{i}\}_{i \in I}$ of elements of $\mathcal{L}(X)$ converges to $\psi \in \mathcal{L}(X)$ if and only if for every $x \in X$ fixed, $\{\psi_{i}(x)\}_{i \in I}$ converges to $\psi(x)$. Then, $(\mathcal{L}(X), \eta_{\mathcal{L}(X)})$ is a net convergence space.  
\end{proposition}

For the rest of this section, we assume that the continuous dual of every net convergence space is equipped with the convergence of the last proposition. 

For the definition of the Pettis integral in this context, we shall need the next condition (see \cite[Lemma 4.2.5, p. 127]{butzmann}):

\begin{definition} [Separated Continuous Dual]

Let $\mathcal{L}(X)$ be the continuous dual of the net convergence space $(X, \eta)$, where $X$ is a topological vector space. We say that $\mathcal{L}(X)$ separates points of $X$ if:

\[\bigcap\{ker(\psi): \psi \in \mathcal{L}(X)\} = \{0\}\]
\end{definition}

By \cite[Lemma 4.2.5, p. 127]{butzmann}, this is equivalent to ask that the canonical embedding of $X$ into its (continuous convergence) double dual is injective. For now on in this section, we shall assume that each $\mathcal{L}(X)$ is separated. 

The last convergence-theoretic notion is the following:

\begin{definition} [Weak Convergence in $X$]

Let $\mathcal{L}(X)$ be the continuous dual of the net convergence space $(X, \eta)$, where $X$ is a topological vector space. Then, we say that $X$ is equipped with the weak net convergence structure $\eta_{w}$ when a net $\{x_{i}\}_{i \in I}$ converges to $x \in X$ if and only if for each $\psi \in \mathcal{L}(X)$, the (real valued) net $\{\psi(x_{i})\}_{i \in I}$ converges to $\psi(x)$ in the euclidean topology of $\mathbb{R}$. We denote this net convergence space by $(X, \eta_{w})$.
\end{definition}

Suppose now that $(\Omega, \mathcal{H},\mu)$ is a fixed complete finite measure space, and that $(X, \eta)$ is a net convergence space, where $X$ is a topological vector space. As said previously, we shall also suppose that $\mathcal{L}(X)$ is non-empty\footnote{This is not automatic, as the case of non-locally convex topological vector spaces show. See \cite{adasch}.} and separated.

We now have the next essential definition:

\begin{definition} [Weakly Measurable Function]

We say that a function $f: \Omega \rightarrow X$ is weakly measurable if for every $\psi \in \mathcal{L}(X)$, the function $\psi(f): \Omega \rightarrow \mathbb{R}$, defined by $\psi(f)(\omega) = \psi(f(\omega))$ for each $\omega \in \Omega$, is a Borel measurable function. 
    
\end{definition}

We also have, associated to this last definition:

\begin{definition} [Scalarly Integrable Function]
    We say that a weakly measurable function $f: \Omega \rightarrow X$ is scalarly integrable if:

    \[\int_{\Omega}^{(L)}\psi(f)d\mu < \infty,\]

    for each $\psi \in \mathcal{L}(X)$, where this last integral is taken in the Lebesgue sense. 
\end{definition}

We now state the definition of Pettis mode of integration (for the locally convex topological vector space case, see \cite{pettischakraborty} or \cite{musialhandbookpettis}):

\begin{definition} [Pettis Integral]

Let $f: \Omega \rightarrow X$ be a scalarly integrable function. Then, we say that $f$ if Pettis integrable on  $A \in \mathcal{H}$ with Pettis integral $x_{A}$ if there exists an element $x_{A}$ in $X$ such that,

\[\psi(x_{A}) = \int_{A}^{(L)}\psi(f)d\mu,\]

for every $\psi \in \mathcal{L}(X)$ and the integral in the right-hand side taken in Lebesgue sense.

If f is Pettis integrable on every $A \in \mathcal{H}$, then we say that $f$ is Pettis integrable. 
\end{definition}

In general, we denote the Pettis integral of a map $f$ by:

\[\int^{(P)}_{\cdot} fd\mu,\]

with the integrating set omitted when no confusion is possible. 

We point out now that this notion of "weak integration" contains not only the classical concept of \cite{pettisoriginalarticle}, but also its partially ordered analogue in \cite{siposintegrationpariallyordered}, and the Radon type construction of \cite{lucyshynintegralconvergencespaces}.

To guarantee that the Pettis integral is well defined, we need the next uniqueness result for this integral. As it is straightforward,  $\mathcal{L}(X)$ being supposed to be separated,  we skip it. 

\begin{proposition} [Uniqueness of the Pettis Integral]
    Let $f: S \rightarrow X$ be a Pettis integrable function. Then, its Pettis integral is unique.
\end{proposition}

We now prove the main result we alluded before, we the cumulative notation fixed until now:

\begin{theorem} [Net Riemann Integral and Pettis Integration] \label{pettisriemann}

Let $f: S \rightarrow X$ be a scalarly integrable function, where $X$ is equipped with its weak convergence structure $\eta_{w}$. Suppose also that if $g: \Omega \rightarrow \mathbb{R}$ is a real valued function (Lebesgue) integrable with respect to $\mu$, then its net Riemann integrable exists coincides with its Lebesgue integral, and conversely\footnote{That includes most of the integrals that we shall study. See the next section. Note the hypothesis of weak measurability, which is essential here for this hypothesis too.}. Then, we have:

\begin{enumerate} [(i)]
    \item $f$ is Pettis integrable if and only if $\int_{A} fd\mu \in X$ for each $A \in \mathcal{H}$, where this last integral is the net Riemann integral of $f$;
    \item f is Pettis integrable on every $A \in \mathcal{H}$ if and only if $f$ is net Riemann integrable;
    \item We have the identification:

    \[\int_{A}^{(P)}fd\mu = \int_{A} fd\mu,\]

    for each $A \in \mathcal{H}$.
\end{enumerate}
    
\end{theorem}

\begin{proof}
    The hearth of the result is part $(ii)$ and $(iii)$, which we now prove. 

    For that, first assume that, given an arbitrary (fixed) set $A \in \mathcal{H}$, $f$ is Pettis integrable on $A$ with (Pettis) integral $x_{A} \in X$. That means that we have, for each $\psi \in \mathcal{L}(X)$,

    \[\psi(x_{A}) = \int^{(P)}_{A} \psi(f) d\mu,\]

    where the last integral is a Lebesgue integral of the (scalarly measurable) function $\psi(f)$ with respect to $\mu$. Therefore, by hypothesis, $\psi(f)$ is net Riemann integrable with respect to $\mu$ and:

    \[\int_{A}^{(L)}\psi(f)d\mu = \int_{A}\psi(f) d\mu,\]

    where the integral in the left side of the equality is a Lebesgue integral, and:

    \[\int_{A}\psi(f)d\mu = \lim_{(\Gamma(A), \mathbb{R})}S_{\gamma}(\psi(f),\mu),\]

    where this last limit is taken in the euclidean topology of $\mathbb{R}$, which we indicate by the $\mathbb{R}$ is the subscript. 

    Now, we prove that the net Riemann integral of $f$ on $A$, in the weak net convergence structure $\eta_{w}$ of $X$, exists and equals $x_{A}$. In this case, we have to prove that, for each $\psi \in \mathcal{L}(X)$, there exists $y_{A} \in X$ such that the following holds:

    \[y_{A} = \lim_{(\Gamma(A), \eta_{w})}S_{\gamma}(f,\mu),\]

    that is, 

    \[\psi(y_{A}) = \lim_{(\Gamma(A), \mathbb{R})}\psi(S_{\gamma}(f,\mu)),\]

    for each $\psi \in \mathcal{L}(X)$. But, from the development above, we have that, as $\psi \in \mathcal{L}(X)$ is linear and the sum occurring in the definition of net Riemann integral is finite\footnote{Compare this with the integral is Section \ref{sectionsummabilitykolmogorov}.},

    \[\psi(S_{\gamma}(f,\mu)) = S_{\gamma}(\psi(f),\mu), \]

    for each $\gamma \in \Gamma(A)$, and the (net) limit of the last sum, in the appropriate convergence structure as above, is exactly $\psi(x_{A})$. Therefore, if we put $y_{A} = x_{A}$, then the net Riemann integral of $f$ with respect to $\mu$ (in the weak net convergence structure of $X$) exists and equals $s_{E}$, where the limit is taken in the weak topology of $X$. Therefore, as $\psi$ and $A$ above were arbitrary, the net Riemann integral in the weak net convergence structure of $X$ exists and is equal to $x_{A}$, for each $A \in \mathcal{H}$. This finishes the first part.

    For the converse, assume that, given an arbitrary (fixed) set $A \in \mathcal{H}$, $f$ is net Riemann integrable on $A$ with integral $y_{A} \in X$, where the net limit is taken in the weak net convergence structure of $X$. 

    Fix $\psi \in \mathcal{L}(X)$ a (continuous) linear functional on $X$. By its continuity, we have that the following limit exists, as consequence of the existence of the net Riemann integral:

    \[\psi(y_{A}) =  \lim_{\Gamma(A), \mathbb{R})}\psi(S_{\gamma}(f,\mu)), \]

    where we identify the right hand side of the expression to be, by definition, 

    \[\int_{A}\psi(f)d\mu,\]

    which is also, by hypothesis, the Lebesgue integral of $L$ with respect to $\tau$. Therefore, as $\psi$ was arbitrary, if we put $x_{A} = y_{A}$, we have that, for each $\psi \in \mathcal{L}(X)$,

    \[\psi(y_{A}) = \int_{A}\psi(f)d\mu = \int_{A}^{(L)}\psi(f)d\mu,\]

    which, by uniqueness, proves that $f$ is Pettis integrable on $A$ with Pettis integral equal to the net Riemann integral $y_{A}$. As $A \in \mathcal{H}$ was arbitrary, we conclude that $f$ is Pettis integrable and its Pettis integral equals the net Riemann integral. This finishes the proof, as the other points are immediate.  
\end{proof}

This shows not only that a general Pettis integration theory can be developed in the net convergence spaces, but that we also may get a Riemmanian representation of such integral. Similar constructions can be made for the Dunford integral in the double dual of net convergence spaces (for the Dunford integral, see Chapter 2 of \cite{sionsemigroup}), but we are not going to develop it. Some elements of this discussion with respect to the Dunford integral can also be found in \cite{lucysshynthesis}.

Having now constructed the basic theory of the net Riemann integral, we pass to describe some specific examples of it, and argue that it indeed represents a very interesting generalization of a wide range of Riemann type integrals in the literature (later, we will see the same applied to Lebesgue ones). 

\subsection{Examples and Further properties of the Net Riemann Integral.} \label{examples}

We now present some special cases of the general net Riemann integral of Definition \ref{netriemannpartiallyordered}. The first three of them can be found, in the topological case of Banach spaces in \cite{szazdefiningnets}. We point out that all of the examples can be defined in general net convergence monoids, with a partial order or not, but we leave such considerations for a further work, and focus in the vector space case.  

Notice also that to give examples of the general net Riemann Integration it is only necessary to construct examples of approximating sums in the (directed) sets $\Gamma_{A}$, which should be defined in each procedure for each $A \in \mathcal{H}$. For simplicity, we will generally omit $A$ and work directly with $\Omega$, as the statement for each $A \in \mathcal{H}$ being the same. We thus begin with a simple (abstract) example:

\begin{example} [Kolmogorov S-Integral] \label{KolmogorovSintegralconvergencestructures}

Let $\mathcal{P}^{f}(\Omega)$ be the set of all finite coverings of a set $\Omega$ in the context of an integration structure as in Definition \ref{netriemannpartiallyordered}. Suppose also that the paved space in the definition, $(\Omega, \mathcal{H})$, is a ring paved space such that $\Omega$ is a disjoint union of elements of $\mathcal{H}$ \footnote{This is to guarantee that $\Gamma$ in this example is non-empty and directed. We can weaken these and get the same results. This comment is also valid for the next examples.}. Then, if we define the refinement partial order  $\leq$ on $\mathcal{P}^{f}(\Omega) \times \mathcal{F}^{\mathcal{P}}(\Omega)$ as:

\[(P, \{t_{i}\}_{i=1}^{m}) \leq (Q, \{s_{i}\}_{j=1}^{n}) \iff Q \ \text{refines} \ P,\]

for $P,Q \in \mathcal{P}^{f}(\Omega)$,  $m,n \in \mathbb{N}$, and $t_{i} \in P$ , $s_{j} \in Q$, $(\mathcal{P}^{f}(\Omega) \times \mathcal{F}^{\mathcal{P}}(\Omega), \leq)$ becomes a (downwards) directed set (see \cite{szazdefiningnets}). 

Then, putting:

\[\Gamma = (\mathcal{P}^{f}(\Omega) \times \mathcal{F}(\Omega), \leq),\]

and considering the net formed by the identity, we get that the resulting Net Riemann integral is a refinement net limit (with respect to the net convergence structure of $Z$) of finite sums based on the whole covering $P = \sigma \in \Gamma$. This is a generalization (as we consider coverings, not partitions) of the S-integral of Kolmogorov \cite{kolmogorovuntersuchungen}.
\end{example}

A non-finite partition version of the following example is given next:

\begin{example} [Sion Integral] \label{sionintegralpartiallyorderedconvergence}

Considering different types of refinement type integrals, we may also modify  \cite{sionsemigroup} integral to our setting. Fix now the notation as in  Definition \ref{netriemannpartiallyordered}.

In this setting, consider $\mathcal{D}(\Omega)$ the set formed by all pairs $(P, \Delta)$ such that:

\begin{enumerate} [(i)]
    \item $P \in \mathcal{P}(\Omega)$ is, as above, a disjoint covering of $\Omega$.
    \item $\Delta: \mathcal{P}(\Omega) \rightarrow \mathcal{F}(\mathcal{P}(\Omega))$, defined by $\Delta(P) = \mathcal{I} \in \mathcal{F}(P)$, is given, and called a truncation.
\end{enumerate}

Suppose also that the paved space in the definition, $(\Omega, \mathcal{H})$, is a ring-paved space and that $\Omega$ is a disjoint union of elements of $\mathcal{H}$. Now, define a partial order in $\mathcal{D}(\Omega)$, denoted by $>>$, as:

\[(Q, \Delta_{2}) >> (P, \Delta_{1}) \iff Q \ \text{refines} \ P \ \text{and} \ \Delta_{1}(M) \subseteq \Delta_{2}(M), \forall M \in \mathcal{P}(\Omega) \ \text{such that} \ M \ \text{refines} \ Q\]

Then, putting:

\[\Gamma = \{(P, \tau, \Delta): (P, \Delta) \in \mathcal{D}(\Omega) \ \text{and} \  \tau = \{\tau_{k}\}_{k \in K}\ \text{countable such that} \ \tau_{k} \in \sigma \subseteq P\},\]

and considering it equipped with the partial ordering above independent of the chosen points as in Example \ref{KolmogorovSintegralconvergencestructures} (i.e, the net limit has to be the same for all points or choice functions in the terminology of \cite{szazdefiningnets}), it becomes a directed set and the net:

\[\mathfrak{N}(P, \tau, \Delta) = (\{\sigma_{i}\}_{i \in \Delta(P)},\{\tau_{i}\}_{i \in \Delta(P)}),\]

is a Fleischer-Szabó-Száz net defined on it - for more details, such as the proof that $\Gamma$ is a directed set with the partial orderings above, see \cite{szazdefiningnets}.

In this case, the Net Riemann integral obtained is the integral from \cite{sionsemigroup} which we call the Sion Integral, which in turn constitutes a generalization of the classical Rickart and Phillips integrals of \cite{rickartintegralarticle} and  \cite{phillipsintegratonarticle}.

\end{example}

We now give an abstract version of the Henstock-Kurzweil (see \cite{henstockmagnum}) integral as in \cite{szazfundamentaltheorem}.

\begin{example} [Száz Gauge Integral]

We now take $\Omega$ to be a topological space and $\mathcal{R}$ the set of all relations $R \subseteq \Omega \times \Omega$ such that $t \in int(R(t))$ for all $t \in \Omega$ and let $\mathcal{H}$ be a family of sets in $\Omega$ such that:

\begin{enumerate} [(i)]
    \item For each $A, B \in \mathcal{H}$, there exists a disjoint finite sequence $\{\sigma_{i}\}_{i=1}^{n}$ in $\mathcal{H}$ such that $A \setminus B = \cup_{i=1}^{n}\sigma_{i}$.
    \item For each $A \in \mathcal{H}$ and $R \in \mathcal{R}$, there exists finite sequences $\sigma = \{\sigma_{i}\}_{i=1}^{n}$ and $\tau= \{\tau_{i}\}_{i=1}^{m}$ in $\mathcal{H}$ and $\mathcal{R}$ respectively such that $(\sigma,\tau,R) \in \Gamma$ and $\cup_{i=1}^{n} \sigma_{i} = A$.
\end{enumerate}

Where $\Gamma$ is, in the examples above, the domain of the Fleischer-Szabó-Száz net and is given by the collection of all triples $(\sigma,\tau,R)$ such that $\sigma = \{\sigma_{i}\}_{i \in I}$ and $\tau = \{\tau_{i}\}_{i \in I}$ are finite families in $\mathcal{H}$ and $\Omega$ respectively, and $R \in \mathcal{R}$ is a relation, such that $\sigma$ is a disjoint collection and:

\[\tau_{i} \in cl(\sigma_{i}) \ \text{and} \ \sigma_{i} \subseteq R(\tau_{i}),\]

for all $i \in I$.

We then define a (downwards) directed partial order in $\Gamma$ by: for any $(\sigma,\tau,R)$ and $(\rho, \nu, U)$ in $\Gamma$ we say that

\[(\sigma,\tau,R) \leq (\rho, \nu, U),\]

if and only if $\sigma$ is a refinement of $\rho$ and $U \subseteq S$.

By the properties $(i)$ and $(ii)$ of the collection $\mathcal{H}$, the directness of $(\Gamma, \leq)$ is a consequence - see \cite{szazfundamentaltheorem}.

The net defined by these structures:

\[\mathfrak{N}(\sigma,\tau,R) = (\sigma, \tau),\]

is a Fleischer-Szabó-Száz net which gives the abstract gauge integral of \cite{szazfundamentaltheorem}.

In general, the relations in $\mathcal{R}$ are given, in specific examples, by gauges (see \cite{henstockmagnum}) in the collection of open sets of the topology of $\Omega$. Thus, the classical gauge integral is contained in this example, as are the integrals of \cite{boccutovrabelovariecan}, \cite{riecanvrabelovaoperator}, and other integrals of gauge type by specifying the spaces $X,Y,Z$ in Definition \ref{netriemannpartiallyordered}.
\end{example}

We now give an abstract generalization of the construction given in the last example. This is based on the concept of derivation basis in gauge integration, for which our exposition will be based  on \cite{skvortsov2025} (see also \cite{boccutoskvortsovabstract}, \cite{boccutoskvortsovwalsh} and the real line case in \cite{detivationbasis1} and \cite{derivationbasis2}). In Henstock's and Mcgill's terminologies of \cite{henstockmagnum}  and \cite{mcgillitnegrationinvectorlattices} respectively, we are going to integrate interval functions on division spaces (therefore, the lattice valued integral of  \cite{mcgillitnegrationinvectorlattices} is contained in the Net Riemann integral).

\begin{example} [Derivation Basis Integral]

We consider now the paved space $(\Omega, \mathcal{H})$ equipped with a measure $\mu$ such that $(\Omega, \mathcal{H}, \mu)$ is a measure space - i.e, $\mathcal{H}$ is a $\sigma$-algebra and $\mu$ is a real, positive $\sigma$-additive set function \footnote{We can actually weaken these hypothesis and consider the general group valued integral in Chapter 2 of \cite{henstockmagnum}. Nevertheless, the procedure is the same, so we omit further details.}. 

We call, following \cite{skvortsov2025}, a filterbase $\mathcal{B}$ in the product space $\mathcal{I} \times \Omega$, where $\mathcal{I}$ is a family of sets in $\mathcal{H}$ of positive measure, a derivation basis. We assume now the following conditions on $\mathcal{B}$, which we refer to \cite{skvortsov2025} for details:

\begin{enumerate} [(i)]
    \item The derivation basis $\mathcal{B}$ ignores no points;
    \item The derivation basis $\mathcal{B}$ has the partitioning property.
\end{enumerate}

In this case, we denote $P_{\beta}(L) $ the set of all $\beta$-partitions of a fixed $\mathcal{B}$-interval $L$. Then, by the partitioning property and the filterbase condition on $\mathcal{B}$, we have that $\{P_{\beta}(L)\}_{\beta \in \mathcal{B}}$ is a filterbase on the product $\mathcal{P}^{f}(L) \times \mathcal{F}(\Omega)$ -  \cite{skvortsov2025} again contains all necessary details.

Now, we associate to $\{P_{\beta}(L)\}_{\beta \in \mathcal{B}}$ a Fleischer-Szabó-Száz net. For that, we follow the procedure given in \cite{priolanetsandfilters} (see also \cite{bartlenetsandfilters} and \cite{keneddynetsandfilters}) for general nets and filterbasis. That is, let

\[D(\{P_{\beta}(L)\}_{\beta \in \mathcal{B}}) = \{\alpha = (\pi, P_{\beta}(L)): \pi \in P_{\beta}(L) \ \text{and}  \ (\pi, P_{\beta}(L)) \in \{P_{\beta}(L)\}_{\beta \in \mathcal{B}}\},\]

which is directed (as $\{P_{\beta}(L)\}_{\beta \in \mathcal{B}}$ is a filterbase) by the relation:

\[(\pi_{1}, P_{\beta_{1}}(L)) \leq (\pi_{2}, P_{\beta_{2}}(L)),\]

if and only if $P_{\beta_{2}}(L) \subseteq P_{\beta_{1}}(L)$. Define now the net $\{x_{\gamma}\}_{\gamma \in D(\{P_{\beta}(L)\}_{\beta \in \mathcal{B}})} $ given by:

\[\gamma = (\pi, P_{\beta}(L)) \mapsto x_{\gamma} =  \pi,\]

which has, by definition, values in $\mathcal{P}^{f}(L) \times \mathcal{F}(\Omega)$. Then, if we pick, using the same notation as in the examples above,

\[\Gamma = D(\{P_{\beta}(L)\}_{\beta \in \mathcal{B}}), \]

the net:

\[\mathfrak{N}(\pi, P_{\beta}(L)) = x_{(\pi, P_{\beta}(L))} = \pi,\]

where $\pi = (\sigma_{(\pi,P_{\beta}(L))},\tau_{(\pi,P_{\beta}(L))})$ is, by definition, a finite collection of $\beta$-intervals and points, is a Fleischer-Szabó-Száz net for the set $L$, for which the Riemann sums based on a set function $\eta$ as in Definition \ref{netriemannpartiallyordered} are given by:

\[S_{(\pi, P_{\beta}(L))}(f,\eta) = \sum_{(x,I) \in \pi} f(x)\eta(I),\]

which can be seen to be the same as the ones given in Definition 1 of \cite{skvortsov2025} for the integral (in his notation) $\int_{L}\Phi$. Therefore, his integral is contained in the Net Riemann integral.

Finally, we point out that a similar construction taking $\Gamma = \mathcal{B}$ as the (directed set) domain of the Fleischer-Szabó-Száz net can be made. From the point of view of the convergence structures used, it gives the same results as above. 

As a last comment for this example, we point out that the by using the inverse construction of associating a net to a filterbase, it should be possible, by taking a specific $\Gamma$, to obtain the equivalence of the two integrals of the present example in this restricted case, but we will not pursue this route here.

\end{example}

We finish this stream of examples with Millington integral of \cite{millingtonruns} (see also \cite{millingtonriesz}), constructed as an abstract integration in terms of runs, an alternative concept to nets and filterbasis created by \cite{runs} (in fact, with integration theory in mind!). As to not introduce even more concepts in this example, we will translate the results of Millington's to the language of nets, and also work with a more restricted set of coverings and refer to Millington's original article for more details concerning the definitions involved here. 

\begin{example} [Millington Process Integral] \label{millingtonintegral}

For \cite{millingtonruns} construction, consider $D$ a directed (downwards or upwards) set, $(\Omega,\mathcal{H})$ as in Examples \ref{KolmogorovSintegralconvergencestructures} and \ref{sionintegralpartiallyorderedconvergence}, $A (\subseteq \Omega) \in \mathcal{H}$ a fixed set, and let $\mathcal{T}$ be the set of all truncations, as in Example \ref{sionintegralpartiallyorderedconvergence}, in $\mathcal{P}(A)$, and $\mathcal{C}$ the set of all choice functions in $\mathcal{H}$ (i.e, $h(C) = c \in C$, for $h: \mathcal{H} \rightarrow \Omega$). Now, consider the net 

\[\mathfrak{M}: D \times \mathcal{T} \rightarrow \mathcal{P}(A) \times \mathcal{C} \times \mathcal{T},\]

defined by:

\[\mathfrak{M}_{1}(d, \Delta) = (P_{d}, h_{d}, \Delta),\]

where we endow $D \times \mathcal{T}$ and $\mathcal{P}(A) \times \mathcal{C} \times \mathcal{T}$ with the product directed partial order specific to each case, with the individual directions of the components of each product being the same as in Example \ref{sionintegralpartiallyorderedconvergence}.

Now, with the same structures, consider a second net:

\[\mathfrak{M}_{2}: \mathcal{P}(A) \times \mathcal{C} \times \mathcal{T} \rightarrow \mathcal{F}(\mathcal{P}(A)) \times \mathcal{F}(\Omega),\]

defined by:

\[\mathfrak{M}_{2}(P, h, \Delta) = (\Delta(P), h_{\Delta(P)}),\]

where $h_{\Delta(P)}$ is the choice function $h$ applied in each member of $\Delta(P)$, i.e, $h_{\Delta(P)}$ is a finite collection of points from $\Omega$.

In this case, take $\Gamma = D \times \mathcal{T}$ with the aforementioned partial directed order. Then, the net (with values in the product of finite collection of sets and points from $\mathcal{H}$ and $\Omega$ respectively):

\[\mathfrak{N}: \Gamma \mapsto (\sigma, \tau),\]

defined by:

\[\mathfrak{N}(d, \Delta) = (\mathfrak{M}_{2} \circ \mathfrak{M}_{1})(d, \Delta), \]

that is, $(\sigma, \tau) = (\Delta(P_{d}), h_{d,\Delta(P_{d})})$, is a Fleischer-Szabó-Száz net such that the Riemann sums corresponding to it are equal to the second component of the integration run in Definition 2.3 of \cite{millingtonruns}.  

As in the Derivation Basis integral, a suitable choice of a Fleischer-Szabó-Száz net should give rise to a Millington Process integral in general, but we will not pursue this here. 

\end{example}

Further examples of the Net Riemann integral can be constructed readily by the same methods above. Indeed, we shall enumerate only those integration procedures that can be fitted, in a method of construction, in some of the examples above. We also stress that, in Example \ref{sionintegralpartiallyorderedconvergence}, we noticed the fact that the net Riemann integral contain the Lebesgue type integrals of \cite{phillipsintegratonarticle} and \cite{rickartintegralarticle}. Some more examples of the Lebesgue case will be given later.

\begin{enumerate}
    \item The Riemann-Stieltjes, Moore-Pollard, Young and Young-Stieltjes type integrals of \cite{candelororiemann}, \cite{hayesstieltjes}, \cite{gowurin} and \cite{dudleynorvaisa}, as well as the finite partition refinement integral of \cite{edwardswaymentvintegral} (which they call the $v$-integral).
    \item The multilinear integrals of Young, Stieltjes and Moore-Pollard in the sense of \cite{halilovicthesis}, developed further in \cite{halilovicmultilinear}, and more recently in \cite{halilovicnewest}.
    \item The Burkill-Cesari-Henstock integrals of \cite{henstockburkillthesis}, \cite{sworowskiburkill} 
    \cite{boccutosambuciniburkill} and \cite{skvortsovburkill}. In fact, \cite{chernyavskii1} and \cite{chernyavskii2} obtains a result that the Kolmogorov $S$-integral (i.e, for finite partitions) contains the group-valued integral of Burkill.
    \item The finite partition refinement type integral of \cite{saeksgoldsteinintegral}, called the Cauchy integral. 
    \item The (SL), metric-semigroup valued and topological Henstock-Kurzweil integrals of \cite{slintegralboccutoriecan}, \cite{abstractboccutocandelororiecan}, \cite{abstracttopologicalboccutoriecan}, \cite{improperboccutoriecan} and \cite{vrabelevrabelova} (omitting some more publications of the same authors, but with the same integrals).
    \item The Saks type integral of \cite{bochnerkyfanintegral}.
    \item The integral of \cite{robdera2025}, which is based on nets of partitions (i.e $\Gamma$ is a directed set based on partitions). 
    \item The Riemann partition integral of Erben and Grimeisen presented in \cite{erberthesis}, \cite{grimeisenerben1}, \cite{grimeisenerben2}.
    \item The general Riemann integral of \cite{spaltenstein}, which covers the Denjoy integral and  other generalized (Riemann) integrals.
    \item The BV-sets and geometric theory of Riemann integration of \cite{pfeffer1993riemann}, as well as the locally compact Hausdorff space integrals of \cite{ahmed1986riemann}, \cite{pfeffer1969integral} and \cite{pfeffer1970integral}.
    \item The non-tagged gauge integral of \cite{mcshaneriemann} and \cite{mcgillelementary}, as well as the Mcshane integral in quasi-Radon measure spaces of \cite{rodriguezdobrakov}, \cite{fremlin1994integration}, \cite{reynolds1997generalized}, \cite{fremlin1995generalized} and \cite{fremlin1994henstock} with a textbook treatment of the Banach case contained in \cite{schwabikgouju}. Also, the $H$-integral of Ridder, which is strictly general than the Henstock-Kurzweil integral, is contained in our integral by the results of \cite{leeridderintegral}.
    \item The totalization integral of \cite{totalizations}.
    \item The topological semigroup-valued truncation integrals of \cite{goguadze1}, \cite{goguadze2}, \cite{Impens} and \cite{delangheimpens}, which are constructed in the same way as in \cite{sionsemigroup}.
    \item The product Kolmogorov integrals of \cite{chernyavskii3}, \cite{areshkinkoroleva} and \cite{ferreira2004vector} in the form of the abstract Riemann net integral for (commutative) products in groups.
\end{enumerate}

Therefore, the net Riemann integral provides a uniquely and very general integration procedure in integration theory literature, which can be used in its various specializations to solve diverse problems in analysis (as can be gauged from the references just cited). A cursory look at the vast bibliography and examples of \cite{henstockmagnum} will also produce many more examples, as his division space integral is included in ours. 

We finish this section by exposing two shortcomings of these examples:

\begin{enumerate} [(i)]
    \item The relation between them and the Lebesgue, and other types of non-Riemann integration procedures, has not been fully clarified. 
    \item There are integrals, such as \cite{kolmogorovuntersuchungen} type $II$ integral, \cite{birkhofforiginalarticle} and \cite{goguadzebook} integrals, the integrals of \cite{popescurickart}, \cite{popescu1}, \cite{popescu2}, \cite{popescu3}, \cite{popescu4} and \cite{tulceaadditive}, \cite{tulceaensemble}, \cite{tulceaordone} and \cite{deleanu} that use infinite Riemann sums in the approximating scheme, and therefore are not contained ("directly") in the examples above and haven't been studied much in the literature. 
\end{enumerate}

We study first point (ii), and then prove convergence theorems for the net Riemann integral and the infinite sums Riemann type-integrals to obtain a possible solution to (i).

\section{Summability in Riesz Spaces and Integrals Based on Infinite Riemann Sums.} \label{sectionsummabilitykolmogorov}

In this section, we shall study a type of integral that uses infinite Riemann sums to base its approximation. Various integrals in the literature use this approach (see point (ii) in the ending of last section), but only is some restricted settings such as locally convex spaces (see \cite{popescurickart}) or real functions and measures (see \cite{goguadzebook} for an in depth study of this case). To obtain such integration procedure in our case, we have to first develop the rudiments of the concept of conditionally and unconditionally summable series, as well as their properties, in convergence spaces. Besides the references cited in this section, we draw the reader attention to the concept of abstract axiomatic sums in \cite{freni2023vector}, \cite{nielsen2025algebraization}, \cite{katz1965infinite} and \cite{andres2025categories}.

To begin the preparations to the integral, we note that, as said above, we have two notions of converging series, one corresponding to unconditional convergence, and other to conditional convergence, as in the classical reference of  \cite{bourbaki2013general}. In this context, consider the next definition:

\begin{definition} [Unconditionally Convergent Series] \label{UnconditionallyConvergentSeries}

Let $(X, \eta)$ be a vector space equipped with a net convergent structure $\eta$. Then, we say that a net $\{x_{i}\}_{i \in I}$ in $X$, with partial (unordered) sums of elements (in $X$)

\[\{\sum_{j \in J}x_{j}\}_{J \in \mathcal{F}(I)},\]

has an unconditionally summable (or convergent) series to $x \in X$, which we represent by $x = \sum_{i \in I}x_{i}$, if:

\[x = \lim_{(\mathcal{F}(I),\subseteq ,\eta)}\sum_{j \in J}x_{j}\]
    
\end{definition}

Similarly, we have the notion of conditionally convergence of series, which will be developed only for sequences:

\begin{definition} [Conditionally Convergent Series]

Let $(X, \eta)$ be a vector space equipped with a net convergent structure $\eta$. Then, we say that a sequence of elements $\{x_{n}\}_{n \in \mathbb{N}}$, with partial (ordered) sums of elements (in $X$)

\[\{\sum_{i=1}^{n}x_{i}\},\]

has a conditionally summable (or convergent) series to $x \in X$, which we represent by $x = \sum_{i=1}^{\infty}x_{i}$, if:

\[x = \lim_{(\mathbb{N},\eta)}\sum_{i=1}^{n} x_{i}\]
    
\end{definition}

Similar notions, specifically the first one, can be found in \cite{Higgs}, \cite{brunker} and \cite{hebisch}. We now prove some connections between these two concepts in general net convergence structures. 

\begin{proposition} [Unconditional Convergence implies Conditional Convergence] \label{unconditionalimpliesconditional}

Let $(X, \eta)$ be a vector space equipped with a net convergent structure $\eta$. and $\{x_{n}\}_{n \in \mathbb{N}}$ a sequence with unconditionally convergent series. Then, it has a conditionally convergent series with same limit.
    
\end{proposition}

\begin{proof}
    We prove this result using a quasi-subnet argument. In fact, suppose that $\{x_{n}\}_{n \in \mathbb{N}}$ has a unconditionally convergent series, and take an arbitrary $J \in \mathcal{F}(\mathbb{N})$. Choose $N_{0} = \{1, \cdots, \max(J)\}$. Then, if we take $n \in \mathbb{N}$ with $n \geq \max(J)$, it is clear that:

    \[\{\sum_{i=1}^{n}x_{i}\}_{n \geq \max(J)} \subseteq \{\sum_{k \in K }x_{k}\}_{J \subseteq K},\]

    which shows that $\{\sum_{i=1}^{n}x_{i}\}_{n \in \mathbb{N}}$ is a quasi-subnet of $\{\sum_{j \in J}x_{j}\}_{J \in \mathcal{F}(\mathbb{N})}$. As this last one is convergent to $x \in X$, it follows that, by axiom (ii) of Definition \ref{netconvergencestructure}, the sequence has a conditionally convergent series to $x$. 
\end{proof}

Regarding to the tail of a series, we have, by a similar argument:

\begin{proposition} [Tail of Series] \label{tailofseries}

Let $(X, \eta)$ be a vector space equipped with a net convergent structure $\eta$. and $\{x_{n}\}_{n \in \mathbb{N}}$ a sequence with a conditionally convergent series. Then, the tail of the series, i.e., the sequence:

\[\{\sum_{i=n}^{\infty}x_{i}\}_{n \in \mathbb{N}},\]

is convergent (with respect to $\eta$) to $0$.
    
\end{proposition}

Both results, specially Proposition \ref{unconditionalimpliesconditional}, are interesting in the context of measure theory in partially ordered spaces with net convergence structures compatible with order (in the sense of the last section). Regarding this context, the next definition is also useful:

\begin{definition} [Monotonically Convergence Property]

Let $(X, \eta)$ be a Riesz space equipped with a partially ordered net convergent structure $\eta$. Then, we say that $(X, \eta)$ has the monotonically convergence property (MCP) if, given an increasing bounded above net with a supremum $\{x_{i}\}_{i \in I}$, it converges to its supremum in $\eta$, i.e.

\[ \lim_{(I, \eta)} x_{i} = \sup_{i \in I}x_{i}\]
\end{definition}

This is motivated by the following classical result, which can be found in various references, such as  \cite{pavlakosintegration}, \cite{papangelouorder}, \cite{zaanenluxemburg1}, \cite{boccutoxenofon}, \cite{boccutovrabelovariecan} (for further examples, see specially \cite{boccutoxenofon}):

\begin{proposition} [Spaces with the MCP]

If $X$ is a Dedekind complete Riesz space with partially ordered net convergence structure $\eta$ being $(o_{3})$ or $(o)$ convergence, then it has the MCP property.
\end{proposition}

We also have the following example of a topological convergence with the MCP, showing that the concept is not restricted to spaces with a "purely" order based convergence.

\begin{example}
    Let $(X, \eta)$ be the space $L^{0}(\Omega, \Sigma, \mathbb{P})$ of (equivalence classes of) real Borel-measurable functions defined on a probability space, equipped with convergence in probability, and the partial ordering of almost sure pointwise inequality. It is not difficult to check that this convergence is topological and gives a partially ordered net convergence structure in $X = L^{0}(\Omega, \Sigma, \mathbb{P})$. In this case, taking $\{X_{n}\}_{n \in \mathbb{N}}$ a monotone sequence of random variables in $X$ bounded above (by a random variable $Y$), it follows from Theorem 3.5 of \cite{gut} that this sequence has a limit in probability given by its supremum. It follows that the present $(X, \eta)$ has the MCP.
\end{example}

In this case, as we will use repeatedly, we have the following result, which can be proved as in Lemma 1.2 of \cite{pavlakosintegration} using a quasi-subnet argument:

\begin{proposition} [Non-negative Series Convergence] \label{unconditionalconvergencegeneral}

Let $(X, \eta)$ be a Riesz space equipped with a (locally solid) net convergent structure $\eta$ with MCP property and $\{x_{n}\}_{n \in \mathbb{N}}$ a non-negative sequence with conditionally convergent series and partial sums $\{\sum_{i=1}^{n}x_{i}\}_{n \in \mathbb{N}}$ bounded above in $X$. Then, the series associated to  $\{x_{n}\}_{n \in \mathbb{N}}$ is unconditionally convergent and:

\[\sum_{i=1}^{\infty}x_{i} = \lim_{(\mathcal{F}(\mathbb{N}), \subseteq, \eta)} \sum_{j \in J} x_{j} = \sup_{n \in \mathbb{N}} \sum_{j=1}^{n}x_{i}\]
    
\end{proposition}

\begin{observation} [Ideal Infinite Element]

We can dispense with the existence of the supremum, or the MCP property, in the results above by annexing an ideal element $\infty$ to $X$ in such a way that every increasing sequence has a supremum equal to it. Another alternative would be to work with the Dedekind completion of such space (see \cite{fleischer1988convergence} for a general treatment of this case).  
    
\end{observation}

For $(D)$-convergence, which in general has the MCP property only for sequences, we have:

\begin{proposition} \label{unconditionalconvergence(D)}
    Let $(X, \eta)$ be a Dedekind complete Riesz space where $\eta$ is the partially ordered net convergence structure given by $(D)$-convergence. Then, a non-negative sequence $\{x_{n}\}_{n \in \mathbb{N}}$ has a conditionally convergent series if and only if it has an unconditionally convergent series with the same limit.
\end{proposition}

\begin{proof}
    The converse is valid is general by Proposition \ref{unconditionalimpliesconditional}, so we prove the first implication. As the series of $\{x_{n}\}_{n \in \mathbb{N}}$ is conditionally convergent and $X$ is Dedekind complete, it follows that, by the proposition above and Proposition \ref{(D)ando}:

    \[\sum_{n=1}^{\infty}x_{n} = (D) - \lim_{n \rightarrow \infty}\sum_{i=1}^{n}x_{i} = (o) - \lim_{n \rightarrow \infty}\sum_{i=1}^{n}x_{i} = \sup_{n \in \mathbb{N}} \sum_{i=1}^{n}x_{i},\]

    which means that we can find a non-negative sequence $\{p_{n}\}_{n \in \mathbb{N}}$ in $X$ such that $p_{n} \downarrow 0$ and:

    \[\sup_{n \in \mathbb{N}} \sum_{i=1}^{n}x_{i} - \sum_{i=1}^{n}x_{i} \leq p_{n}, \]

    for each $n$. Define $a_{ij} = p_{j}$, for every $i \in \mathbb{N}$. Then, by the definition of $\{p_{n}\}_{n \in \mathbb{N}}$, $\{a_{ij}\}$ is a $(D)$-sequence. Fix now any $\varphi \in \mathbb{N}^{\mathbb{N}}$ and take, independent of this $\varphi$, $J_{0} = \{1, \cdots, k\} \in \mathcal{F}(\mathbb{N})$, where $k \in \mathbb{N}$, $k > 1$, is fixed. Then, if $J \in \mathcal{F}(\mathbb{N})$ is such that $J_{0} \subseteq J$,

    \[\sum_{j \in J_{0}}x_{j} = \sum_{i=1}^{k}x_{i} \leq \sum_{j \in J}x_{j} .\]

    Therefore, 

    \[\sup_{n \in \mathbb{N}} \sum_{i=1}^{n}x_{i} - \sum_{j \in J}x_{j} \leq  \sup_{n \in \mathbb{N}} \sum_{i=1}^{n}x_{i} - \sum_{j \in J_{0}}x_{j} \leq p_{k} \leq \bigvee_{i=1}^{\infty}a_{i\varphi(i)} =  \bigvee_{i=1}^{\infty}p_{\varphi(i)}, \]

    which shows that $\{\sum_{j \in J}x_{j}\}_{J \in \mathcal{F}(\mathbb{N})}$ is $(D)$-convergent to $\sup_{n \in \mathbb{N}} \sum_{i=1}^{n}x_{i}$, and therefore that the sequence has a unconditionally convergent series with the same (conditional) limit.

\end{proof}

These results allow us to construct a theory of $\sigma$-additive measures with values in spaces with a general net convergence structure, and having further interesting properties in presence of a partial order. In this regard,  we now state a fundamental definition:

\begin{definition} [$\sigma$-additive Measure] \label{sigmaadditivemeasurenetconvergence}

Let $(\Omega, \mathcal{H})$ a paved space and $(Y, \eta)$ be a set equipped with a net convergence structure $\eta$. Then, we say that a set function:

\[\mu: \mathcal{H} \rightarrow Y,\]

is $\sigma$-additive, or a $\sigma$-additive measure, if for every disjoint sequence $\{A_{n}\}_{n \in \mathbb{N}}$ in $\mathcal{H}$ such that $\bigcup_{n \in \mathbb{N}}A_{n} \in \mathcal{H}$, then 

\[\mu(\bigcup_{n \in \mathbb{N}}A_{n}) = \lim_{(\mathcal{F} (\mathbb{N}),\subseteq, \eta)}\sum_{j \in J}\mu(A_{j})\]
\end{definition}

Its important to notice that the series in the definition above is a unconditional one, because the use of the (commutative operation of) countable union in its definition - a point forcefully put in the topological context by \cite{fleischertraynorequicontinuity}, and which gives interesting consequences in case we consider measure theory in non-commutative structures (see \cite{papnull}). 

By using Propositions \ref{unconditionalimpliesconditional} and \ref{tailofseries}, its possible to prove the following fact directly (as it depends only in the series definition):

\begin{proposition} [Monotone Continuity and $\sigma$-Additivity]

Let $(\Omega, \mathcal{H})$ a paved space by a ring and $(Y, \eta)$ be a set with partially ordered net convergence structure $\eta$. Then,  a finitely additive set function:

\[\mu: \mathcal{H} \rightarrow Y^{+},\]

is $\sigma$-additive if and only if for every decreasing sequence (in $\mathcal{H}$) $\{A_{n}\}_{n \in \mathbb{N}} \downarrow \emptyset$,

\[\lim_{(\mathbb{N}, \eta)}\mu(A_{n}) = 0\]

If $Y$ is a Riesz space equipped with a partially ordered net convergence structure and $\mu$ assumes values in the non-negative cone of $Y$, the limit is decreasing in $n$.
    
\end{proposition}

It follows that our definition of $\sigma$-additive measure is the same as the one used in the theory of non-negative $\sigma$-additive measures in Riesz spaces equipped with order or (D)-convergence - see Section 2.3 of \cite[Chapter 2]{boccutoxenofon} for details. 

Having a notion of series in spaces with abstract convergence structures allow us to define a different type of integral that is different in nature compared to the Net Riemann integral. 

Before the definition itself, consider the following auxiliary object and convention:

\begin{definition} [Choice Function]

Let $\Omega$ be an abstract (i.e, general) set. We say that a function $\delta: 2^{\Omega} \rightarrow \Omega$ is a choice function if, for each $A \in 2^{\Omega}$, $\delta(A) = \delta_{A}  \in A$.
    
\end{definition}

Also, for a countable collection of sets $P = \{\alpha\}_{\alpha \in P}$, we denote for $J \in \mathcal{F}(\mathbb{N})$,

\[\sum_{j \in J, \alpha \in P}g(\alpha_{j}),\]

where $g$ is a set function, to be the sum of the elements of $P$ indexed by $J$.

Now we may define a new integral, which in its original form is due to  \cite{kolmogorovuntersuchungen}, but the modern usage of it is due (initially) to  \cite{deleanu}, and Ionescu-Tulcea in \cite{tulceaadditive} and \cite{tulceaensemble}, called the $S^{*}$-integral:

\begin{definition} [Kolmogorov $S^{*}$ integral] \label{kolmogorovS*integralconvergencestructure}

Given and ordered pair:

\[((\Omega, \mathcal{H}), ((X,\eta_{1}), (Y,\eta_{2}), (Z,\eta_{3}); \bullet)),\]

consisting of 

\begin{enumerate} [(i)]
    \item $(\Omega, \mathcal{H})$, a paved space $\mathcal{H}$;
    \item $(X,\eta_{1}), (Z,\eta_{2}) (Z,\eta_{3})$,  spaces with (partially ordered, if Riesz spaces) net convergence structures connected by a bilinear product $\bullet$;
\end{enumerate}

we say that a function $f: \Omega \rightarrow X$ is $S^{*}$ integrable on $A \in \mathcal{H}$ with respect to a set function $\mu: \mathcal{H} \rightarrow Y$ if the following iterated limit exists and is independent of each choice function $\delta$:

\[\int_{A}^{S^{*}}fd\mu = \lim_{(\mathcal{P}(A), \leq_{ref},\eta_{3})}\lim_{(\mathcal{F}(\mathbb{N}), \subseteq,\eta_{3})}\sum_{j \in J, \alpha \in P}f(\delta_{\alpha_{j}})\mu(\alpha_{j}),\]

where $P = \{\alpha_{j}\}_{j \in \mathbb{N}} \in \mathcal{P}(A)$, the set of all disjoint countable coverings of $A$, equipped by the refinement partial ordering $\leq_{ref}$\footnote{Or only $\leq$ if the context is clear.}, which is assumed non-empty and directed\footnote{In general, that requires more hypothesis on the paving $\mathcal{H}$.}.

In general, we denote the second limit as the series $\sum_{\alpha \in P}f(\delta_{\alpha})\mu(\alpha)$.
\end{definition}

If in the previous definition we consider $\mathcal{P}(A)$, for $A \in \mathcal{H}$, consisting only of partitions of the set $A$, we call the resulting integral the $S^{*}$-partition Integral, for which the next results and examples continues to be valid. 

The integral in Definition \ref{kolmogorovS*integralconvergencestructure} may be paraphrased as follows: 

\begin{definition} [$S^{*}$-integral Unconditional Series] \label{S*integralassumofseries}

With the same structures of Definition \ref{kolmogorovS*integralconvergencestructure}, an map $f$ is $S^{*}$-integrable on $A \in \mathcal{H}$ if and only if the following limit exists:

\[\int_{A}^{S^{*}}fd\mu = \lim_{(\mathcal{P}(A), \leq_{ref},\eta_{3})}\sum_{\alpha \in P}f(\delta_{\alpha})\mu(\alpha),\]

where the series is understood as in Definition \ref{UnconditionallyConvergentSeries}, and $\mathcal{P}(A)$ is assumed non-empty and directed.

\end{definition}

In other words, whenever the $S^{*}$-integral exists, we assume the existence of the series on the integral for partitions at least cofinal for the convergence (i.e, which refines a fixed element of $\mathcal{P}(\Omega)$). We also comment that its possible to generalize the partitions above to specific classes generated by other means (i.e, gauges and similar techniques) and obtain integrals in the sense of \cite{duboisdefinition1} and \cite{duboisdefinition2}.

For the rest of the text, we denote by $L_{S^{*}}(\mu)$ the space of all functions $f: \Omega \rightarrow X$ that are $S^{*}$-integrable with respect to the (fixed) set function $\mu: \Omega \rightarrow Y$ on each $A \in \mathcal{H}$ with the integration structures of Definition \ref{kolmogorovS*integralconvergencestructure}. 

We now summarize some properties of the $S^{*}$-integral in this abstract setting. As the proofs are exactly the same as in the net Riemann integral, we omit them. Some of these may be found (in a different and less general form) in \cite{deleanu}.

\begin{proposition} [Properties of the $S^{*}$-Integral] \label{Properties oftheS*}

The $S^{*}$-Integral has the following properties, where we assume $X,Y,Z$ Riesz spaces in items with  inequalities:

\begin{enumerate} [(i)]
    \item Let $f,g \in L_{S^{*}}(\mu)$. Then, for each fixed $A \in \mathcal{H}$

    \[\int^{S^{*}}_{A}(f+g)d\mu = \int^{S{*}}_{A}fd\mu + \int^{S{*}}_{A}gd\mu,\]

    i.e, the integral is additive on the space $L_{S^{*}}(\mu)$;
    \item Let $f \in L_{S^{*}}(\mu)$ such that $f \geq 0$ and $\mu \geq 0$. Then, for each fixed $A \in \mathcal{H}$

    \[\int^{S^{*}}_{A}f d\mu \geq 0,\]

    i.e, the integral is a positive map on the space of integrable functions.
    \item Let $\mu \geq 0 $ and $f,g \in L_{S^{*}}(\mu)$ such that:

    \[g \geq f.\]

    Then, for each fixed $A \in \mathcal{H}$

    \[\int^{S^{*}}_{A}gd\mu \geq \int^{S^{*}}_{A}fd\mu,\]

    i.e, the integral is an isotone map on $L_{S^{*}}(\mu)$ equipped with the partial ordering inherited from $X$.

    \item Let $\mu \geq 0$ and $f \in L_{S^{*}}(\mu)$ such that $|f| \in L_{S^{*}}(\mu)$, then for each fixed $A \in \mathcal{H}$:

    \[|\int^{S^{*}}_{A}fd\mu| \leq \int^{S^{*}}_{A}|f|d\mu.\]

    \item Let $f: \Omega \rightarrow X$ and $\mathcal{H}$ an algebra, such that, given $A,B \in \mathcal{H}$ disjoint, $f$ is $S^{*}$-integrable on $A$ and $B$. Then, $f$ is integrable on $A \cup B$ and:

    \[\int^{S^{*}}_{A \cup B}fd\mu = \int^{S^{*}}_{A}fd\mu+\int^{S^{*}}_{B}fd\mu\]

    \item If $(Z, \eta_{3})$ is complete and $\mathcal{H}$ an algebra, given $A,B \in \mathcal{H}$ such that $A \subseteq B$ in $\mathcal{H}$, $f$ being $S^{*}$-integrable on $B$ implies $S^{*}$-integrability on $A$.

    \item Let $f: \Omega \rightarrow X$ be constant function equal to $c \in X$ and $\mu: \mathcal{H} \rightarrow Y^{+}$ $\sigma$-additive. Then f is $S^{*}$-partition integrable on each $A \in \mathcal{H}$ and:

    \[\int_{A}^{S^{*}}fd\mu = c\mu(A)\]
\end{enumerate}
    
\end{proposition}

What is interesting about this integral is that it gives $\sigma$-additivity automatically in the presence of a summability condition, which \cite{deleanu} (in section
$4$ of the article, for which we refer the reader for details) calls the $(*)$-condition, and is satisfied for any complete (Hausdorff) topological group or/and Dedekind complete lattice ordered group. In fact, he obtain the following result:

\begin{theorem} [$\sigma$-additivity of the $S^{*}$-integral] \label{sigmaadditives*integral}

Let $f: \Omega \rightarrow X$ be a $S^{*}$-integrable function with respect to a(n) (arbitrary) set function $\mu: \mathcal{H} \rightarrow Y$, where $\mathcal{H}$ is a $\sigma$-ring of subsets of $\Omega$ and the integrability occurs with respect to a fixed net convergence structure where $Z$ satisfies the $(*)$-condition. Then, the set function defined on $\mathcal{H}$:

\[\nu^{S^{*}}_{f}({A})= \int_{A}^{S^{*}}fd\mu, \]

is a $\sigma$-additive measure with values in $Z$.
    
\end{theorem}

We refer the reader to \cite[Page 33]{deleanu} for a proof. 

We now formulate a result to connect the net Riemann integral and the $S^{*}$-integral. This aforementioned equivalence is of a restricted nature because it shows 
that the existence of the $S^{*}$-integral implies the existence of a particular example of the Net Riemann integral when $Z$ equipped with a net convergence structure that is topological (notice that $X$ and $Y$ need not be of this type). In fact, we have the following result that proves a fact cited by \cite{sionsemigroup}, on Page 16, without proof:

\begin{proposition} \label{sandsionintegrals}
    Let $f \in L_{S^{*}}(\mu)$ be a $S^{*}$-integrable function in the context of conditions (i) and (ii) of Definition \ref{kolmogorovS*integralconvergencestructure}, where $Z$ is equipped with a net convergence structure which is generated by a vector space topology. Then, $f$ is Sion integrable\footnote{Here and thereafter, we denote the Sion integral of $f$ on $\Omega$ by $\int_{\Omega}^{(Si)}fd\mu$}.
\end{proposition}

\begin{proof}

As the convergence in $Z$ is topological, we may find a topology $\tau$ for which each convergence net in $Z$ is convergent in this vector topology. Denote by $\mathcal{U}$ the filter of absorbent and symmetric neighborhoods of $\tau$. In this case, as $f$ is $S^{*}$-integrable on $\Omega$ (the proof for other sets is immediate), fix an arbitrary $U \in \mathcal{U}$ and take symmetric and absorbent $V \in \mathcal{U}$ such that $V + V \subseteq U$ and $P_{0} \in \mathcal{P}(\Omega)$ such that, denoting by $z = \int_{\Omega}^{S^{*}}fd\mu$,

\[\sum_{\alpha \in P}f(\delta_{\alpha})\mu(\alpha) - z \in V,\]

for every $P \geq P_{0}$ in $\mathcal{P}(\Omega)$ and independent of the choice function $\delta$.

By definition of the unconditional summability of each series above, for the given $V \in \mathcal{U}$ and $P$ above we may find $J(P) = J_{P} \in \mathcal{F}(\mathbb{N})$ such that, for every $J \in \mathcal{F}(\mathbb{N})$ such that $J_{P} \subseteq J$,

\[\sum_{j \in J, \alpha \in P}f(\delta_{\alpha_{j}})\mu(\alpha_{j}) - \sum_{\alpha \in P}f(\delta_{\alpha})\mu(\alpha) \in V .\]

Now, choose $(P_{0}, \Delta_{0}) \in \mathcal{D}(\Omega)$ such that the truncation is given by:

\[\Delta(Q) = \{\{\alpha_{j}\}\}_{j \in J_{Q}},\]

where $Q \in \mathcal{P}(\Omega)$ is such that $Q \geq P_{0}$. In this case, for each $(P, \Delta) >> (P_{0}, \Delta_{0})$, each sum:

\[\sum_{\alpha \in \Delta(P)}f(\delta_{\alpha})\mu(\alpha),\]

is of the form $\sum_{j \in J, \alpha \in P}f(\delta_{\alpha_{j}})\mu(\alpha_{j}) $, where $J_{P} \subseteq J$. Therefore, by the choice above, 

\[\sum_{\alpha \in \Delta(P)}f(\delta_{\alpha})\mu(\alpha) \in z+U.\]

As the choice of $U \in \mathcal{U}$ is arbitrary, it follows that $f$ is Sion integrable on $\Omega$ and:

\[\int_{\Omega}^{S^{*}}fd\mu = z = \int_{\Omega}^{(Si)}fd\mu\]

\end{proof}

In the partially ordered case, we can prove a similar statement, but with a further restriction. To motivate this restriction, we note that several limit theorems for lattice-group/Riesz space-valued measures can be proved only if we choose a $(D)$-sequence for convergence that is uniform in some sense: either the same for a family of measures or for a collection of sets. This is, for example, the technique of \cite{boccutocandelorovitalihahnsaks} to prove a type of Vitali-Hahn-Saks theorem for lattice group-valued measures, the same occurs in \cite{boccutodimitriou2011}. Intuitively, we need this to control simultaneously an uncountable number of nets, and reduce the case to an application of Fremlin's lemma. In this setting, then, we have:

\begin{proposition}
 Let $f \in L_{S^{*}}(\mu)$ be a $S^{*}$-integrable function in the context of conditions (i) and (ii) of Definition \ref{kolmogorovS*integralconvergencestructure}, where $Z$ is equipped with the partially ordered net convergence structure of $(D)$-convergence. Suppose that, given $P_{0} \in \mathcal{P}(\Omega)$, for each $P \in \mathcal{P}(\Omega)$ such that $P \geq P_{0}$, we may find a common (unique) regular for the convergence of the series $\{\sum_{\alpha \in P}f(\delta_{\alpha})\mu(\alpha)\}_{\alpha \in P}$ (for every choice function $\delta$), which is assumed convergent for each $P$ chosen this way. Then, $f$ is Sion integrable. 
\end{proposition}

\begin{proof}
    By definition, as $f$ is $S^{*}$-integrable on $\Omega$, there exists a $(D)$-sequence $\{a_{ij}\}_{ij}$ in $Z$ such that, for each $\varphi \in \mathbb{N}^{\mathbb{N}}$ we may choose $P_{0} \in \mathcal{P}(\Omega)$ such that, for each $P \geq P_{0}$ in $\mathcal{P}(\Omega)$,

    \[|\sum_{\alpha \in P}f(\delta_{\alpha})\mu(\alpha) - z| \leq \bigvee_{i=1}^{\infty}a_{i\varphi(i)},\]

    where $z \in Z$ is the $S^{*}$-integral of $f$ on $\Omega$. 

    Choose a common $(D)$-sequence for the convergence of the nets $\{\sum_{j \in J, \alpha \in P}f(\delta_{\alpha_{j}})\mu(\alpha_{j})\}_{J \in \mathcal{F}(\mathbb{N})}$. That is, there exists, for all $P \in \mathcal{P}(\Omega)$, a $(D)$-sequence $\{b_{ij}\}_{ij}$ in $Z$ such that, for each $\varphi \in \mathbb{N}^{\mathbb{N}}$ we may choose $J = J(P) \in \mathcal{F}(\mathbb{N})$ such that, for each $J^{'} \in \mathcal{F}(\mathbb{N})$ which contains $J$,

\begin{equation} \label{eq*}
    |\sum_{j \in J^{'}, \alpha \in P}f(\delta_{\alpha_{j}})\mu(\alpha_{j}) - \sum_{\alpha \in P}f(\delta_{\alpha})\mu(\alpha)| \leq \bigvee_{i=1}^{\infty}b_{i\varphi(i)}.
\end{equation}

   By Fremlin lemma (in this setting, the Dedekind completeness and weak-$\sigma$-distributiveness are not needed), there exists $\{c_{ij}\}_{ij}$ a $(D)$-sequence in $Z$ such that:
   
\begin{equation} \label{eq**}
    \bigvee_{i=1}^{\infty}a_{i\varphi(i)} + \bigvee_{i=1}^{\infty}b_{i\varphi(i)} \leq \bigvee_{i=1}^{\infty}c_{i\varphi(i)}.
\end{equation}

Now, fix $P_{0}^{(Si)} = P_{0}$ and define a truncation $\Delta_{0}$ by:

\[\Delta_{0}(Q) = \{\alpha_{j}\}_{j \in J(Q)},\]

where $J(Q) \in \mathcal{F}(\mathbb{N})$ is given in the convergence statement above for $Q \geq P_{0}$. Therefore, if $(P, \Delta) >> (P_{0}^{(Si)}, \Delta_{0})$, we have that, by equations \ref{eq*} and \ref{eq**},

\[|x - \sum_{\alpha \in \Delta(P)}f(\delta_{\alpha})\mu(\alpha)| \leq \bigvee_{i=1}^{\infty}c_{i\varphi(i)},\]

for each $\varphi \in \mathbb{N}^{\mathbb{N}}$ as above. That is, $f$ is Sion integrable on $\Omega$ and:

\[\int_{\Omega}^{S^{*}}fd\mu = z = \int_{\Omega}^{(Si)}fd\mu\]

\end{proof}

An interesting open question is if we may prove the same without supposing the existence of the unique $(D)$-sequence for each series of refining partitions. Maybe a possibility would be a conjunction of $Z$ being super Dedekind complete with another type of "order separability" property, but we leave this as an open problem.  

We finish this section now, and consider some convergence theorems for the integrals introduced so far, and justifications for them.

\section{Convergence Theorems for Net Riemann and Related Integrals.}

To obtain more refined properties of the net Riemann integral and correlated concepts introduced in previous sections, we need some convergence theorems, specially (order and topological) uniform ones. The purpose of these convergence theorems are twofold:

\begin{enumerate} [(i)]
    \item Obtaining uniform convergence theorems for the Net Riemann integral, and specially the $S^{*}$(-partition) integral allows us to compare them to topological and order based Lebesgue integrals in the literature, such as \cite{rodriguezsalazar}, \cite{fidelgeneral1}, \cite{fidelgeneral2}, \cite{haluskaintegrable}, \cite{pavlakosintegration}, \cite{wrightintegralmeasure} and
    to other order/topological based and sequentially/net defined integrals in the literature, unifying an and extending a wide range of integration procedures;
    \item Based on the comparisons contained in item (i), we may obtain  monotone and dominated convergence theorems and other limit type properties for these integrals, as well as initial steps for a general classification of such integration procedures.
\end{enumerate}

To obtain these convergence theorems, we would like to make as few assumptions as possible on the integrability of the sequence of functions and (order/topological based) limits involved. For the order case, we will adopt a generalization of the uniform convergence theorem of \cite{monteirouniform}, which is based on the proof of a related theorem in \cite{vrabelovariecanlast}, to the Henstock-Kurzweil integral (with respect to the Lebesgue measure) which, compared to \cite{integralmeasureandordering} - which extends the partially ordered operator-valued framework of \cite{riecanvrabelovaoperator} - posits less strong integrability conditions on the limit of the approximating functions. This theorem is given in the cited paper for the Lebesgue measure and the Kurzweil integral. We shown in this section that this theorem is actually valid for a wide class of net Riemann integrals, and even for a outside of null sets version of $S^{*}$-integral, which is necessary to compare it to the order-convergence integrals of \cite{pavlakosintegration} and \cite{pavlakosrepresentation} (originated in \cite{pavlakosthesis}). 

We then prove a general Henstock lemma and obtain from it monotone and dominated convergence theorems in the setting or Riesz spaces, and survey some topological versions of convergence theorems, such as the convergence theorems of \cite{fleischercontent}, \cite{fleischerinterchange}, which generalized \cite{sionsemigroup} dominated convergence theorem, and the outside of null sets topological convergence theorem of \cite{millingtonproduct}, which we generalize for a specific class of partition refinement integrals of the net Riemann integral, which contains more integrals than the one exposed by Millington (which is a variant of the Sion integral of \cite{sionsemigroup}). This last result in particularly important to obtain equivalences with Lebesgue integrals defined by a type of Egorov procedure of local uniform convergence. Finally, we prove a general convergence theorem in the setting of lattice normed spaces. 

We now fix some conventions for this section (and subsections):

\begin{enumerate} [(i)]
    \item The class of sets $\mathcal{H}$ will be a $\sigma$-algebra. Nevertheless, many results are valid supposing that this class is only a $\delta$-ring. 
    \item The existence of any integral with values on $Z$, when it is a Riesz space, will be understood to take place in the net convergence structure given by $(D)$-convergence, unless stated to the contrary. All other types of limits will be specified. 
    \item For $Z$ a Riesz space, the series appearing in the $S^{*}$-partition below will be understood to be taken in order convergence when the issue of existence arises. That will cause no problems with the requirement of $(D)$-convergence in $(i)$, as these will appear only when the functions and measures involved are non-negative  and $Z$ weakly $\sigma$-distributive and Dedekind complete: this gives a increasing sequence (or net, in the $\mathcal{F}(\mathbb{N})$ case) of (unordered) sums which, by Propositions \ref{equivalenceoand(D)convergence}, \ref{unconditionalconvergencegeneral} and \ref{unconditionalconvergence(D)} gives a well defined (and equal) limit in $(D)$-convergence as well. Therefore, we have compatibility with Definition \ref{kolmogorovS*integralconvergencestructure} considering (i), where $\eta_{3}$ is $(D)$-convergence.
    \item Wherever we use the net Riemann integral, we assume that the Szabó-Száz-Fleischer net is composed of sets such that:

    \[\sum_{i \in I_{\gamma}}\mu(\sigma_{\gamma_{i}}) = \mu(\bigcup_{i \in I_{\gamma}}\sigma_{\gamma_{i}}),\]

    for $\mu$ finitely additive, which is not the same as asking that the sequence of sets is disjoint, as it can have overlapping null sets - that is $A,B \in \mathcal{H}$ are not disjoint, but $A \cap B \in \mathcal{H}$ is such that $\mu(A \cap B) = 0$, which does not alter any of the problems below. 
\end{enumerate}

We begin by the following useful concept of \cite{monteirouniform} for the order convergence theorems:

\begin{definition} [Uniform Convergence in Order]

We say that a sequence of functions $f_{n}: \Omega \rightarrow X$ converges order uniformly to $f: \Omega \rightarrow X$ is there exists a non-negative and non-increasing sequence $\{u_{n}\}_{n \in \mathbb{N}}$ in $X$  such that $u_{n} \downarrow 0$ and:

\[|f_{n}(\omega)-f(\omega)| \leq u_{n},\]

for every $\omega \in \Omega$.

\end{definition}

One important comment is that we may modify the definition above to make a "tail based" convergence concept (i.e, making the inequality valid for a "sufficient large $n$", and not every $n$) based on $p_{n}$, but this modification will not cause any changes in the arguments below, as we can always make "tail convergence" modifications of these (that would not be the case if we consider uniform convergence based on nets of functions). 

We also have the following Cauchy definition from \cite{monteirouniform}:

\begin{definition} [Cauchy Sequence Uniformly in Order]

We say that a sequence of functions $f_{n}: \Omega \rightarrow X$ is uniformly Cauchy in order, or only uniformly Cauchy, if is there exists a non-negative and non-increasing sequence $\{u_{n}\}_{n \in \mathbb{N}}$ in $X$ such that $u_{n} \downarrow 0$ and for each $n \in \mathbb{N}$:

\[|f_{i}(\omega)-f_{j}(\omega)| \leq u_{n},\]

for every $\omega \in \Omega$ and $i,j \geq n$.

\end{definition}

It is straightforward to check that every uniformly order convergent sequence of functions is uniformly Cauchy. 

We enounce the following direct consequence of Definition \ref{netriemannpartiallyordered} and Proposition \ref{orderconvergenceboundednets}, which will be useful next.

\begin{proposition} \label{limsupconvergenceintegral}

Suppose $(Z, \eta_{3})$ is a Dedekind Complete Riesz space with $\eta_{3}$ being given by order convergence. Then the net Riemann integral of a function $f: \Omega \rightarrow X$ with respect to a set function $\mu: \mathcal{H} \rightarrow Y$ exists with respect to $\eta_{3}$ if and only if:

    \begin{enumerate} [(i)]
        \item The net $\{S_{\gamma}(f,\mu)\}_{\gamma \in \Gamma}$ is bounded in $Z$ and,
        \item The quantities $\limsup_{\gamma}S_{\gamma}(f,\mu)$ and $\liminf_{\gamma}S_{\gamma}(f,\mu)$ both exists, and

    \[\limsup_{\gamma}S_{\gamma}(f,\mu) = \bigwedge_{\gamma \in \Gamma} \bigvee_{\gamma^{'} \geq \gamma}S_{\gamma^{'}}(f,\mu) = \bigvee_{\gamma \in \Gamma} \bigwedge_{\gamma^{'} \geq \gamma} S_{\gamma^{'}}(f,\mu) = \liminf_{\gamma}S_{\gamma}(f,\mu).\]
    \end{enumerate}
\end{proposition}
    
For the main convergence theorem, we need the following analogue of Proposition 3.2 of \cite{monteirouniform}:

\begin{proposition} \label{monteirofernandez1}
    Let $\mu: \mathcal{H} \rightarrow Y^{+}$ be a regular integrator and $\{f_{n}\}_{n \in \mathbb{N}}$ a sequence of Net Riemann Integrable functions, with respect to $\mu$, defined on $\Omega$. If $\{f_{n}\}_{n \in \mathbb{N}}$ is uniformly Cauchy in order, then the following limit exists:

    \[\limsup_{n \rightarrow \infty} \int_{\Omega} f_{n} d\mu = \bigwedge_{m=1}^{\infty} \bigvee_{i=m}^{\infty} \int_{\Omega}f_{i}d\mu,\]

    and,

    \[\int_{\Omega}f_{n}d\mu \xrightarrow[]{o} \bigwedge_{m=1}^{\infty}\bigvee_{i=m}^{\infty} \int_{\Omega}f_{i}d\mu, \]

    where this convergence is in order. 
\end{proposition}

The proof is the same as in \cite{monteirouniform}, but adapted to our abstract framework. Therefore, to present the modifications and for completeness, we give the full description. 

\begin{proof}
    By Proposition \ref{limsupconvergenceintegral}, it is sufficient to show that $\{\int_{\Omega}f_{n}d\mu\}_{n \in \mathbb{N}}$ is bounded and that the following holds:

    \[\bigwedge_{m=1}^{\infty}\bigvee_{i=m}^{\infty} \int_{\Omega}f_{i}d\mu \leq \bigvee_{m=1}^{\infty}\bigwedge_{i=m}^{\infty}\int_{\Omega}f_{i}d\mu, \]

    as the reverse inequality is always valid for bounded sequences with values in $X$, and the limit in the proposition is a direct consequence of these inequalities. 

    Now, as $\{f_{n}\}_{n \in \mathbb{N}}$ is uniformly Cauchy, there exists a non-increasing sequence $\{u_{n}\}_{n \in \mathbb{N}}$ consisting of elements of $X$ such that $u_{n} \downarrow 0$ and, for each $n \in \mathbb{N}$,

    \[|f_{i}(\omega)-f_{j}(\omega)| \leq u_{n}, \ \forall \omega \in \Omega, \ i,j \geq k.\]

    Therefore, given $n \in \mathbb{N}$ and $i,j \geq n$, we obtain from this last equation that:

    \[f_{i}(\omega) \leq f_{j}(\omega) + u_{n}, \forall \omega \in \Omega,\]

    and, therefore, as these functions are Net Riemann integrable, and the integral in this case is an isotone functional for the given bilinear product, we have that, integrating both sides of this inequality and using the regularity of the integrator:

    \begin{equation} \label{equacao}
        \int_{\Omega}f_{i}d\mu \leq \int_{\Omega}f_{j}d\mu + u_{n}\mu(\Omega).
    \end{equation}

    In particular, from that it follows that, for $j,n=1$:

    \begin{equation} \label{eq4}
        \int_{\Omega}f_{1}d\mu - u_{1}\mu(\Omega) \leq \int_{\Omega}f_{i}d\mu \leq \int_{\Omega}f_{1}d\mu + u_{1}\mu(\Omega),
    \end{equation}

    Thus, $\{\int_{\Omega}f_{n}d\mu\}_{n \in \mathbb{N}}$ is a bounded sequence.

    By (\ref{equacao}), for each $n \in \mathbb{N}$, we have that:

    \[\bigvee_{i=n}^{\infty}\int_{\Omega}f_{i}d\mu \leq \bigwedge_{j=n}^{\infty}\int_{\Omega}f_{j}d\mu + u_{n}\mu(\Omega),\]

    which implies that,

    \[\bigwedge_{m=1}^{\infty}\bigvee_{k=m}^{\infty}\int_{\Omega}f_{k}d\mu - \bigvee_{m=1}^{\infty}\bigwedge_{k=m}^{\infty}\int_{\Omega}f_{k}d\mu \leq \bigvee_{i=n}^{\infty}\int_{\Omega}f_{i}d\mu - \bigwedge_{j=n}^{\infty}\int_{\Omega}f_{j}d\mu \leq u_{n}\mu(\Omega),\]

    from this last equation, and from the properties of the bilinear product, as $u_{n}\mu(\Omega) \downarrow 0$ when $n \rightarrow \infty$, we get the result. 

\end{proof}

The main result of this section, which uses the last result, is the analogue of Theorem 3.3 of \cite{monteirouniform}, but modified for the purposes of the Net Riemann integral, which we now state:

\begin{theorem} [Uniform Converge Theorem for the Net Riemann integral] \label{uniformconvergencetheoremriemann}

Let $Z$ be Dedekind complete and weakly $\sigma$-distributive, $\mu: \mathcal{H} \rightarrow Y^{+}$ a regular integrator and $\{f_{n}\}_{n \in \mathbb{N}}$ a sequence of Net Riemann integrable functions, with respect to $\mu$, defined on $\Omega$. If $f_{n}$ converges order uniformly to $f$, then $f$ is Net Riemann Integrable with respect to $\mu$ and:

\[\int_{\Omega} f_{n}d\mu \xrightarrow[]{o} \int_{\Omega}fd\mu,\]

where this last convergence is in order in $X$.

\end{theorem}

In the proof below, we follow mainly \cite{monteirouniform} and make some modifications inspired by the analogous result from \cite{vrabelovariecanlast}. In particular, we will assume that $\Gamma$ is upwards directed, the downwards directed case being a straightforward modification.

\begin{proof}

    As $f_{n}$ converges to $f$ order uniformly, there exists a non-increasing sequence $\{u_{n}\}_{n \in \mathbb{N}}$ in $X$ such that $u_{n} \downarrow 0$ and:

    \begin{equation} \label{eq5}
        |f_{k}(\omega)-f(\omega)| \leq u_{n}, \ \forall \omega \in \Omega \ \text{and} \ k \in \mathbb{N},
    \end{equation}

    which implies that $\{f_{n}\}_{n \in \mathbb{N}}$ is Cauchy uniformly.

    Therefore, by Proposition \ref{monteirofernandez1}, there exists $L \in Z$ such that the following order limit statement is valid:

    \[\int_{\Omega}f_{n}d\mu \xrightarrow[]{o} L.\]

    In other words, there exists a non-decreasing sequence $\{w_{n}\}_{n \in \mathbb{N}} \in Z$ such that $w_{n} \downarrow 0$ and:

    \begin{equation} \label{eq6}
        |\int_{\Omega}f_{n}d\mu - L| \leq w_{n}, \ \forall n \in \mathbb{N}.
    \end{equation}

    As each $f_{n}$ is Net Riemann integrable on $\Omega$, for each $n \in \mathbb{N}$, there exists a $(D)$-sequence $\{a^{(n)}_{ij}\}_{ij} \in Z$ such that, $\forall \ \psi \in \mathbb{N}^{\mathbb{N}}$, exists $\gamma_{n} \in \Gamma$ such that for each $\gamma \geq \gamma_{n}$ in $\Gamma$ we have

    \begin{equation} \label{eq7}
        |\sum_{i \in I_{\gamma}}f_{n}(\tau_{\gamma_{i}})\mu(\sigma_{\gamma_{i}}) - \int_{\Omega}f_{n}d\mu| \leq \bigvee_{i=1}^{\infty}a^{(n)}_{i\psi(i)}
    \end{equation}

    Therefore for $\gamma \geq \gamma_{n}$, using (\ref{eq5}),(\ref{eq6}) and (\ref{eq7}), the following inequalities are valid:

    \[|\sum_{i \in I_{\gamma}}f(\tau_{\gamma_{i}})\mu(\sigma_{\gamma_{i}}) - L| \leq |\sum_{i \in I_{\gamma}}f(\tau_{\gamma_{i}})\mu(\sigma_{\gamma_{i}})-\sum_{i \in I_{\gamma}}f_{n}(\tau_{\gamma_{i}})\mu(\sigma_{\gamma_{i}})|\]
    
    \[+ |\sum_{i \in I_{\gamma}}f_{n}(\tau_{\gamma_{i}})\mu(\sigma_{\gamma_{i}}) - \int_{\Omega}f_{n}d\mu| + |\int_{\Omega}f_{n}d\mu - L|\]

    \[ \leq |\sum_{i \in I_{\gamma}}(f(\tau_{\gamma_{i}})-f_{n}(\tau_{\gamma_{i}}))\mu(\sigma_{\gamma_{i}})| + \bigvee_{i=1}^{\infty}a_{i\psi(i)}+w_{n}.\]

Continuing from this last inequality and, using \ref{eq5} in the first term, we get that this last line is
    
    \[\leq \sum_{i \in I_{\gamma}}u_{n}\mu(\sigma_{\gamma_{i}})+\bigvee_{i=1}^{\infty}a^{(n)}_{i\psi(i)}+w_{n} = u_{n}\mu(\Omega)+ \bigvee_{i=1}^{\infty}a^{(n)}_{i\psi(i)}+w_{n}\]

    In resume, until now we have the following inequality:

    \begin{equation} \label{eq8}
        |\sum_{i \in I_{\gamma}}f(\tau_{\gamma_{i}})\mu(\sigma_{\gamma_{i}})-L| \leq u_{n}\mu(\Omega) + w_{n} + \bigvee_{i=1}^{\infty}a^{(n)}_{i\psi(i)}, \ \forall \ \gamma \geq \gamma_{n}.
    \end{equation}

    Now, consider the following equation defining the bounded triple sequence $\{c_{nij}\}_{nij}$ in $Z$:

    $$
c_{nij}=\begin{cases}
			w_{j}+u_{j}\mu(S), & \text{if $n=1$},\\
            a^{(n-1)}_{ij}, & \text{if $n \geq 2,$}
		 \end{cases}
$$

for $ i, j \in \mathbb{N}$. 

In this case, notice that, for each $n \in \mathbb{N}$ fixed, the bounded double sequence $\{c_{nij}\}_{ij}$ is a $(D)$-sequence of elements of $Z$ because, for each $i \in \mathbb{N}$, $\{c_{nij}\}_{j \in \mathbb{N}}$ is non-increasing and $c_{nij} \downarrow 0$, when $j \rightarrow \infty$.

Now, let $L^{'} = u_{1}\mu(\Omega) + w_{1} + \bigvee_{i,j=1}^{\infty}a^{(1)}_{ij}$. If $L^{'}=0$, then $u_{1}=0$ and, by (\ref{eq5}), $f = f_{k}$ $\forall k \in \mathbb{N}$. Assuming, then, that $L^{'} \in Z^{+}\setminus \{0\}$ exists, by Fremlin Lemma, there exists a double sequence $\{b_{ij}\}_{ij}$ in $Z$ such that $\{L^{'} \wedge b_{ij}\}_{ij}$ is a $(D)$-sequence and, $\forall \psi \in \mathbb{N}^{\mathbb{N}}$,

\begin{equation} \label{eq9}
    L^{'} \wedge (\sum_{r=1}^{n}\bigvee_{i=1}^{\infty}c_{ri\psi(i+r)}) \leq \bigvee_{j=1}^{\infty}(L^{'} \wedge b_{j\psi(j)}), \ \forall n \in \mathbb{N}.
\end{equation}

We now show, that, the $(D)$-sequence $\{a_{ij}\}_{ij}$ given by:

\[a_{ij} =  (L^{'} \wedge b_{ij}), \]

is the $(D)$-sequence for the existence of the Net Riemann integral of $f$.

In fact, given $\varphi \in \mathbb{N}^{\mathbb{N}}$, let $p = \min_{i \in \mathbb{N}}\varphi(i+1)$ and $\psi \in \mathbb{N}^{\mathbb{N}}$ the function defined by $\psi(j) = \varphi(j+p+1)$, $\forall j \in \mathbb{N}$. By (\ref{eq8}), for $\gamma \geq \gamma_{p}$, $\gamma_{p}=\gamma_{p}(\varphi)$, all in $\Gamma$, we have that:

\begin{equation} \label{penultimate}
    |\sum_{i \in I_{\gamma}}f(\tau_{\gamma_{i}})\mu(\sigma_{\gamma_{i}})-L| \leq u_{p}\mu(\Omega) + w_{p} + \bigvee_{i=1}^{\infty}a^{(p)}_{i\varphi(i+p+1)}.
\end{equation}

Also, we have, by the properties of the bilinear product:

\[\bigvee_{j=1}^{\infty}(u_{\varphi(j+1)}\mu(\Omega)+w_{\varphi(j+1)}) = u_{p}\mu(\Omega) + w_{p}.\]

Therefore, from (\ref{penultimate}), we may write that:

\[|\sum_{i \in I_{\gamma}}f(\tau_{\gamma_{i}})\mu(\sigma_{\gamma_{i}}) - L| \leq \bigvee_{j=1}^{\infty}(u_{\varphi(j+1)}\mu(\Omega)+w_{\varphi(j+1)})+\bigvee_{i=1}^{\infty}a^{(p)}_{i\varphi(i+p+1)},\]

that is,

\begin{equation} \label{eq10}
    |\sum_{i \in I_{\gamma}}f(\tau_{\gamma_{i}})\mu(\sigma_{\gamma_{i}}) - L| \leq \bigvee_{j=1}^{\infty}c_{ij\varphi(j+1)}+\bigvee_{i=1}^{\infty}c_{(p+1)i\varphi(i+p+1)} \leq \sum_{r=1}^{p+1} \bigvee_{i=1}^{\infty}c_{ri\phi(i+r)}.
\end{equation}

On the other hand, by (\ref{eq8}), and for $n=1$, we have that:

\begin{equation} \label{eq11}
    |\sum_{i \in I_{\gamma}}f(\tau_{\gamma_{i}})\mu(\sigma_{\gamma_{i}}) - L| \leq u_{1}\mu(S) + w_{1} + \bigvee_{i,j=1}^{\infty}a^{(1)}_{ij} = L^{'},
\end{equation}

whenever $\gamma \geq \gamma_{1}$, with, as above, $\gamma_{1} \in \Gamma$.

Now, taking $\gamma_{0}$ an upper bound of $\gamma_{1}$ and $\gamma_{p}$ in $\Gamma$ (which always exists, as $\Gamma$ is upwards directed), for $\gamma \geq \gamma_{0}$, we have that, by (\ref{eq9}), (\ref{eq10}) and (\ref{eq11}), 

\[|\sum_{i \in I_{\gamma}}f(\tau_{\gamma_{i}})\mu(\sigma_{\gamma_{i}}) - L| \leq L^{'} \wedge (\sum_{r=1}^{p+1}\bigvee_{i=1}^{\infty}c_{ri\varphi(i+r)}) \leq \bigvee_{j=1}^{\infty}(L^{'} \wedge b_{j\varphi(j)}) = \bigvee_{j=1}^{\infty}a_{j\varphi(j)},\]

which shows that, as indicated before, that $\{a_{ij}\}_{ij}$ is indeed the $(D)$-sequence for the existence of the Net Riemann integral of $f$, implying that $f$ is Net Riemann integrable and, by an appeal to Proposition \ref{equivalenceoand(D)convergence}:

\[\int_{\Omega}f_{n}d\mu \xrightarrow[]{o} \int_{\Omega}fd\mu,\]

This finishes the proof of the theorem. 

\end{proof}

\subsection{Null sets and a Uniform Convergence Theorem for the S*-partition Integral.} \label{convergenceoutsidenullset}

For this subsection, we will now present an uniform convergence for the $S^{*}$-partition integral outside of (to be defined) nulls sets of $\mu$. In general, this type of result will not be true for the general Net Riemann integral which contains, in its special cases, integrals which behaves quite badly with respect to null sets integration - see, for example, the Henstock-Kurzweil case in \cite{boccutovrabelovariecan}. Nevertheless, we also present such a theorem for the specific case of the Sion integral. 

We begin wit the following classical definition. 

\begin{definition} [Null Set and A.E. Properties]

Given a set function $\mu: \mathcal{H} \rightarrow Y$, we say that a set $N \in \mathcal{H}$ is $\mu$-null if $\mu(N) = 0$.

We then say that some property (based on $\Omega$) is valid $\mu-a.e.$ if  there is some null set $N \in \mathcal{H}$ such that the property is valid outside $N$, i.e in $\Omega \setminus N$.
\end{definition}

Before presenting the next convergence results, we prove three minor results.

\begin{proposition}
    Let $\mu: \mathcal{H} \rightarrow Y^{+}$ be a non-negative (finitely) additive set function. If $B \in \mathcal{H}$ is $\mu$-null and $A \subseteq B$, $A \in \mathcal{H}$, then $B$ is $\mu$-null
\end{proposition}

\begin{proof}
    In fact, we have that, by the non-negativity and additivity of $\mu$, 

    \[\mu(A) \leq \mu(B),\]

    which shows, as $\mu(B) = 0$ and $\mu(A) \geq 0$, that $\mu(A) = 0$. 
\end{proof}

Also, we shall need the next result for integration on null sets.

\begin{proposition}
    Let $\mu: \mathcal{H} \rightarrow Y^{+}$ (finitely) additive, $N$ a $\mu-null$ set and $f: \Omega \rightarrow X$ a function. Then, $f$ is $S^{*}$-partition integrable on $N$ and:

    \[\int_{N}^{S^{*}}fd\mu = 0\]
\end{proposition}

\begin{proof}
    Let $P \in \mathcal{P}(N)$ be any partition of $N$ in $\mathcal{H}$-measurable sets and $\delta$ an arbitrary choice function. By the proposition above, we have the relation

    \[\mu(\alpha) = 0, \ \forall \alpha \in P,\]

    and, as a consequence, we have:

    \[\sum_{\alpha \in P}f(\delta_{\alpha})\mu(\alpha) = 0.\]

    Therefore, considering an arbitrary $(D)$-sequence $\{a^{(n)}_{ij}\}_{ij} \in Z$ we have that, $\forall \ \psi \in \mathbb{N}^{\mathbb{N}}$, there exists $P \in \mathcal{P}(N)$ such that for each $P \geq P_{n}$ in $\mathcal{P}(N)$ the unconditionally convergent series of $f$ and $\mu$ exists in $P$ and:

    \[|\sum_{\alpha \in P}f_{n}(\delta_{\alpha})\mu(\alpha)| \leq \bigvee_{i=1}^{\infty}a^{(n)}_{i\psi(i)},\]

    which shows that $f$ is $S^{*}$-partition integrable on $N$ and its integral is $0$.
\end{proof}

As a corollary that we will use frequently below (without mention) is:

\begin{corollary}
    Let $\mu: \mathcal{H} \rightarrow Y^{+}$ (finitely) additive, $N$ a $\mu-null$ set and $f: \Omega \rightarrow X$ a $S^{*}$-partition integrable function on $\Omega$. Then, 

    \[\int_{\Omega}^{S^{*}}fd\mu = \int_{\Omega \setminus N}^{S^{*}}f\mu.\]
\end{corollary}

\begin{proof}
    As $f$ is integrable on $\Omega$, by (vi) of Theorem \ref{Properties oftheS*}, we have that it is integrable on $\Omega \setminus N$ and, as $f$ is integrable on $N$ by the last result:

    \[\int^{S^{*}}_{\Omega}fd\mu = \int^{S^{*}}_{N}fd\mu + \int^{S^{*}}_{\Omega \setminus N}fd\mu.\]

    As $\int_{N}^{S{*}}fd\mu = 0 $ by the last cited result, if follows that:

    \[\int^{S^{*}}_{\Omega}fd\mu = \int^{S^{*}}_{\Omega \setminus N}fd\mu\]
\end{proof}

Later, we will prove an stronger result concerning $S^{*}$-partition integration outside of a null set.

We now formulate the following modification of uniform convergence in order:

\begin{definition} [Almost Everywhere Uniform Convergence]

We say that a sequence of functions $f_{n}: \Omega \rightarrow X$ converges uniformly almost everywhere in order on $\Omega$ to a function $f:\Omega \rightarrow X$, written as $(u)-\lim_{n \rightarrow \infty}f_{n} = f \ \mu-\text{a.e.}$, if there exists a null set $N \in \mathcal{H}$ such that $f_{n}$ converges order uniformly to $f$ on $\Omega \setminus N$.
    
\end{definition}

In this case, we have the following modifications of Proposition \ref{monteirofernandez1} and Theorem \ref{uniformconvergencetheoremriemann}:

\begin{proposition} \label{monteirofernandez2}
    Let $X$ be a Dedekind $\sigma$-complete Riesz space that is weakly $\sigma$-distributive and $\{f_{n}\}_{n \in \mathbb{N}}$ a sequence of $S^{*}$-partition integrable functions on $\Omega$. Suppose that $\mu: \mathcal{H} \rightarrow Y^{+}$ is $\sigma$-additive. If $\{f_{n}\}_{n \in \mathbb{N}}$ is Cauchy uniformly almost everywhere in order, then the following limit exists:

    \[\limsup_{n \rightarrow \infty} \int^{S^{*}}_{\Omega \setminus N} f_{n} d\mu = \bigwedge_{m=1}^{\infty} \bigvee_{i=m}^{\infty} \int^{S^{*}}_{\Omega \setminus N}f_{i}d\mu,\]

    and,

    \[\int^{S^{*}}_{\Omega}f_{n}d\mu \xrightarrow[]{o} \bigwedge_{m=1}^{\infty}\bigvee_{i=m}^{\infty} \int^{S^{*}}_{\Omega \setminus N}f_{i}d\mu, \]

    where $N$ is the null set in the  uniformly almost everywhere convergence in order related to $f_{n},f$. 
\end{proposition}

\begin{theorem} [Almost Everywhere Uniform Converge Theorem for the $S^{*}$-partition Integral]\label{uniformconvergencetheorem2}
    Let $X$ be a Dedekind complete Riesz space that is weakly $\sigma$-distributive and $\{f_{n}\}_{n \in \mathbb{N}}$ a sequence of non-negative $S^{*}$-partition integrable functions in $S$. Suppose that $\mu: \mathcal{H} \rightarrow Y^{+}$ is $\sigma$-additive. If $f_{n}$ converges uniformly almost everywhere in order to $f$, a non-negative function, then $f$ is $S^{*}$-partition integrable on $\Omega \setminus N$ and:

\[\int^{S^{*}}_{\Omega} f_{n}d\mu \xrightarrow[]{o} \int^{S^{*}}_{\Omega \setminus N}fd\mu,\]

where $N$ is the null set in the  uniformly almost everywhere convergence in order related to $f_{n},f$. 
\end{theorem}

The proof of Proposition \ref{monteirofernandez2} is the same for the original one, and requires only the modification of including the null set, which doesn't affect the proof. Therefore, we omit it. 

For Theorem \ref{uniformconvergencetheorem2}, we need some small, but important modifications of Theorem \ref{uniformconvergencetheoremriemann} relating to the nature of the $S^{*}$-partition integral. With the danger of being a little bit repetitive, we present the proof and the necessary modifications.

\begin{proof}
    First, we prove the infinite summations that occurs in the definition of the $S^{*}$-partition integral of $f$, on $\Omega \setminus N$ exist. Throughout   the proof, consider $\delta$ an arbitrary choice function.
 
    As $f_{n}$ converges to $f$  uniformly almost everywhere in order, there exists a non-increasing sequence $\{u_{n}\}_{n \in \mathbb{N}}$ such that $u_{n} \downarrow 0$ and:

    \[|f_{n}(\omega)-f(\omega)| \leq u_{n}, \ \forall \omega \in \Omega \setminus N \ \text{and} \ n \in \mathbb{N}.\]

    Fix an arbitrary $n \in \mathbb{N}$ such that the inequality above is valid, and let $P \in \mathcal{P}(\Omega \setminus N)$ a partition of in the definition of the $S^{*}$--partition integral of $f_{n}$ (as $f_{n}$ is integrable on $\Omega$, it is on $\Omega \setminus N$ by the results above) such that the following (increasing limit) exists:

    \[\sum_{\alpha \in P}f_{n}(\delta_{\alpha})\mu(\alpha),\]

    and consider an arbitrary, but fixed, $J \in \mathcal{F}(\mathbb{N})$. Then, we have that, for these choices:

    \[|\sum_{j \in J, \alpha \in P}f(\delta_{\alpha_{j}})\mu(\alpha_{j})| \leq \sum_{j \in J, \alpha \in P}|(f(\delta_{\alpha_{j}})-f_{n}(\delta_{\alpha_{j}}))\mu(\alpha_{j})| + \sum_{j \in J, \alpha \in P}f_{n}(\delta_{\alpha_{j}})\mu(\alpha_{j}),\]

    but, from the first inequality above, 

    \[\sum_{j \in J, \alpha \in P}|(f(\delta_{\alpha_{j}})-f_{n}(\delta_{\alpha_{j}}))\mu(\alpha_{j})| \leq u_{n}\mu(\Omega \setminus N),\]

    and,

    \[\sum_{j \in J, \alpha \in P}f_{n}(\delta_{\alpha_{j}})\mu(\alpha_{j}) \leq \sum_{\alpha \in P}f_{n}(\delta_{\alpha})\mu(\alpha),\]

    which implies that, by the previous inequalities:

    \[|\sum_{j \in J, \alpha \in P}f(\delta_{\alpha_{j}})\mu(\alpha_{j})| \leq u_{n}\mu(\Omega \setminus N) + \sum_{\alpha \in P}f_{n}(\delta_{\alpha})\mu(\alpha).\]

    Therefore, the net $\{\sum_{j \in J, \alpha \in P}f(\delta_{\alpha_{j}})\mu(\alpha_{j})\}_{J \in \mathcal{F}(\mathbb{N})}$ is bounded for the given $P$, which shows, as $X$ is Dedekind complete, that the order limit of this net, which is given by:

    \[\sum_{\alpha \in P} f(\delta_{\alpha})\mu(\alpha),\]

    is a well defined quantity (i.e, the order limit exists). By $(ii)$ in the conventions above, this limit can also be taken in the sense of $(D)$-convergence. The restriction put on $P \in \mathcal{P}(\Omega \setminus N)$ is not so restrictive, as we will only need partitions for which the $S^{*}$-partition integrals of $f_{n}$ are well defined, thus fulfilling the argument above. 

    We repeat the first part above with different notation to avoid confusion: as $f_{n}$ converges to $f$ uniformly almost everywhere in order, there exists a non-increasing sequence $\{u_{k}\}_{k \in \mathbb{N}}$ such that $u_{k} \downarrow 0$ and:

    \begin{equation} \label{eq5-2}
        |f_{k}(x)-f(x)| \leq u_{k}, \ \forall x \in \Omega \setminus N \ \text{and} \ k \in \mathbb{N},
    \end{equation}

    which implies that $\{f_{n}\}_{n \in \mathbb{N}}$ is Cauchy uniformly.

    Therefore, by Proposition \ref{monteirofernandez2}, there exists $L \in Z$ such that the following order limit is valid:

    \[\int^{S^{*}}_{\Omega}f_{n}d\mu = \int^{S^{*}}_{\Omega \setminus N}f_{n}d\mu \xrightarrow[]{o} L.\]

    In other words, there exists a non-decreasing sequence $\{w_{n}\}_{n \in \mathbb{N}} \in X$ such that $w_{n} \downarrow 0$ and:

    \begin{equation} \label{eq6-2}
        |\int^{S^{*}}_{\Omega}f_{n}d\mu - L| \leq w_{n}, \ \forall n \in \mathbb{N}.
    \end{equation}

    As each $f_{n}$ is $S^{*}$-partition integrable in on $\Omega \setminus N$, for each $n \in \mathbb{N}$, there exists a $(D)$-sequence $\{a^{(n)}_{ij}\}_{ij} \in Z$ such that, $\forall \ \psi \in \mathbb{N}^{\mathbb{N}}$, there exists $P_{n} = P_{n}(\psi) \in \mathcal{P}(\Omega \setminus N)$ such that for each $P \geq P_{n}$ in $\mathcal{P}(\Omega \setminus N)$ the unconditionally convergent series of $f_{n}$ and $\mu$ exists in $P$ and:

    \begin{equation} \label{eq7-2}
        |\sum_{\alpha \in P}f_{n}(\delta_{\alpha})\mu(\alpha) - \int^{S^{*}}_{\Omega}f_{n}d\mu| \leq \bigvee_{i=1}^{\infty}a^{(n)}_{i\psi(i)}
    \end{equation}

    As $\sum_{\alpha \in P}f(\delta_{\alpha})\mu(\alpha)$ exists, by the argument in the first part of the present proof, for $P \geq P_{n}$, using (\ref{eq5-2}),(\ref{eq6-2}) and (\ref{eq7-2}), the following inequalities are valid:

    \[|\sum_{\alpha \in P}f(\delta_{\alpha})\mu(\alpha) - L| \leq |\sum_{\alpha \in P}f(\delta_{\alpha})\mu(\alpha)-\sum_{\alpha \in P}f_{n}(\delta_{\alpha})\mu(\alpha)|\]
    
    \[+ |\sum_{\alpha \in P}f_{n}(\delta_{\alpha})\mu(\alpha) - \int^{S^{*}}_{\Omega}f_{n}d\mu| + |\int^{S^{*}}_{\Omega}f_{n}d\mu - L|\]

    \begin{equation} \label{eqlastmain}
      \leq |\sum_{\alpha \in P}(f(\delta_{\alpha})-f_{n}(\delta_{\alpha}))\mu(\alpha)| + \bigvee_{i=1}^{\infty}a_{i\phi(i)}+w_{n}.
    \end{equation}

    Now, for each $P$ as above, the net $\{\sum_{j \in J, \alpha \in P}|(f_{n}(\delta_{\alpha_{j}})-f(\delta_{\alpha_{j}}))\mu(\alpha_{j})|\}_{J \in \mathcal{F}(\mathbb{N})}$ is order and $(D)$- convergent to its supremum, as it is increasing and bounded above by $u_{n}\mu(\Omega \setminus N)$ for each $n$ and $X$ is Dedekind complete. Therefore,

    \[\sum_{j \in J, \alpha \in P}|(f_{n}(\delta_{\alpha_{j}})-f(\delta_{\alpha_{j}}))\mu(\alpha_{j})| \leq \sum_{\alpha \in P}|(f_{n}(\delta_{\alpha})-f(\delta_{\alpha}))\mu(\alpha_{j})|,\]

    for each $J \in \mathcal{F}(\mathbb{N})$. Hence, by Proposition 1.20 of \cite{popovriesz} and using that $|\cdot|$ is order continuous:

    \[|\sum_{\alpha \in P}(f_{n}(\delta_{\alpha})-f(\delta_{\alpha}))\mu(\alpha_{j})| \leq \sum_{\alpha \in P}|(f_{n}(\delta_{\alpha})-f(\delta_{\alpha}))\mu(\alpha_{j})|,\]

    Therefore, from (\ref{eqlastmain}):

    \[\leq \sum_{\alpha \in P}|(f_{n}(\delta_{\alpha})-f(\delta_{\alpha}))\mu(\alpha)| + \bigvee_{i=1}^{\infty}a^{(n)}_{i\psi(i)}+w_{n} \]
    
    \[\leq \sum_{\alpha \in P}u_{n}\mu(\alpha)+\bigvee_{i=1}^{\infty}a^{(n)}_{i\psi(i)}+w_{n} = u_{n}\mu(\Omega \setminus N)+ \bigvee_{i=1}^{\infty}a^{(n)}_{i\psi(i)}+w_{n}\]

    In resume, until now we have the following inequality:

    \begin{equation} \label{eq8-2}
        |\sum_{\alpha \in P}f(\delta_{\alpha})\mu(\alpha)-L| \leq u_{n}\mu(S) + w_{n} + \bigvee_{i=1}^{\infty}a^{(n)}_{i\psi(i)}, \ \forall \ P \geq P_{n}.
    \end{equation}

    Now, consider the following equation defining a triple sequence $\{c_{nij}\}_{nij}$ in $Z$:

    $$
c_{nij}=\begin{cases}
			w_{j}+u_{j}\mu(\Omega \setminus N), & \text{if $n=1$},\\
            a^{(n-1)}_{ij}, & \text{if $n \geq 2,$}
		 \end{cases}
$$

for $ i, j \in \mathbb{N}$. 

In this case, notice that, for each $n \in \mathbb{N}$ fixed, the (bounded) double sequence $\{c_{nij}\}_{ij}$ is a $(D)$-sequence of elements of $Z$ because, for each $i \in \mathbb{N}$, $\{c_{nij}\}_{j \in \mathbb{N}}$ is non-increasing and $c_{nij} \downarrow 0$, when $j \rightarrow \infty$.

Now, let $L^{'} = u_{1}\mu(\Omega \setminus N) + w_{1} + \bigvee_{i,j=1}^{\infty}a^{(1)}_{ij}$. If $L^{'}=0$, then $u_{1}=0$ and, by (\ref{eq5-2}), $f = f_{k}$ $\forall k \in \mathbb{N}$ outside of the null set $N$, which finishes the proof. Assuming, then, that $L^{'} \in Z^{+}\setminus \{0\}$ exists, by  Fremlin Lemma, there exists a double sequence $\{b_{ij}\}_{ij}$ in $X$ such that $\{L^{'} \wedge b_{ij}\}_{ij}$ is a $(D)$-sequence and, $\forall \psi \in \mathbb{N}^{\mathbb{N}}$, 

\begin{equation} \label{eq9-2}
    L^{'} \wedge (\sum_{r=1}^{n}\bigvee_{i=1}^{\infty}c_{ri\psi(i+r)}) \leq \bigvee_{j=1}^{\infty}(L^{'} \wedge b_{j}\psi(j)), \ \forall n \in \mathbb{N}.
\end{equation}

We now show, that, the $(D)$-sequence $\{a_{ij}\}_{ij}$ given by:

\[a_{ij} =  (L^{'} \wedge b_{ij}), \]

is the $(D)$-sequence for the existence of the $S^{*}$-partition integral of $f$ on $\Omega \setminus N$.

In fact, given $\varphi \in \mathbb{N}^{\mathbb{N}}$, let $p = \min_{i \in \mathbb{N}}\varphi(i+1)$ and $\psi \in \mathbb{N}^{\mathbb{N}}$ the function defined by $\psi(j) = \varphi(j+p+1)$, $\forall j \in \mathbb{N}$. By (\ref{eq8-2}), for $P \geq P_{p}$, $P_{p}=P_{p}(\varphi)$, all in $\mathcal{P}(\Omega \setminus N)$, we have that:

\begin{equation} \label{penultimate2}
    |\sum_{\alpha \in P}f(\delta_{\alpha})-L| \leq u_{p}\mu(\Omega \setminus N) + w_{p} + \bigvee_{i=1}^{\infty}a^{(p)}_{i\varphi(i+p+1)}.
\end{equation}

Also, we have, by the properties of the bilinear product:

\[\bigvee_{j=1}^{\infty}(u_{\varphi(j+1)}\mu(\Omega \setminus N)+w_{\varphi(j+1)}) = u_{p}\mu(\Omega \setminus N) + w_{p}.\]

Therefore, from (\ref{penultimate2}), we may write that:

\[|\sum_{\alpha \in P}f(\delta_{\alpha})\mu(\alpha) - L| \leq \bigvee_{j=1}^{\infty}(u_{\varphi(j+1)}\mu(\Omega \setminus N)+w_{\varphi(j+1)})+\bigvee_{i=1}^{\infty}a^{(p)}_{i\varphi(i+p+1)},\]

that is,

\begin{equation} \label{eq10-2}
    |\sum_{\alpha \in P}f(\delta_{\alpha})\mu(\alpha) - L| \leq \bigvee_{j=1}^{\infty}c_{ij\varphi(j+1)}+\bigvee_{i=1}^{\infty}c_{(p+1)i\varphi(i+p+1)} \leq \sum_{r=1}^{p+1} \bigvee_{i=1}^{\infty}c_{ri\phi(i+r)}.
\end{equation}

On the other hand, by (\ref{eq8-2}), and for $n=1$, we have that:

\begin{equation} \label{eq11-2}
    |\sum_{\alpha \in P}f(\delta_{\alpha})\mu(\alpha) - L| \leq u_{1}\mu(\Omega \setminus N) + w_{1} + \bigvee_{i,j=1}^{\infty}a^{(1)}_{ij} = L^{'},
\end{equation}

whenever $P \geq P_{1}$, with, as above, $P_{1} \in \mathcal{P}(\Omega \setminus N)$.

Now, taking $P_{0} = P_{1} \wedge_{\mathcal{P}(\Omega \setminus N)} P_{p}$, the refinement of $P_{1}$ and $P_{p}$, for $P \geq P_{0}$, we have that, by (\ref{eq9-2}), (\ref{eq10-2}) and (\ref{eq11-2}), 

\[|\sum_{\alpha \in P}f(\delta_{\alpha})\mu(\alpha) - L| \leq L^{'} \wedge (\sum_{r=1}^{p+1}\bigvee_{i=1}^{\infty}c_{ri\varphi(i+r)}) \leq \bigvee_{j=1}^{\infty}(L^{'} \wedge b_{j\varphi(j)}) = \bigvee_{j=1}^{\infty}a_{j\varphi(j)},\]

which shows that, as indicated before, that $\{a_{ij}\}_{ij}$ is indeed the $(D)$-sequence for the existence of the $S^{*}$partition integral of $f$ on $\Omega \setminus N$, implying that $f$ is $S^{*}$-partition integrable on $\Omega \setminus N$ and by an appeal to Proposition \ref{equivalenceoand(D)convergence}:

\[\int^{S^{*}}_{\Omega}f_{n}d\mu \rightarrow \int^{S^{*}}_{\Omega \setminus N}fd\mu,\]

This finishes the proof.

\end{proof}

We now prove a result to connect the convergence theorems related to the $S^{*}$-partition integral, outside of null sets, to the whole $\Omega$.

\begin{theorem} \label{integrabilityonnullsets}
    Suppose that $Z$ is Dedekind complete and let $f$ be a non-negative function $S^{*}$-partition integrable on $\Omega \setminus N$, where $N$ is a $\mu$-null set, $\mu: \mathcal{H} \rightarrow Y$ non-negative and additive. Then, $f$ is $S^{*}$-partition integrable on $\Omega$ and:

    \[\int^{S^{*}}_{\Omega}fd\mu = \int^{S^{*}}_{\Omega \setminus N}fd\mu\]
\end{theorem}

\begin{proof}
    By hypothesis, following the notation in the theorem, for all  choice functions $\delta$, there exists an $s \in Z$ and a $(D)$-sequence $\{a_{ij}\}$ such that, for every $\varphi \in \mathbb{N}^{\mathbb{N}}$, there exists $P_{0} \in \mathcal{P}(\Omega \setminus N)$ such that:

    \[|\sum_{\alpha \in P}f(\delta_{\alpha})\mu(\alpha) - s| \leq \bigvee_{i=1}^{\infty}a_{i\varphi(i)},\]

    for every $P \geq P_{0}$, $P \in \mathcal{P}(\Omega \setminus N)$.

    Consider now $P_{0}^{'} = P_{0} \cup \{N\}$, which is an element of $\mathcal{P}(\Omega)$ by construction. Take now any $P^{'} \geq P_{0}{'}$ in $\mathcal{P}(\Omega)$. Define the following class of sets in relation to $P^{'}$:

    \[A_{P^{'}} = \{\alpha \in P^{'}: \alpha \subseteq \beta, \text{for some} \ \beta \in P_{0}\},\]

    which satisfies, by definition:

    \[\bigcup_{\alpha \in A_{P^{'}}}\alpha = \bigcup_{\beta \in P_{0}}(\bigcup_{\alpha \in P^{'}, \alpha \subseteq \beta} \alpha) = \bigcup_{\beta \in P_{0}}\beta = \Omega \setminus N,\]

    which gives that $A_{P^{'}} \in \mathcal{P}(\Omega \setminus N)$ and $A_{P^{'}} \geq P_{0}$.

    As $N$ is a $\mu$-null set, we have that, for any $J \in \mathcal{F}(\mathbb{N})$,

    \[\sum_{j \in J,\alpha \in P^{'}}f(\delta_{\alpha_{j}})\mu(\alpha_{j}) = \sum_{j \in J,\alpha \in A_{P^{'}}}f(\delta_{\alpha_{j}})\mu(\alpha_{j}),\]

    but $A_{P^{'}} \in \mathcal{P}(\Omega \setminus N)$ and $A_{P^{'}} \geq P_{0}$, therefore, by the integrability condition on $f$ enounced at the first paragraph of the proof, we get that the series

    \[\sum_{\alpha \in A_{P^{'}}}f(\delta_{\alpha})\mu(\alpha),\]

    exists and as $f$ and $\mu$ are non-negative:

    \[\sum_{j \in J,\alpha \in P^{'}}f(\delta_{\alpha_{j}})\mu(\alpha_{j}) \leq \sum_{\alpha \in A_{P^{'}}}f(\delta_{\alpha})\mu(\alpha),\]

    for every $J \in \mathcal{F}(\mathbb{N})$. As $Z$ is Dedekind complete,

    \[\sum_{\alpha \in P^{'}}f(\delta_{\alpha})\mu(\alpha),\]

    exists for each $P^{'}$ given above. By the same reasoning above, we have not only that the series index on $P^{'}$ exists, but also:

    \[\sum_{\alpha \in P^{'}}f(\delta_{\alpha})\mu(\alpha) = \sum_{\alpha \in A_{P^{'}}}f(\delta_{\alpha})\mu(\alpha),\]

    and using the fact that $A_{P^{'}} \geq P_{0}$, we get by the relation involving $s$, 

    \[|\sum_{\alpha \in P^{'}}f(\delta_{\alpha})\mu(\alpha) - s| \leq \bigvee_{i=1}^{\infty}a_{i\varphi(i)}, \]

    which shows that $f$ is integrable on $\Omega$ and:

    \[s = \int^{S^{*}}_{\Omega \setminus N}fd\mu = \int^{S^{*}}_{\Omega}fd\mu,\]

    which finishes the proof.

\end{proof}

Combining Theorems \ref{integrabilityonnullsets} and \ref{uniformconvergencetheorem2} we get the final result:

\begin{theorem} [Almost Everywhere Uniform Converge Theorem for the $S^{*}$-partition Integral]\label{uniformconvergencetheorem2null}
    Let $X$ be a Dedekind complete Riesz space that is weakly $\sigma$-distributive and $\{f_{n}\}_{n \in \mathbb{N}}$ a sequence of non-negative $S^{*}$-partition integrable functions on $\Omega$. Suppose that $\mu: \mathcal{H} \rightarrow Y^{+}$ is $\sigma$-additive. If $f_{n}$ converges uniformly almost everywhere in order to $f$, a non-negative function, then $f$ is $S^{*}$-partition integrable on $\Omega$ and:

\[\int^{S^{*}}_{\Omega} f_{n}d\mu \xrightarrow[]{o} \int^{S^{*}}_{\Omega}fd\mu,\]

\end{theorem}

A almost identical proof (in fact, simpler one) can be made for the Sion integral, when considered with respect to partitions and not coverings, of all the results above for the $S^{*}$-partition integral related to null sets, which leads us to enounce:

\begin{theorem} [Almost Everywhere Uniform Converge Theorem for the Sion Integral]\label{uniformconvergencetheorem2sion}
    Let $X$ be a Dedekind complete Riesz space that is weakly $\sigma$-distributive, $\mu: \mathcal{H} \rightarrow Y^{+}$ a regular integrator, and $\{f_{n}\}_{n \in \mathbb{N}}$ a sequence of Sion (partition) integrable functions on $\Omega$. If $f_{n}$ converges uniformly almost everywhere in order to $f$ then $f$ is Sion (partition) integrable on $\Omega$ and:

\[\int^{Si}_{\Omega} f_{n}d\mu \xrightarrow[]{o} \int^{Si}_{\Omega}fd\mu,\]

\end{theorem}

This finishes the basic uniform convergence theorems for the net Riemann and $S^{*}$-partition integral. We now obtain some results related to the Henstock lemma and its applications in the form of monotone and dominated convergence theorems for the cited integrals.

\subsection{Henstock's Lemma and the Monotone and Dominated Convergence Theorems.}

The monotone and convergence theorems for Riemann type integrals provide a striking difference from the classical Lebesgue case, as the usual results of the latter theory are not present in full generality in the former Riemmanian case. To obtain some useful results in this direction, we follow the strategy originally due to \cite{henstockfirstbook}, and developed by \cite{mcshaneriemann}, and prove a result now cited and Henstock Lemma to obtain such theorems. 

The classical proof of the Henstock lemma in references such as \cite{haluskacompletevectorlattices}, \cite{henstockgreek}, \cite{mchsaneorder}, \cite{integralmeasureandordering}, and others, is generally divided in the topological and order-based cases. We will provide a proof related to the already introduced case of lattice normed spaces, which unites the topological and Riesz space cases in one proof, and gives a more general view of this result. Also, as expected, to prove this result we need (several) restrictions on the types of Szabó-Száz-Fleischer nets that are given in the integration procedure, as in subset integration of Theorem \ref{subsetintegration}:

\begin{theorem} [Henstock Lemma] \label{henstocklemma}

Let $f: \Omega \rightarrow X$ be an net Riemann integrable function with respect to a set function $\mu: \mathcal{H} \rightarrow Y$ such that:

\begin{enumerate} [(i)]
    \item $f$ is subset integrable: that is, if $f$ is net Riemann integrable on $B \in \mathcal{H}$, then it is net Riemann integrable on each $A \subseteq B$, $A \in \mathcal{H}$;
    \item The net Riemann integral of $f$ is  additive (i.e, if the sets are not disjoint)\footnote{Notice that this not a big restriction, as the disjoint case is covered, and the Henstock-Kurzweil or abstract gauge integration are defined in terms of partitions of sets with measure $0$ intersection.} on the sets (belonging to $\mathcal{H}$) which form (the second component of) the Szabó-Száz-Fleischer nets corresponding to each $\gamma \in \Gamma(B)$, $B \in \mathcal{H}$ on which it is net Riemann integrable;
    \item The second component (i.e, the sets on which we sum $\mu$) of the Szabó-Száz-Fleishcer nets from $\Gamma(B)$ consists of sets (from $\mathcal{H}$) such that their union is $B$, for all $B \in \mathcal{H}$ on which $f$ is net Riemann integrable;
    \item Let $B \in \mathcal{H}$ be such that $f$ is net Riemann integrable on it and let $\sigma = \{B_{j}\}_{i=1}^{n}$ be a finite collection of sets corresponding to a $\gamma \in \Gamma(B)$. Then, for each finite sub-collection $\{B_{j}\}_{j \in J}$ of $\sigma$, given a $\gamma^{0} \in \Gamma(\bigcup_{j \in J}B_{j})$, there exists $\gamma^{*} \in \Gamma(\bigcup_{j \in J}B_{j})$ such that $\gamma^{*} \geq \gamma^{0}$ and there exists $\gamma^{(B)} \in \Gamma(B)$ of the form:

    \[\gamma^{(B)} \rightarrow (\tau^{\gamma^{*}} \cup \tau^{\cup_{k \in \{1, \cdots,n\}\setminus J}B_{k}}, \sigma^{\gamma^{*}} \cup \bigcup_{k \in \{1, \cdots,n\}\setminus J}\{B_{k}\}),\]
    
    where $\tau^{\cup_{k \in \{1, \cdots,n\}\setminus J}B_{k}}$ is the collection of points in the Szabó-Száz-Fleischer net of $\gamma$ corresponding to the sets $\bigcup_{k \in \{1, \cdots,n\}\setminus J}B_{k}$, and such that $\gamma^{B} \geq \gamma$.
    
\end{enumerate}

Let also $Z$ be a lattice normed space with lattice norm $||\cdot||: Z \rightarrow \mathcal{Z}$ equipped with $(D)$-convergence and $\mathcal{Z}^{+}$ weakly-$\sigma$-distributive. Let $B \in \mathcal{H}$ be a fixed set such that $f$ is net Riemann integrable on $B$. Then, given $\{a_{ij}\}$ a regulator in $\mathcal{Z}$ such that, for every $\varphi \in \mathbb{N}^{\mathbb{N}}$, there exists  $\gamma^{0,B} \in \Gamma(B)$ such that:

\[||\int_{B}fd\mu - \sum_{i \in I_{\gamma}}f(\tau_{\gamma_{i}})\mu(\sigma_{\gamma_{i}}) || \leq \bigvee_{i=1}^{\infty}a_{i\varphi(i)},\]

for each $\gamma \geq \gamma^{0,B}$, where $\gamma^{0,B} = (\tau^{0,B}, \sigma^{0,B} = \bigcup_{k=1}^{n}B_{k})$, and for every finite collection $K = \{1, \cdots, n\} \setminus J$, where $J$ if a finite (proper) subset of $\{1, \cdots, n\}$,

\[||\sum_{k \in K}\int_{B_{k}}fd\mu - \sum_{k \in K}f(\tau^{0,B_{k}})\mu(B_{k})|| \leq \bigvee_{i=1}^{\infty}a_{i\varphi(i)},\]

denoting by $\{\tau^{0,B_{k}}\}_{k \in K}$ the collection of points in the Szabó-Száz-Fleischer net of $\gamma^{0,B}$ associated to each set $B_{k}$.
    
\end{theorem}

\begin{proof}

We fix the notation given for the net Riemann integrability of $f$ on $B$ for this proof. Then, by hypothesis $(i)$, we have that $f$ is integrable on $\bigcup_{j \in J}B_{j}$ and therefore we can find $\{b_{ij}\}$ a regulator in $\mathcal{Z}$ such that, for every $\psi \in \mathbb{N}^{\mathbb{N}}$, there exists  $\gamma^{0} \in \Gamma(\bigcup_{j \in J}B_{j})$ such that:

\begin{equation} \label{eq1henstock}
    ||\int_{\bigcup_{j \in J}B_{j}}fd\mu - \sum_{i \in I_{\alpha}}f(\tau_{\alpha_{i}})\mu(\sigma_{\alpha_{i}}) || \leq \bigvee_{i=1}^{\infty}b_{i\psi(i)},
\end{equation}

for each $\alpha \geq \gamma^{0}$.

Now, by hypothesis $(iv)$, there exists $\gamma^{*} \in \Gamma(\bigcup_{j \in J}B_{j})$ such that $\gamma^{*} \geq \gamma^{0}$ and there exists $\gamma^{(B)} \in \Gamma(B)$ of the form:

    \[\gamma^{(B)} \rightarrow (\tau^{\gamma^{*}} \cup \tau^{\cup_{k \in \{1, \cdots,n\}\setminus J}B_{k}}, \sigma^{\gamma^{*}} \cup \bigcup_{k \in \{1, \cdots,n\}\setminus J}\{B_{k}\}),\]

such that $\gamma^{B} \geq \gamma$. For this $\gamma^{*}$ we have, by (\ref{eq1henstock}),

\[||\int_{\bigcup_{j \in J}B_{j}}fd\mu - \sum_{i \in I_{\gamma^{*}}}f(\tau_{\gamma^{*}_{i}})\mu(\sigma_{\gamma^{*}_{i}}) || \leq \bigvee_{i=1}^{\infty}b_{i\psi(i)},\]

and, by definition of $\gamma^{(B)}$,

\[||\int_{B}fd\mu - \sum_{i \in I_{\gamma^{(B)}}}f(\tau_{\gamma^{(B)}_{i}})\mu(\sigma_{\gamma^{(B)}_{i}}) || \leq \bigvee_{i=1}^{\infty}a_{i\varphi(i)}.\]

Now, as a consequence of hypotheses $(ii)$ and $(iii)$,

\[\int_{\bigcup_{j \in J}B_{j}}fd\mu + \int_{\bigcup_{k \in K}B_{k}}fd\mu = \sum_{j \in J} \int_{B_{j}}fd\mu + \sum_{k \in K} \int_{B_{k}}fd\mu = \int_{B}fd\mu.\]

Thus, from these results, we have:

\[||\sum_{k \in K}\int_{B_{k}}fd\mu - \sum_{k \in K}f(\tau^{0,B_{k}})\mu(,B_{k})|| \]

\[ = ||\int_{B}fd\mu - \int_{\bigcup_{j \in J}B_{j}}fd\mu + \sum_{i \in I_{\gamma^{(B)}}}f(\tau_{\gamma^{(B)}_{i}})\mu(\sigma_{\gamma^{(B)}_{i}}) - \sum_{i \in I_{\gamma^{*}}}f(\tau_{\gamma^{*}_{i}})\mu(\sigma_{\gamma^{*}_{i}}) || \]

\[\leq   \bigvee_{i=1}^{\infty}a_{i\varphi(i)} +  \bigvee_{i=1}^{\infty}b_{i\psi(i)}.\]

And using the weakly-$\sigma$-distributivity of $\mathcal{Z}$, from this last inequality we get:

\[||\sum_{k \in K}\int_{B_{k}}fd\mu - \sum_{k \in K}f(\tau^{0,B_{k}})\mu(,B_{k})|| - \bigvee_{i=1}^{\infty}a_{i\varphi(i)} \leq 0,\]

which implies that:

\[||\sum_{k \in K}\int_{B_{k}}fd\mu - \sum_{k \in K}f(\tau^{0,B_{k}})\mu(,B_{k})|| \leq \bigvee_{i=1}^{\infty}a_{i\varphi(i)} ,\]

giving the result.
    
\end{proof}

We now give some comments related to this last result, for which we place special emphasis on $3$.

\begin{enumerate}
    \item The hypotheses $(i) - (iv)$ are not so general as of subset integration, but are satisfied by most of the integrals that consists of nets of finite tagged partitions, such as the Henstock-Kurzweil (or Száz abstract version of it), Kolmogorov $S$ and other related procedures listed above (a general one being \cite{mcshaneriemann}). Also, it contains most of the called Henstock lemmas in the literature, an exception being the result of \cite{boccutoskvortsovabstract}\footnote{It would be interesting to know if the proof can be given for this case.}.
    \item It is not difficult to generalize the last proof for a space, say $Z$, endowed with a collection of lattice (valued) norms $||\cdot||_{l}: Z \rightarrow \mathcal{Z}_{l}$, with $\{Z_{l}\}_{l \in L}$ a possibly uncountable collection of Riesz spaces equipped with $(D)$-convergence and being weakly-$\sigma$-distributive, the "lattice norm convergence" in this case being convergence defined as $(D)$-convergence in each $\mathcal{Z}_{l}$. In this sense, the result above, or its straightforward generalization, contains not only the Henstock lemma for Riesz and topological $F$-normed spaces, but also for topological groups by the uniformization theorem of \cite{bourbaki2013general}.
    \item Another set of hypothesis may be seen to be sufficient to obtain this proof. For example, with the same notation as the theorem, consider the alternative of $(iv)$, which we change $B$ by $\Omega$ for simplicity:

    \begin{enumerate} 
        \item For $F = \Omega \setminus \bigcup_{j \in J}B_{j}$, suppose that given $\gamma^{(F)}_{0} \in \Gamma(F)$ in the net Riemann integrability definition of $f$ on $F$ and $\gamma^{(\Omega)}_{0} \in \Gamma(\Omega)$ in the net Riemann integrability of $f$ on $\Omega$, there exists $\gamma^{(F)}_{00} \in \Gamma(F)$ such that there exists $\gamma^{(\Omega)}_{00} \in \Gamma(\Omega)$ of the form:

            \[\gamma^{(\Omega)}_{00} \rightarrow (\tau^{\gamma^{(F)}_{00}} \cup \tau^{\cup_{j \in J}B_{j}}, \sigma^{\gamma^{(F)}_{00}} \cup \bigcup_{k \in \{1, \cdots,n\}\setminus J}\{B_{k}\}),\]

        such that $\gamma^{(\Omega)}_{00} \geq \gamma^{(\Omega)}_{0}$.    
    \end{enumerate}

    Then, supposing only finite additivity of the integral in place of $(ii)$ above (which we may call "complete additivity"), we obtain the same result by taking complements in the proof above with respect to the fixed collection given in the statement of the Henstock lemma. 

\end{enumerate}

This last result also motivates the next definition to summarize all related conditions for its validity:

\begin{definition} [Henstock Property]

We say that an integration procedure derived from the net Riemann integral (i.e, the Szabó-Száz-Fleischer) satisfies the Henstock property if it satisfies the hypotheses, or the alternatives in $(3)$ above\footnote{The proofs of the results below are the same for each set of hypothesis.},  of the Henstock Lemma, i.e  Theorem \ref{henstocklemma}.
    
\end{definition}

To combine the results and definitions in this section into a monotone and dominated convergence theorem, we need some more results that are modeled (statements and proofs)\footnote{Even though we expand and correct some minor portions of the proofs.} after \cite[Theorem 5.3.3, p. 83]{integralmeasureandordering} and \cite[Theorem 5.4.1, p. 85]{integralmeasureandordering} (see also the analogous results in \cite{riecanvrabelovaoperator}, Chapter 6 of \cite{boccutovrabelovariecan} and \cite{vrabelovariecanlast}). 

We begin with the concept of common regular convergence of \cite[Definition 5.3.2, p. 83]{integralmeasureandordering}:

\begin{definition} [Convergence with Common Regulator]

Given $f: \Omega \rightarrow X$ and $f_{n}: \Omega \rightarrow X$ functions with values in a Riesz space $X$, we say that $f_{n}$ converges to $f$ with a common regulator if there exists a $(D)$-sequence $\{a_{ij}\}$, with values in $X$, such that for every $\varphi \in \mathbb{N}^{\mathbb{N}}$ and every $\omega \in \Omega$ there exists $p = p(\omega) \in \mathbb{N}$ such that:

\[|f_{n}(\omega) - f(\omega)| \leq \bigvee_{i=1}^{\infty}a_{i\varphi(i)},\]

for any $n \geq p$.
    
\end{definition}

With this, we have the analogue of the first theorem of \cite{integralmeasureandordering} cited above, which is a sort of "master convergence theorem":

\begin{theorem} [Main Convergence Theorem] \label{mastertheorem}

Let $f_{n}: \Omega \rightarrow X$ be a sequence of Riemann integrable functions with respect to a finitely additive set function $\mu: \mathcal{H} \rightarrow Y^{+}$ in the context of a partially ordered integration structures (in particular, $X,Y,Z$ are Riesz spaces). Let one of the following hypothesis be satisfied:

\begin{enumerate} [(i)]
    \item The sequence $\{f_{n}\}_{n \in \mathbb{N}}$ has uniformly regulated net Riemann integrals, that is: there exists a non-negative triple sequence $\{a_{nij}\}$ in $Z$ satisfying the following properties
      \begin{enumerate}
          \item $\{a_{nij}\}$ is bounded for every $n$ and $a_{nij} \downarrow 0$, when $j \rightarrow \infty$;
          \item $\{\sum_{n=1}^{m}\bigvee_{i=1}^{\infty}a_{ni\varphi(i+n+1)}\}_{m=1}^{\infty}$ is a bounded sequence in $Z$ for every $\varphi \in \mathbb{N}^{\mathbb{N}}$;
          \item For $\varphi \in \mathbb{N}^{\mathbb{N}}$ and every $n \in \mathbb{N}$ and $A \in \mathcal{H}$ there exists $\gamma_{n} \in \Gamma(A)$ such that:

          \[|\int_{A}f_{n}d\mu - \sum_{i \in I_{\gamma}}f(\tau_{\gamma_{i}})\mu(\sigma_{\gamma_{i}})| \leq \bigvee_{i=1}^{\infty}a_{ni\varphi(i+n+1)},\]

          for each $\gamma \geq \gamma_{n}$, $\gamma \in \Gamma(A)$.
      \end{enumerate}
    \item There is an element $z \in Z$ such that $|\int_{A}f_{n}d\mu - \sum_{i \in I_{\gamma}}f(\tau_{\gamma_{i}})\mu(\sigma_{\gamma_{i}})| \leq z$ for every $\gamma \in \Gamma(A)$ (and every $A \in \mathcal{H}$) and every $n \in \mathbb{N}$.
    
\end{enumerate}

Suppose also that the integration procedure used for $f$ satisfies the Henstock Property. Then, if $f_{n}$ converges to a function $f: \Omega \rightarrow X$ with common regulator, then there exists a regulator $\{b_{ij}\}$ such that for every $\varphi \in \mathbb{N}^{\mathbb{N}}$ there exists a $\gamma_{0} \in \Gamma(A)$ such that:

\[|\int_{A}f_{n}d\mu - \sum_{i \in I_{\gamma}}f(\tau_{\gamma_{i}})\mu(\sigma_{\gamma_{i}})| \leq \bigvee_{i=1}^{\infty}b_{i\varphi(i)} + \sum_{m=l}^{n-1}|\int_{C_{m}}(f_{m}-f_{n})d\mu|,\]

for each $\gamma \geq \gamma_{0}$, $\gamma \in \Gamma(A)$ and $n,l \in \mathbb{N}$, $n > l$, where $C_{m} = \bigcup_{p(t_{k}) = m}B_{k}$, where $B_{k}$ is an element of the collection of sets $\sigma$ in $\gamma$. 
\end{theorem}

For simplicity, in the proof we suppose that $\Gamma(A)$, for each $A \in \mathcal{H}$, is upwards directed (the downwards case gives the same proof) and prove the result assuming the first set of hypothesis (as in \cite{integralmeasureandordering}), the proof for the second being more direct.

\begin{proof}
    As $f_{n}$ converges to $f$ with respect to a common regulator, there is a $(D)$-sequence $\{a_{ij}\}$ in $X$ such that for every $\varphi \in \mathbb{N}^{\mathbb{N}}$ and every $\omega \in \Omega$ there is a $p(\omega) \in \mathbb{N}$ such that:

    \[|f_{n}(\omega) - f_{m}(\omega)| \leq \bigvee_{i=1}^{\infty}a_{i\varphi(i)},\]

    for every $n,m \geq p(\omega)$. Since $f_{n}$ is net Riemann integrable, by hypothesis $(i)$ there is (in $Z$) a non-negative bounded triple sequence $\{a_{nij}\}$ for every $n$ and $a_{nij} \downarrow 0$, when $j \rightarrow \infty$, such that for every $\varphi \in \mathbb{N}^{\mathbb{N}}$ there exists a $\gamma_{n} \in \Gamma(A)$ such that:

    \begin{equation} \label{ultima}
        |\int_{A}f_{n}d\mu - \sum_{i \in I_{\gamma}}f(\tau_{\gamma_{i}})\mu(\sigma_{\gamma_{i}})| \leq \bigvee_{i=1}^{\infty}a_{ni\varphi(i+n+1)},
    \end{equation}

    for each $\gamma \geq \gamma_{n}$, $\gamma \in \Gamma(A)$. As $\Gamma(A)$ is upwards directed, there exists an $\gamma^{*} \in \Gamma(A)$ such that:

    \[\gamma^{*} \geq \gamma_{n}, \ \text{for every} \ n \in \mathbb{N}.\]

    Let now be an $\gamma \in \Gamma(A)$ such that $\gamma \geq \gamma^{*}$, such that the Szabó-Száz-Fleischer net associated to $\gamma$ is given by $(\tau^{\gamma}, \bigcup_{i=1}^{l}{B_{i}})$. Fix $n > l$. By the Henstock Lemma (Theorem \ref{henstocklemma}), 

    \begin{equation} \label{eq531}
        |\sum_{p(t_{k}) \geq n}f_{n}(\tau^{\gamma}_{\gamma_{{k}}})\mu(B_{k}) - \sum_{p(t_{k}) \geq n}\int_{B_{k}}f_{n}d\mu| \leq \bigvee_{i=1}^{\infty}a_{ni\varphi(i+n+1)}.
    \end{equation}

    By applying the Henstock Lemma (Theorem \ref{henstocklemma}) a second time, we get, similarly:

    \[|\sum_{p(t_{k}) = m}f_{n}(\tau^{\gamma}_{\gamma_{{k}}})\mu(B_{k}) - \sum_{p(t_{k}) = m}\int_{B_{k}}f_{n}d\mu| \leq \bigvee_{i=1}^{\infty}a_{ni\varphi(i+n+1)}.\]

    Therefore, we get the following chain of inequalities:

    \[|\sum_{p(t_{k}) < n}f_{n}(\tau^{\gamma}_{\gamma_{{k}}})\mu(B_{k}) - \sum_{p(t_{k}) < n}\int_{E_{k}}f_{n}^d\mu| \leq |\sum_{p(t_{k}) < n}f_{n}(\tau^{\gamma}_{\gamma_{{k}}})\mu(B_{k}) - \sum_{p(t_{k}) < n}f_{p(t_{k})}(\tau^{\gamma}_{\gamma_{{k}}})\mu(B_{k})|  \]

    \[ + \sum_{m=l}^{n-1}|\sum_{p(t_{k}) = m}f_{m}(\tau^{\gamma}_{\gamma_{{k}}})\mu(B_{k}) - \sum_{p(t_{k}) = m}\int_{B_{k}}f_{m}d\mu| + \sum_{m=l}^{n-1}|\int_{C_{m}}(f_{m}-f_{n})d\mu|\]

    \begin{equation} \label{eq52}
        \leq \bigvee_{i=1}^{\infty}a_{i\varphi(i)}\mu(\Omega) + \sum_{m=1}^{n-1}\bigvee_{i=1}^{\infty}a_{mi\varphi(i+m+1)} + \sum_{m=l}^{n-1}|\int_{C_{m}}(f_{m}-f_{n})d\mu|
    \end{equation}

    Now, consider the following equation defining a triple sequence $\{b_{nij}\}_{nij}$ in $Z$:

$$
b_{nij}=\begin{cases}
			a_{ij}\mu(\Omega), & \text{if $n=1$},\\
            a_{(n-1)ij}, & \text{if $n \geq 2,$}
		 \end{cases}
$$

By equations (\ref{eq531}) and (\ref{eq52}), it follows that (for $\gamma \in \Gamma(A)$, $\gamma \geq \gamma^{*}$ fixed above):

\begin{equation} \label{ultimadeste}
    |\int_{A}f_{n}d\mu - \sum_{i \in I_{\gamma}}f(\tau_{\gamma_{i}})\mu(\sigma_{\gamma_{i}})| \leq \sum_{m=1}^{n}\bigvee_{i=1}^{\infty}b_{mi\varphi(i+m+1)} + \sum_{m=l}^{n-1}|\int_{C_{m}}(f_{m}-f_{n})d\mu|.
\end{equation}

By point $(b)$ of hypothesis $(i)$, there exists an $c$ in $Z$ such that, by equation (\ref{ultima}):

\[|\int_{A}f_{n}d\mu - \sum_{i \in I_{\gamma}}f(\tau^{\gamma}_{\gamma_{i}})\mu(\sigma_{\gamma_{i}})| \leq c,\]

for every $n \in \mathbb{N}$. Now, by Fremlin lemma, there exists a $(D)$-sequence in $Z$ denoted by $\{b_{ij}\}$ such that:

\[c \wedge \sum_{m=1}^{\infty}b_{mi\varphi(i+m+1)} \leq \bigvee_{i=1}^{\infty}b_{i\varphi(i)}.\]

By these two last facts, and from equation (\ref{ultimadeste}), we finally get:

\[|\int_{A}f_{n}d\mu - \sum_{i \in I_{\gamma}}f(\tau^{\gamma}_{\gamma_{i}})\mu(\sigma_{\gamma_{i}})| \leq \bigvee_{i=1}^{\infty}b_{i\varphi(i)} + \sum_{m=l}^{n-1}|\int_{C_{m}}(f_{m}-f_{n})d\mu|,\]

which was the desired fact. This finishes the proof. 

\end{proof}

We now present the second main theorem:

\begin{theorem} \label{secondmastertheorem}
    Let $\{f_{n}\}_{n \in \mathbb{N}}$ be a sequence of net Integrable functions with respect to a finitely additive set function $\mu: \mathcal{H} \rightarrow Y^{+}$ in a partially ordered integration structure (in particular, $X,Y,Z$ are Riesz spaces), where $Z$ is weakly $\sigma$-distributive. Suppose that this sequence has uniformly approximable integrals in the following sense: there exists a $(D)$-sequence in $Z$, $\{b_{ij}\}$, such that for every $\varphi \in \mathbb{N}^{\mathbb{N}}$ there exists a $\gamma_{0} \in \Gamma(A)$ such that $|\int_{A}fd\mu - \sum_{i \in I_{\gamma}}f(\tau_{\gamma_{i}})\mu(\sigma_{\gamma_{i}})| \leq \bigvee_{i=1}^{\infty}b_{i\varphi(i)}$ for every $\gamma \geq \gamma_{0}$ and $n \in \mathbb{N}$, with $\gamma \in \Gamma(A)$ and $A \in \mathcal{H}$. Also, suppose that $f_{n}$ converges to $f: \Omega \rightarrow X$ with a common regulator. Then, $f$ is net Riemann integrable and:

    \[\lim_{n \rightarrow \infty}\int_{A}f_{n}d\mu = \int_{A}fd\mu \]
\end{theorem}

\begin{proof}
    By integrability of each $f_{n}$ and the hypothesis of uniformly approximable integrals, there exists a $(D)$-sequence in $Z$, $\{b_{ij}\}$, such that for every $\varphi \in \mathbb{N}^{\mathbb{N}}$ there exists $\gamma_{0}^{1}, \gamma_{0}^{2} \in \Gamma(A)$ such that:

    \[|\int_{A}fd\mu - \sum_{i \in I_{\gamma^{1}}}f(\tau_{\gamma^{1}_{i}})\mu(\sigma_{\gamma^{1}_{i}})| \leq \bigvee_{i=1}^{\infty}b_{i\varphi(i+1)},\]

    \[|\int_{A}fd\mu - \sum_{i \in I_{\gamma^{2}}}f(\tau_{\gamma^{2}_{i}})\mu(\sigma_{\gamma^{2}_{i}})| \leq \bigvee_{i=1}^{\infty}b_{i\varphi(i+2)},\]

    for each $\gamma^{1} \geq \gamma_{0}^{1}$ and $\gamma^{2} \geq \gamma_{0}^{2}$, with $\gamma^{1},\gamma^{2} \in \Gamma(A)$. Without loss of generality, assume $\Gamma(A)$ is upwards directed. Then, there exists $\gamma^{*} \in \Gamma(A)$ such that $\gamma \geq \gamma^{1}$ and $\gamma \geq \gamma^{2}$. Using this fact and the common regulator convergence, we have that: there exists a $(D)$-sequence in $Z$ , $\{a_{ij}\}$, such that for every $\varphi \in \mathbb{N}^{\mathbb{N}}$,

    \[|\sum_{i \in I_{\gamma^{**}}}f_{n}(\tau_{\gamma^{**}_{i}})\mu(\sigma_{\gamma^{**}_{i}})
     - \sum_{i \in I_{\gamma^{**}}}f(\tau_{\gamma^{**}_{i}})\mu(\sigma_{\gamma^{**}_{i}}) | \leq |\sum_{i \in I_{\gamma^{**}}}(f_{n}(\tau_{\gamma^{**}_{i}})-f(\tau_{\gamma^{**}_{i}}))\mu(\sigma_{\gamma^{**}_{i}})| \leq \bigvee_{i=1}^{\infty}a_{i\varphi(i+3)},\]

     for any $\gamma^{**} \geq \gamma^{*}$ and $n \geq n_{1}$, $\gamma^{**} \in \Gamma(A)$. Similarly, 

     \[|\sum_{i \in I_{\gamma^{***}}}f_{n}(\tau_{\gamma^{***}_{i}})\mu(\sigma_{\gamma^{***}_{i}})
     - \sum_{i \in I_{\gamma^{***}}}f(\tau_{\gamma^{***}_{i}})\mu(\sigma_{\gamma^{***}_{i}}) |  \leq \bigvee_{i=1}^{\infty}a_{i\varphi(i+4)},\]

     for any  for any $\gamma^{***} \geq \gamma^{*}$ and $n \geq n_{2}$, $\gamma^{***} \in \Gamma(A)$. By the (elementary) Fremlin lemma, now choose a $(D)$-sequence $\{c_{ij}\}$ in $Z$ such that, for $n \geq \max(n_{1},n_{2})$, 

     \[\bigvee_{i=1}^{\infty}b_{i\varphi(i+1)} + \bigvee_{i=1}^{\infty}b_{i\varphi(i+2)}+ \bigvee_{i=1}^{\infty}a_{i\varphi(i+3)}+ \bigvee_{i=1}^{\infty}a_{i\varphi(i+4)} \leq \bigvee_{i=1}^{\infty}c_{i\varphi(i)},\]

     and, by the equations above, 

     \[|\sum_{i \in I_{\gamma^{**}}}f(\tau_{\gamma^{**}_{i}})\mu(\sigma_{\gamma^{**}_{i}}) - \sum_{i \in I_{\gamma^{***}}}f(\tau_{\gamma^{***}_{i}})\mu(\sigma_{\gamma^{***}_{i}})| \leq \bigvee_{i=1}^{\infty}c_{i\varphi(i)},\]

     for each $\gamma^{**}, \gamma^{***} \geq \gamma$. As $(D)$-convergence is Cauchy, it follows that $f$ is net Riemann integrable on each $A \in \mathcal{H}$. By this integrability, it follows that there exists a $(D)$-sequence $\{d_{ij}\}$ such that for every $\varphi \in \mathbb{N}^{\mathbb{N}}$ there exists a $\gamma_{1} \in \Gamma(A)$ such that:

    \[|\int_{A}fd\mu - \sum_{i \in I_{\gamma}}f(\tau_{\gamma_{i}})\mu(\sigma_{\gamma_{i}})| \leq \bigvee_{i=1}^{\infty}d_{i\varphi(i+1)},\]

    for each $\gamma \geq \gamma_{1}$. By the uniform approximability hypothesis, we now get a $(D)$-sequence $\{g_{ij}\}$ in $Z$ such that for every $\varphi \in \mathbb{N}^{\mathbb{N}}$ there exists a $\gamma_{2} \in \Gamma(A)$ such that 
    
    \[|\int_{A}fd\mu - \sum_{i \in I_{\gamma^{*}}}f(\tau_{\gamma^{*}_{i}})\mu(\sigma_{\gamma^{*}_{i}})| \leq \bigvee_{i=1}^{\infty}g_{i\varphi(i+2)}\]
    
    for every $\gamma^{*} \geq \gamma_{2}$ and $n \in \mathbb{N}$. Using again that $f_{n}$ converges to $f$ with a common regulator, there exists another $(D)$-sequence $\{h_{ij}\}$, with values in $Z$, such that for every $\varphi \in \mathbb{N}^{\mathbb{N}}$ and $\gamma^{*} \geq \gamma$ there exists an $n_{0} \in \mathbb{N}$ such that:

    \[|\sum_{i \in I_{\gamma^{*}}}f_{n}(\tau_{\gamma^{*}_{i}})\mu(\sigma_{\gamma^{*}_{i}}) - \sum_{i \in I_{\gamma^{*}}}f_{n}(\tau_{\gamma^{*}_{i}})\mu(\sigma_{\gamma^{*}_{i}})|,\]

    for every $n \geq n_{0}$. 

    Again by the (elementary) Fremlin lemma, choose a $(D)$-sequence $\{l_{ij}\}$ in $Z$ such that, for $n \geq \max(n_{1},n_{2})$, 

     \[\bigvee_{i=1}^{\infty}d_{i\varphi(i+1)} + \bigvee_{i=1}^{\infty}g_{i\varphi(i+2)}+ \bigvee_{i=1}^{\infty}h_{i\varphi(i+3)}+ \leq \bigvee_{i=1}^{\infty}l_{i\varphi(i)},\]

     which, for an arbitrary $\gamma \geq \gamma^{**}$, with $\gamma^{**} \geq \gamma_{1}, \gamma_{2}$ (which exists, as $\Gamma$ is upwards directed), gives, by the inequalities above:

     \[|\int_{A}fd\mu - \int_{A}f_{n}d\mu| \leq \bigvee_{i=1}^{\infty}l_{i\varphi(i)},\]

     for every $n \geq n_{0}$. As convergence of sequences for $Z$ Dedekind complete is equivalent for $(o)$-convergence and $(D)$-convergence, we get the desired result.      
\end{proof}

The monotone and dominated convergence theorems in the setting of Riesz spaces above now follow easily as in \cite[Theorem 5.4.2, p. 87]{integralmeasureandordering} (and \cite[Theorem 5.4.3, p. 88]{integralmeasureandordering}) and \cite[Theorem 5.4.4, p. 88]{integralmeasureandordering}, from which we use their proofs.

\begin{theorem} [Monotone Convergence Theorem - First Version]

Let $\{f_{n}\}_{n \in \mathbb{N}}$ be an increasing sequence of net Riemann integrable functions with respect to a finitely additive set function $\mu: \mathcal{H} \rightarrow Y^{+}$ in a partially ordered integration structure with $Z$ $\sigma$-Dedekind complete, which satisfies the hypothesis of Theorem \ref{mastertheorem} and Theorem \ref{secondmastertheorem}, and such that, for a fixed $A \in \mathcal{H}$, $\{\int_{A}f_{n}\}_{n \in \mathbb{N}}$ is bounded in $Z$. Suppose that $f_{n}$ converges to $f: \Omega \rightarrow X$ with a common regulator and that it (the sequence) satisfies $(i)$ of Theorem \ref{mastertheorem}. Then, $f$ is net Riemann integrable and:

\[\int_{A}fd\mu = (o)-\lim_{n \rightarrow \infty}f_{n}d\mu = \bigvee_{n=1}^{\infty}\int_{A}f_{n}d\mu.\]

\end{theorem}

\begin{proof}
    By Theorem \ref{mastertheorem}, we have that there exists a regulator $\{b_{ij}\}$ such that for every $\varphi \in \mathbb{N}^{\mathbb{N}}$ there exists a $\gamma_{0} \in \Gamma(A)$ such that:

\[|\int_{A}f_{n}d\mu - \sum_{i \in I_{\gamma}}f(\tau_{\gamma_{i}})\mu(\sigma_{\gamma_{i}})| \leq \bigvee_{i=1}^{\infty}b_{i\varphi(i)} + \sum_{m=l}^{n-1}|\int_{C_{m}}(f_{m}-f_{n})d\mu|,\]

for each $\gamma \geq \gamma_{0}$, $\gamma \in \Gamma(A)$ and $n,l \in \mathbb{N}$, $n > l$, where $C_{m} = \bigcup_{p(t_{k}) = m}B_{k}$, where $B_{k}$ is an element of the collection of sets $\sigma$ in $\gamma$. 

Since $f_{l} \leq f_{m} \leq f_{n}$, we obtain that (using the isotonocity property of the net Riemann integral given in Theorem \ref{propertiesnetriemannintegral}):

\[\sum_{m=l}^{n-1}|\int_{C_{m}}(f_{m}-f_{n})d\mu| \leq \sum_{m=l}^{n-1}\int_{C_{m}}(f_{n}-f_{l})d\mu \leq \int_{A}(f_{n}-f_{l})d\mu = |\int_{A}(f_{n}-f_{l})d\mu|. \]

Since $\{\int_{A}f_{n}d\mu\}_{n \in \mathbb{N}}$ is bounded and increasing, and $Z$ is $\sigma$-Dedekind complete, there exists $\bigvee_{n=1}^{\infty}\int_{A}f_{n}d\mu$.  Therefore, $\int_{A}f_{n}d\mu - \int_{A}f_{l}d\mu$ converges to $0$ in $Z$ as $n,l \rightarrow \infty$, and the conclusion follows from Theorem \ref{secondmastertheorem}, which has the hypothesis verified by the development above. 
    
\end{proof}

Similarly, we can prove:

\begin{theorem} [Monotone Convergence Theorem - Second Version]

Let $\{f_{n}\}_{n \in \mathbb{N}}$ be an increasing sequence of net Riemann integrable functions with respect to a finitely additive set function $\mu: \mathcal{H} \rightarrow Y^{+}$ in a partially ordered integration structure with $Z$ $\sigma$-Dedekind complete, which satisfies the hypothesis of Theorem \ref{mastertheorem} and Theorem \ref{secondmastertheorem}, and such that, for a fixed $A \in \mathcal{H}$, $\{\int_{A}f_{n}\}_{n \in \mathbb{N}}$ is bounded in $Z$. Suppose that $f_{n}$ converges to $f: \Omega \rightarrow X$ with a common regulator. Suppose that $f_{1}$ and $f$ are bounded in $X$. Then, $f$ is net Riemann integrable and:

\[\int_{A}fd\mu = (o)-\lim_{n \rightarrow \infty}f_{n}d\mu = \bigvee_{n=1}^{\infty}\int_{A}f_{n}d\mu.\]

\end{theorem}

Finally, we obtain the dominated convergence theorem, which we omit the proof as it is the same as in the first version of the monotone convergence theorem:

\begin{theorem} [Dominated Convergence Theorem ]
    Let $\{f_{n}\}_{n \in \mathbb{N}}$ be a sequence of net Riemann integrable functions with respect to a finitely additive set function $\mu: \mathcal{H} \rightarrow Y^{+}$ in a partially ordered integration structure with $Z$ $\sigma$-Dedekind complete, which satisfies the hypothesis of Theorem \ref{mastertheorem} and Theorem \ref{secondmastertheorem}. Suppose that $f_{n}$ converges to $f: \Omega \rightarrow X$ with a common regulator and that there exists a bounded function $h: \Omega \rightarrow X$ such that $|f_{n}| \leq h$ pointwise for every $n \in \mathbb{N}$. Then, $f$ is net Riemann integrable and:

\[\int_{A}fd\mu = (o)-\lim_{n \rightarrow \infty}\int_{A}f_{n}d\mu\]

\end{theorem}

It is not difficult to adapt such results for the $S^{*}$-partition integral but, as we are not going to use them, we omit these developments. 

We now turn to some results in the topological case. 

\subsection{Topological Convergence Theorems.} \label{topologicalconvergencesection}

In this section, we aim to prove a type of general uniform convergence theorem for the topological vector-valued net Riemann integral. In particular, this result should be a type of "Egorov" integral limit theorem: that is, the uniform convergence applies in each element of an increasing union of sets that form the base space $\Omega$. The justification for this result is given by some equivalence results for different types of integrals that we are going to prove later: these need this type of result. 

This result is the generalization of the main uniform convergence theorem of \cite{millingtonproduct}, and also generalize the corresponding uniform convergence theorem of \cite{trombettalimit}. We will not formulate the most general version of this theorem, but only what is needed to obtain further comparison results. We begin with some definitions, which we use \cite{millingtonproduct} as a source. 

For the next results, we will suppose that $X,Y,Z$ are topological vector spaces, $Z$ complete, in the context of a integration structure in the sense of Definition \ref{integrationstructure}. 

The first definition is a generalization of the classical case of finite semivariation from vector measure theory (see the classical works of \cite{dinculeanu2014vector}, \cite{diestel} and \cite{dobrakov1}), and is given by \cite{millingtonproduct}:

\begin{definition} [Finite semivariation] \label{finitesemivariationmillington}

Let $\mu: \mathcal{H} \rightarrow Y$ be a finitely additive set function. We say that $\mu$ has finite semivariation if for every neighborhood of $0$ in $Z$, $W$, there exists a neighborhood of $0$ in $X$, $U$, such that all finite disjoint collection of sets $F$,

\[\sum_{\alpha \in F}U\mu(\alpha) \subseteq W,\]

where $U\mu(\alpha)$ is defined to be the set (using the bilinear product $\bullet$):

\[U\mu(\alpha) =\{z \in Z: z = u \bullet \mu(\alpha), u \in U\}\]
    
\end{definition}

We also use the next definition of convergence, defined in \cite{millingtonproduct}:

\begin{definition} [Quasi-uniform Convergence]

We say that a net of functions $\{f_{i}\}_{i \in I}$, where $f_{i}: \Omega \rightarrow X$, for each $i \in I$, is quasi-uniformly convergent to a function $f: \Omega \rightarrow X$ (on $\Omega \in \mathcal{H}$) with respect to a set function $\mu: \mathcal{H} \rightarrow Y$ if: for each neighborhood of $0$ in $X$, $U$, and $W$ in $Z$, there exists an $i_{0} \in I$ and a countable collection of disjoint sets $P$, in $\mathcal{H}$, such that:

\begin{enumerate} [(i)]
    \item For every finite disjoint collection $F$ (of elements of $\mathcal{H}$) finer than $P$,

    \[\sum_{\alpha \in F}g[\alpha]\mu(\alpha) \subseteq W, \ \text{for all} \ g \in \{f\} \cup \{f_{i}: i \geq i_{0}\},\]

    where $g[\alpha]$ is the image set of $\alpha$ by $g$.
    
    \item For all $i \geq i_{0}$:

    \[f(\omega) \subseteq f_{i} + U, \ f_{i} \subseteq f(\omega) + U \ \text{for all} \ \omega \in \Omega \setminus \bigcup_{\alpha \in P}\alpha.\]
\end{enumerate}
    
\end{definition}

For more details about this concept, see Section 1 of \cite{millingtonproduct}. In particular, we have:

\begin{proposition} [Uniform Convergence Implies Quasi-Uniform Convergence]
    If a net of functions $\{f_{i}\}_{i \in I}$, defined by functions $f_{i}: \Omega \rightarrow X$, for each $i \in I$, converges uniformly to a function $f: \Omega \rightarrow X$, then it converges to it quasi-uniformly for every $\mu: \mathcal{H} \rightarrow Y$ finitely additive.
\end{proposition}

We now give the final necessary definition that is needed for the proof, and which connects the concept of the net Riemann integral with the finite semivariation and quasi-uniform convergence concepts, which is motivated from the integration procedures of $S$, Sion and gauge integration:

\begin{definition} [Compatibility with Disjoint Collections]

We say that, for $A \in \mathcal{H}$, $\Gamma(A)$ is compatible with disjoint collection of sets if: given an arbitrary countable disjoint collection of sets in $\mathcal{H}$ contained in $A$, $P$, there exists $\gamma_{0} \in \Gamma(A)$ such that for each $\gamma \geq \gamma_{0}$, and for each $\sigma_{\gamma_{i}} \in \sigma_{\gamma}$, there exists a $\alpha = \alpha(\sigma_{\gamma_{i}})$ in $P$ such that:

\[\sigma_{\gamma_{i}} \subseteq \alpha.\]

We then say $\Gamma$ is compatible with disjoint collection of sets if each $\Gamma(A)$, $A \in \mathcal{H}$, is compatible with disjoint collection of sets.
    
\end{definition}

We now give the main uniform convergence theorem:

\begin{theorem} [Quasi-Uniform Convergence Theorem] \label{quasiuniformconvergencetheorem}

Let $\{f_{i}\}_{i \in I}$ be a net of $X$ valued functions converging quasi-uniformly to a function $f$ with respect to $\mu: \mathcal{H} \rightarrow Y$, finitely additive and $\mathcal{H}$  a $\sigma$-algebra\footnote{By looking at the proof below, it may be seen that this hypothesis can be weakened, but we assume it for simplicity.}, all of which are net Riemann integrable with respect to $\mu$ in a (fixed) integration structure given by the topological vector spaces $X,Y,Z$. Suppose that:

\begin{enumerate} [(i)]
    \item $\mu$ has finite semivariation,
    \item $f$ and $\{f_{i}\}_{i \in I}$ are subset integrable,
    \item The second component (i.e, the sets on which we sum $\mu$) of the Szabó-Száz-Fleishcer nets from $\Gamma(A)$ consists of sets (from $\mathcal{H}$) that are disjoint, or have intersections that are $\mu$-null sets, and contained in $A$, and the collection of points $\tau_{\gamma}$ is such that $\tau_{\gamma_{i}} \in \sigma_{\gamma_{i}}$ for each $i \in I_{\gamma}$ and $\gamma \in \Gamma(A)$ given.
    \item $\Gamma$ is compatible with disjoint collection of sets.
\end{enumerate}

Then, the net Riemann integral of $f$ exists for each $A \subseteq \Omega$ ,such that $A \in \mathcal{H}$, and

\[\int_{A}fd\mu = \lim_{i \in I}\int_{A}f_{i}d\mu.\]
    
\end{theorem}

We now modify the proof of the corresponding result in \cite{millingtonproduct} to account for the differences for the net Riemann integral and also correct some missteps. Nevertheless, the proof is quite standard. As above, we assume $\Gamma(A)$ upwards directed for each $A \in \mathcal{H}$.

\begin{proof}
    Let $W$ be a neighborhood of $0$ in $Z$. By standard results for topological groups (see chapter 5 of \cite{bourbaki2013general}), we may choose a symmetric neighborhood of $0$ in $Z$, which depends on $W$, such that:

    \[W_{1} + W_{1}+ W_{1}+ W_{1}+ W_{1} \subseteq W.\]

    Now, as $\mu$ has finite semivariation, we may choose $U$ neighborhood of $0$ in $X$ such that, for every finite disjoint collection of sets $F$, 

    \[\sum_{\alpha \in F}U\mu(\alpha) \subseteq W_{1}.\]

    Since $f_{i}$ converges to $f$ quasi-uniformly with respect to $\mu$, there exists a countable collection of disjoint sets from $\mathcal{H}$, $P_{0}$, and an $i_{0} \in I$ such that for each finite collection of disjoint sets from $\mathcal{H}$ finer than $P_{0}$, $F$,

    \begin{equation} \label{2}
        \sum_{\alpha \in F}g[\alpha]\mu(\alpha) \subseteq W_{1}, \ \text{for all} \ g \in \{f\} \cup \{f_{i}: i \geq i_{0}\},
    \end{equation}

    where $g[\alpha]$ is the image set of $\alpha$ by $g$ and,
    
    For all $i \geq i_{0}$:

    \begin{equation} \label{3}
        f(\omega) \subseteq f_{i} + U, \ f_{i} \subseteq f(\omega) + U \ \text{for all} \ \omega \in \Omega \setminus \bigcup_{\alpha \in P_{0}}\alpha.
    \end{equation}

   Now, let $A \in \mathcal{H}$. By hypothesis $(ii)$, each $\int_{A}f_{i}d\mu$ exists for all $i \in I$. Now, by hypothesis $(iv)$, there exists $\gamma_{0} \in \Gamma(A)$ such that, for each $\gamma \geq \gamma_{0}$, and for each $\sigma_{\gamma_{j}} \in \sigma_{\gamma}$, there exists an $\alpha = \alpha(\sigma_{\gamma_{j}})$ in $P_{0} \cup S \setminus \bigcup_{\alpha \in P_{0}}\alpha$ such that:

\begin{equation} \label{1}
    \sigma_{\gamma_{j}} \subseteq \alpha.
\end{equation}

Using that each $f_{i}$ is integrable on $A$, we get that there exists $\gamma_{0}^{(i)}$ in $\Gamma(A)$ such that for each $\gamma \geq \gamma^{(i)}_{0}$:

\[\sum_{j \in I_{\gamma}}f_{i}(\tau_{\gamma_{j}})\mu(\sigma_{\gamma_{j}}) \in \int_{A}f_{i}d\mu + W_{1}.\]

Now, for $i \geq i_{0}$, as $\Gamma(A)$ is upwards directed, there exists $\gamma^{(i)} \geq \gamma_{0}$ and $\gamma^{(i)} \geq \gamma_{0}^{(i)}$. Therefore, we get that, using (\ref{1}):

\[\sum_{j \in I_{\gamma^{(i)}}}f_{i}(\tau_{\gamma^{(i)}_{j}})\mu(\sigma_{\gamma^{(i)}_{j}}) - \sum_{j \in I_{\gamma^{(i)}}}f(\tau_{\gamma^{(i)}_{j}})\mu(\sigma_{\gamma^{(i)}_{j}})\]

\[= \sum_{j \in I_{\gamma^{(i)}}, \sigma_{\gamma^{(i)}} \subseteq \bigcup_{\alpha \in P_{0}}\alpha}f(\tau_{\gamma^{(i)}_{j}})\mu(\sigma_{\gamma^{(i)}_{j}}) - \sum_{j \in I_{\gamma^{(i)}}, \sigma_{\gamma^{(i)}} \subseteq \bigcup_{\alpha \in P_{0}}\alpha}f_{i}(\tau_{\gamma^{(i)}_{j}})\mu(\sigma_{\gamma^{(i)}_{j}}) \]

\[ + \sum_{j \in I_{\gamma^{(i)}}, \sigma_{\gamma^{(i)}} \subseteq S\setminus \bigcup_{\alpha \in P_{0}}\alpha}(f(\tau_{\gamma^{(i)}_{j}})-f_{i}(\tau_{\gamma^{(i)}_{j}}))\mu(\sigma_{\gamma^{(i)}_{j}}).\]

As the summation are all begin done with respect to finite disjoint collection of sets, using (\ref{2}) in the first two terms, and (\ref{3}) in the last one, we get that:

\[\sum_{j \in I_{\gamma^{(i)}}}f_{i}(\tau_{\gamma^{(i)}_{j}})\mu(\sigma_{\gamma^{(i)}_{j}}) - \sum_{j \in I_{\gamma^{(i)}}}f(\tau_{\gamma^{(i)}_{j}})\mu(\sigma_{\gamma^{(i)}_{j}}) \in W_{1} + W_{1} + W_{1} + W_{1},\]

for each $i \geq i_{0}$.

From this and a fixed $i \geq i_{0}$ (say, $i=i_{0}$), we get from this last equation, and from:

\[\sum_{j \in I_{\gamma^{(i)}}}f_{i}(\tau_{\gamma^{(i)}_{j}})\mu(\sigma_{\gamma^{(i)}_{j}}) \in \int_{A}f_{i}d\mu + W_{1},\]

that $\{\sum_{j \in I_{\gamma}}f(\tau_{\gamma_{j}})\mu(\sigma_{\gamma_{j}})\}_{\gamma \in \Gamma(A)}$ is a Cauchy net in $Z$, and as it is complete, we get that $f$ is integrable. 

Using again the previous inequalities with the same choices we get that for all $i \geq i_{0}$:

\[\int_{A}fd\mu - \int_{A}f_{i}d\mu \in W,\]

which finishes the proof. 
\end{proof}

A similar statement can be made for uniform convergence, but this time supposing $\mathcal{H}$ only a ring:

\begin{theorem} [Uniform Convergence Theorem]
\label{uniformconvergencetheoremtopological}

Let $\{f_{i}\}_{i \in I}$ be a net of $X$ valued functions converging uniformly to a function $f$ with respect to $\mu: \mathcal{H} \rightarrow Y$, finitely additive and $\mathcal{H}$ a ring, all of which are net Riemann integrable with respect to $\mu$ in a (fixed) integration structure given by the topological vector spaces $X,Y,Z$. Suppose that:

\begin{enumerate} [(i)]
    \item $\mu$ has finite semivariation,
    \item $f$ and $\{f_{i}\}_{i \in I}$ are subset integrable,
    \item The second component (i.e, the sets on which we sum $\mu$) of the Szabó-Száz-Fleishcer nets from $\Gamma(A)$ consists of sets (from $\mathcal{H}$) that are disjoint, or have intersections that are $\mu$-null sets, and contained in $A$, and the collection of points $\tau_{\gamma}$ is such that $\tau_{\gamma_{i}} \in \sigma_{\gamma_{i}}$ for each $i \in I_{\gamma}$ and $\gamma \in \Gamma(A)$ given.
    \item $\Gamma$ is compatible with disjoint collection of sets.
\end{enumerate}

Then, the net Riemann integral of $f$ exists for each $A \subseteq \Omega$ ,such that $A \in \mathcal{H}$, and

\[\int_{A}fd\mu = \lim_{i \in I}\int_{A}f_{i}d\mu.\]

\end{theorem}

The proof is quite similar, but much simpler, to the one for the quasi-uniform convergence, so we omit it. 

As in the last section, we point out that there is a direct statement of this result for the $S^{*}$-partition (or even $S^{*}$-integral) integral. But, as in the case of topological groups, this integral is contained in the net Riemann integral by a previous result, so we don't specify it. For further (topological) convergence results for the $S^{*}$ or its special case, the Birkhoff integral, see: \cite{rodriguezdobrakov}, \cite{fernandez2009birkhoff}, \cite{cascales2005birkhoff}, \cite{naralenkov2014lusin}, \cite{rodriguez2009pointwise}, \cite{rodriguez2009convergence}, \cite{balcerzak2014convergence}, \cite{balcerzak2008convergence}, \cite{potyrala2007some}, \cite{memetaj2010some}, \cite{kadetskolmogorovintegral},\cite{caponetti2017integration}, \cite{duboisconvergence} and \cite{goguadzebook}.

One immediate consequence of Theorem \ref{quasiuniformconvergencetheorem} is the following result on another example of an (Lebesgue) integration theory contained in the net Riemann integral paradigm, with the same notation as above:

\begin{corollary}
    Let $f$ be a function Ohba-integrable with respect to $\mu$, supposed of finite semivariation and finitely additive (on a $\sigma$-algebra). Then, if the hypothesis of Theorem \ref{quasiuniformconvergencetheorem} are satisfied, $f$ is net Riemann integrable and the two integrals agree. 
\end{corollary}

The integral of \cite{ohbaintegral} is a Lebesgue type integral defined for functions and (finitely additive) set functions, both with values in groups, connected by a (group valued) bilinear product. The definition of the integral is done trough an uniform approximation by simple functions, thus giving the result above. Thus, Theorems \ref{quasiuniformconvergencetheorem} and \ref{uniformconvergencetheoremtopological} can be a valuable tools to prove equivalences of the net Riemann integral with Lebesgue-type integrals that are defined, or can be reduced to, quasi-uniform/uniform convergence.

For a final note for this section we note that, although there are further convergence theorems (of topological type) for general types of Riemann integrals by \cite{sionsemigroup}, \cite{fleischerinterchange} and \cite{fleischersemigroup}, we leave their formulation and proof for the net Riemann integral for a further work. We only note that, in particular, the articles of \cite{fleischerinterchange} and \cite{fleischersemigroup} concern a type of convergence he calls "convergence in measure", but which differs from the usual convergence in measure when we consider set functions with values in spaces more general than the real line.

\subsection{A Convergence Theorem of General Type.}

The convergence theorems for the net Riemann integral in the previous sections give conditions for the interchange of limit operations in the case of topological or order convergence of the integrals and integrands. Nevertheless, these two cases get separated in all the convergence theorems above. This is not really a surprise from the point of view of convergence structures: interchanging limits operations is a difficult step in general, specially in net convergence spaces, where asking this for every net is equivalent to the topologization of the convergence (see \cite{convergencestructuresvanderwalt}). 

To advance the study of this type of theorem, we will propose a version of the convergence theorem of \cite{szaznetintegralconvergence} for lattice normed spaces with an ideal infinite element, which unites the topological and order convergence components of each proof above. This is a first step in obtaining more general convergence results for Riemann integration on general net convergence structures that do not specify topological or/and order properties of these spaces, or only do so minimally. 

Thus, in this section, we consider $X, Y, Z$ complete lattice normed spaces with lattice norms valued in $\mathcal{X}^{+}, \mathcal{Y}^{+}, \mathcal{Z}^{+}$ respectively\footnote{We shall not distinguish them by indexes, but this should not cause problems as all the notions are straightforwardly separated with respect to the spaces.}, and all of the latter being (positive cones of) Riesz spaces equipped with $(o)$-convergence, for which they are (order) complete. We assume these lattice normed spaces are connected by a fixed integration structure. 

We now consider a simple analogue of the supremum completion of Riesz spaces (see \cite{azouzi1} and \cite{azouzi2} for existence and related issues). We assume the existence and adjoin ideal elements $\infty_{X}, \infty_{Y}, \infty_{Z}$ to $\mathcal{X}^{+}, \mathcal{Y}^{+}, \mathcal{Z}^{+}$ respectively, in a way that each subsets of these Riesz spaces is bounded above by the appropriate (for each space) ideal elements in such a way that:

\[\overline{\mathcal{X}}^{+} := \mathcal{X}^{+} \cup \{\infty_{X}\}, \ \overline{\mathcal{Y}}^{+} := \mathcal{Y}^{+} \cup \{\infty_{Y}\}, \ \overline{\mathcal{Z}}^{+} := \mathcal{Z}^{+} \cup \{\infty_{Z}\},\]

are commutative monoids equipped with an operation of multiplication by positive real numbers (see \cite{azouzi2019} or \cite{donner} for more details). Notice that this procedure implies that $\sup$ and $\inf$, and therefore $\limsup$ and $\liminf$, always exists for every subset of these spaces with ideal ("infinite") elements, a fact we shall use without stating below. Equivalently, these operations on, say, $\mathcal{X}^{+}$, can always be applied in giving (existing) elements in $\overline{\mathcal{X}}^{+}$. Also, we shall employ the fact that $\limsup$ in $\overline{\mathcal{X}}^{+}$ this case is sublinear in each of the Riesz spaces adjoined with the ideal element (see also the "abstract" $\limsup$ operation in \cite{boccutosambucini2023}).

We will need to deal (only) with convergence to $0$ in such extended spaces, and we define it in the following way:

\begin{definition} [Extended Convergence] \label{extendedconvergence}

We say that a net $\{x_{i}\}_{i \in I}$ in the positive cone of a Riesz space equipped with ideal element as above, say $\overline{\mathcal{X}}^{+}$, converges to $0 \in \mathcal{X}^{+}$ if and only if, it is a quasi-subnet of a net in $\mathcal{X}^{+}$ that $(o)$-converges to $0$ in $\mathcal{X}^{+}$.
\end{definition}

In particular, we then may assume that, for some $i \geq i_{0}$, $i_{0} \in I$, in the definition above, the net $\{x_{i}\}_{i \geq i_{0}}$ is a net in $\mathcal{X}^{+}$.

The next proposition is a direct consequence of this definition:

\begin{proposition}
    Let $\{x_{i}\}_{i \in I}$ be a net of finite elements in the positive cone of a Riesz space equipped with ideal (infinite) element as above, say $\overline{\mathcal{X}}^{+}$. Then, if:

    \[\limsup_{i \in I}x_{i} = 0,\]

    with $0$ the (additive) unit in $\mathcal{X}^{+}$, the net $\{x_{i}\}_{i \in I}$ is a quasi-subnet of an $(o)$-Cauchy net (i.e, Cauchy in order) in $\mathcal{X}^{+}$.
\end{proposition}

Now, suppose that there exists a bilinear product $\star: \mathcal{X}^{+} \times \mathcal{Y}^{+} \rightarrow \mathcal{Z}^{+}$, in the sense of Definition \ref{bilinearproduct}, that connects each of the Riesz spaces in the context of the lattice normed spaces $X, Y$ and $Z$. We extend it, with notation $\overline{\star}: \overline{\mathcal{X}}^{+} \times \overline{\mathcal{Y}}^{+} \rightarrow \overline{\mathcal{Z}}^{+}$, which we shall suppress in the computations below for simplicity, in the following form:

\begin{enumerate} [(i)]
    \item $x \overline{\star} y = x \star y$ if $x \in \mathcal{X}^{+}$ and $y \in \mathcal{Y}^{+}$.
    \item $x \overline{\star} \infty_{Y} = \infty_{Z}$ and $\infty_{X} \overline{\star} y$ for each $x \in \overline{\mathcal{X}}^{+}$ and $y \in \overline{\mathcal{Y}}^{+}$.
    \item $0_{X} \overline{\star} \infty_{Y} = \infty_{X} \overline{\star} 0_{Y} = 0_{Z}$.
\end{enumerate}

And we assume that this bilinear product is continuous in each argument for (nets in) the convergence in Definition \ref{extendedconvergence}.

We now define an essential concept using all the conventions and definitions fixed in this section:

\begin{definition} [Conjugated Seminorms]

Let $\mathcal{F}$ and $\mathcal{M}$ be the spaces of all functions $f: \Omega \rightarrow X$ and set functions $\mu: \mathcal{H} \rightarrow Y$ respectively. We say that two extended-valued lattice norms (i.e, a sublinear function that takes $0$ to $0$):

\[p: \mathcal{F} \rightarrow \overline{\mathcal{X}}^{+},\]

\[q: \mathcal{M} \rightarrow \overline{\mathcal{Y}}^{+},\]

are conjugated if, for all $A \in \mathcal{H}$,:

\[\limsup_{\gamma \in \Gamma(A)}||S_{\gamma}(f,\mu)|| \leq p(f) \overline{\star} q(\mu),\]

holds for each $f \in \mathcal{F}$ and $\mu \in \mathcal{M}$ such that $p(f) < \infty_{X}$ and $q(\mu) < \infty_{Y}$.
\end{definition}

We shall denote by $\mathcal{F}_{p}$ and $\mathcal{M}_{q}$ the set of all functions in $\mathcal{F}$ with $p(f) < \infty_{X}$ and set functions $\mu \in \mathcal{M}$ with $q(\mu) < \infty_{Y}$ respectively. 

Now, we have the main theorem, a generalization of \cite[Theorem 3.8, p. 59]{szaznetintegralconvergence}:

\begin{theorem} [Lattice Norm Convergence Theorem] \label{latticenormconvergencetheorem}

Let $\{f_{i}\}_{i \in I}$ be a net of functions defined on $\mathcal{F}$, $\mu \in \mathcal{M}_{q}$, $p,q$ conjugated seminorms and $f$ a function in $\mathcal{F}_{p}$ such that:

\begin{enumerate} [(i)]
    \item $f_{i}$ is net Riemann integrable on $A \in \mathcal{H}$ with respect to $\mu$ for every $i \in I$.
    \item $\lim_{i}p(f_{i}-f) = 0$ in the sense of Definition \ref{extendedconvergence}.
    
\end{enumerate}

Then, $f$ is net Riemann integrable on $A \in \mathcal{H}$. Furthermore, if the lattice norm of $Z$ is order continuous, then:

\[\int_{A}fd\mu = \lim_{i \in I}\int_{A}f_{n}d\mu,\]

where this convergence takes place in the lattice normed space $Z$.
\end{theorem}

The proof of this result is quite standard and mimics essentially that of \cite{szaznetintegralconvergence}, but considering the conventions and structures above.

\begin{proof}
    Let $\gamma_{1}, \gamma_{2} \in \Gamma(A)$ be arbitrary. We shall use a Cauchy net type argument to prove the (net Riemann) integrability of $f$. In fact, by using the triangle inequality (sublinearity) of the lattice norm of $Z$, we have that, for each $i \in I$:

    \[||S_{\gamma_{1}}(f,\mu) - S_{\gamma_{2}}(f,\mu)|| \leq ||S_{\gamma_{1}}(f,\mu) - S_{\gamma_{1}}(f_{i},\mu)|| + ||S_{\gamma_{2}}(f,\mu) - S_{\gamma_{2}}(f_{i},\mu)|| + ||S_{\gamma_{1}}(f_{i},\mu) - S_{\gamma_{2}}(f_{i},\mu)||.\]

    As each $f_{i}$ is integrable on $A$, we have that $\{S_{\gamma}(f_{i},\mu)-S_{\gamma^{*}}(f_{i},\mu)\}_{(\gamma,\gamma^{*}) \in \Gamma(A) \times \Gamma(A)}$ is a Cauchy net in $Z$, therefore it $(o)$-converges to $0$ in $\mathcal{Z}^{+}$ and by Proposition \ref{orderconvergenceboundednets}, for every $i \in I$:

    \[\limsup_{\gamma,\gamma^{*}}||S_{\gamma}(f_{i},\mu)-S_{\gamma^{*}}(f_{i},\mu)|| = 0.\]

    Using this fact and that, by the last inequality and the fact that $\limsup$ of a quantity in $\mathcal{Z}^{+}$ always exists in $\overline{\mathcal{Z}}^{+}$, we get:

    \[\limsup_{\gamma_{1},\gamma_{2}}||S_{\gamma_{1}}(f,\mu) - S_{\gamma_{2}}(f,\mu)|| \leq 2\limsup_{\gamma_{1}}||S_{\gamma_{1}}(f,\mu) - S_{\gamma_{1}}(f_{i},\mu)||,\]

    where both $\limsup$ takes value, at this point of the proof, in $\overline{\mathcal{Z}}^{+}$. Now, using that $p$ and $q$ are conjugated norms and the last inequality, we have that, for $i \geq i_{0}$ such that $p(f_{i}-f) < \infty_{X}$, in which $i_{0}$ exists by Definition \ref{extendedconvergence},:

    \[\limsup_{\gamma_{1},\gamma_{2}}||S_{\gamma_{1}}(f,\mu) - S_{\gamma_{2}}(f,\mu)|| \leq p(f_{i}-f)q(\mu),\]

    and by using hypothesis $(ii)$ and the fact that the extended bilinear product $\overline{\star}$ is continuous for the extended convergence (Definition \ref{extendedconvergence}) in each argument (and that a quasi-subnet of a convergent net is convergent by axiom $(ii)$ of Definition \ref{netconvergencestructure}),

    \[\limsup_{\gamma_{1},\gamma_{2}}||S_{\gamma_{1}}(f,\mu) - S_{\gamma_{2}}(f,\mu)|| = 0\]

    Which shows that $\{S_{\gamma}(f,\mu)-S_{\gamma^{*}}(f,\mu)\}_{(\gamma,\gamma^{*}) \in \Gamma(A) \times \Gamma(A)}$ is a (order) Cauchy net in $Z$, which by completeness gives the net Riemann integrability of $f$ on $A \in \mathcal{H}$.

    Now, assuming the lattice norm of $Z$ is order continuous, we have, by integrability of $f$ and each $f_{i}$, $i \in I$, and by Proposition \ref{orderconvergenceboundednets} in the second equality:

    \[||\int_{A}fd\mu - \int_{A}f_{i}d\mu|| = (o)-\lim_{\gamma}||S_{\gamma}(f,\mu)-S_{\gamma}(f_{i},\mu)|| = \limsup_{\gamma}||S_{\gamma}(f,\mu)-S_{\gamma}(f_{i},\mu)||,\]

    and choosing again $i \geq i_{0}$ such that $p(f_{i}-f) \leq \infty_{Z}$, in which $i_{0}$ exists by Definition \ref{extendedconvergence}:

    \begin{equation}\label{lastone}
     \limsup_{\gamma}||S_{\gamma}(f,\mu)-S_{\gamma}(f_{i},\mu)|| \leq p(f_{i}-f)q(\mu),
    \end{equation}

    from which, by hypothesis $(ii)$ again, and the fact that the extended bilinear product $\overline{\star}$ is continuous for the extended convergence (Definition \ref{extendedconvergence}) in each argument (and that a quasi-subnet of a convergent net is convergent by axiom $(ii)$ of Definition \ref{netconvergencestructure}), we can conclude that, using that order convergence is locally solid:

    \[(o)-\lim_{i}||\int_{A}f_{i}d\mu - \int_{A}fd\mu|| = 0,\]

    where the convergence occurs in $\mathcal{Z}^{+}$, as the right-hand side of Equation \ref{lastone} is an element of $\mathcal{Z}^{+}$. Therefore, we get convergence in the sense of lattice normed spaces on $Z$, which finishes the proof. 
\end{proof}

It is readily seen that this theorem is generalizable to $Z$ equipped with a (possibly uncountable) collection of lattice norms.

It is also to be noted that an immediate consequence of this last theorem is the following example of another (Lebesgue) integration theory contained in the net Riemann integral paradigm, with the same notation as above:

\begin{corollary}
    Let $f$ be a function Kusraev-Malyugin integrable with respect to $\mu: \mathcal{H} \rightarrow Y$, supposed of finite variation and finitely additive. Then, $f$ is net Riemann integrable and the two integrals agree. 
\end{corollary}

The integral of \cite{kusraevmalyugin} is a Lebesgue type integral defined for real valued-functions and (finitely additive) set functions with values in a lattice normed space $Y$. The definition of the integral is done trough the classical Lebesgue approximation by simple functions. Choosing the lattice norm on the space of integrable functions, and the notion of finite variation, detailed in \cite{kusraevmalyugin}, or more accessibly in \cite{kusraevmalyuginenglisharticle}, and using the last theorem, the corollary is straightforwardly proven. We shall not detail it for questions of space. Nevertheless, this shows the utility of Theorem \ref{latticenormconvergencetheorem} even in case we don't have a directly usable topology or order in the spaces considered. 

To study the relation of all the Riemmanian constructions so far, and other Lebesgue integrals, we shall present the general concept of Lebesgue-type integration in more depth.

\section{Lebesgue Type Integrals in Net Convergence Structures.}

Taking the cues from the Lebesgue integrals of last sections, we see that, in general, all the Lebesgue type procedures for integration cited previously, and to be cited later as well, follows the same methodology:

\begin{enumerate} [(i)]
    \item Choose a space of functions $L$ that contains a specific class of functions and define the integral for this latter class in some sense. For the simple functions we define the integral, generally, as a linear combination of the constants and measure of the sets in the definition.
    \item Choose now a class of functions that can be approximated, in the sequential or net sense, by the functions in $L$. Prove that the limit of the integrals of such functions is independent of the chosen approximation.
    \item Define, then, using the last point and some sort of completeness, the integral of a function $f$ as the limit of the integrals of the approximating functions on $L$.
\end{enumerate}

The convergence properties given by a topology, order, or other type of (net/filter) convergence structure appears in a very essential way in point (ii) - most of the time to produce uniqueness of the integral for each approximating sequence. In this sense, this strategy for obtaining an integration theory is dependent on the details of the convergence involved, being difficult to generalize to arbitrary convergence structures, something that can be done, as shown above, for the Riemmanian case. 

Nevertheless, a kind of general Lebesgue integral can be propose if we let the uniqueness portion of the procedure be a part of the definition of this integral - something that can be checked in every case of interest.

This general integral is based on an extension of the sequential definition given in Section 3.2 of \cite{boccutoxenofon}, and also in the ideas of Boccuto and Sambucini in \cite{boccutosambuciniconvergencegroups} and \cite{boccutosambucini2023} and of Kusraev and Malyugin in the lattice-normed case of \cite{kusraevmalyugin}.

Therefore, we consider, in this section, the following setting:

\begin{enumerate}
    \item $(X, \eta_{1}),(Y, \eta_{2}),(Z,\eta_{3})$ three vector spaces equipped with net convergence structures, and connected by a product as in Definition \ref{bilinearproduct},
    \item $(\Omega, \mathcal{H})$ an abstract paved space, 
    \item $\mu: \mathcal{H} \rightarrow Y^{+}$ an arbitrary finitely additive set function.
\end{enumerate}

Now, consider a space of functions $F$, consisting of all functions $f: \Omega \rightarrow X$ satisfying some fixed property (e.g, being simple, bounded, etc) such that it contains the class of simple functions $\mathcal{S}$  defined on $(\Omega, \mathcal{H})$, and with the property that $f \mathbbm{1}_{A} \in F$, for each $A \in \mathcal{H}$. 

Suppose now that there exists a (well defined) functional denoted by:

\[\int_{\Omega}^{L}(\cdot) \ d\mu: F \rightarrow Z,\]

such that:

\begin{enumerate}
    \item $\int^{L}_{\Omega} (\cdot) \ d\mu$ is linear on $F$ and,
    \item $\int^{L}_{\Omega} (\cdot) \ d\mu$ is such that, for each simple function $g = \sum_{i=1}^{n}a_{i}\mathbbm{1}_{A}$ in $\mathcal{S}$, 
    \[\int^{L}_{\Omega}gd\mu = \sum_{i=1}^{n}a_{i}\mu(A_{i}),\]
\end{enumerate}

Now put $\int^{L}_{A}fd\mu = \int^{L}_{\Omega} f\mathbbm{1}_{A}d\mu$.

\begin{definition} [Abstract Lebesgue Integral] \label{lebesgueintegral}

A function $f: \Omega \rightarrow X$ is said to be Abstractly Lebesgue integrable, or just Lebesgue integrable in case no confusion arises, if there is a net $\{f_{i}\}_{i \in I}$ of functions in $F$ such that $f_{i} \rightarrow f$ in some convergence defined on the function space $L$ with respect to $(X, \eta_{1})$, which we call the approximation procedure, and such that the limit:

\[\lim_{(I, \eta_{3})}\int^{L}_{A}f_{i}d\mu, \]

exists in $Z$ for every $A \in \mathcal{H}$, and is independent of the approximating net $\{f_{i}\}_{i \in I}$. In this case, we denote by the set function produced for $f$, and for each $A$ in the limit above, as $I_{f}(\cdot)$, and call:

\[I_{f}(A) =  \lim_{(I, \eta_{3})}\int^{L}_{A}f_{i}d\mu,\]

the Abstract Lebesgue integral of $f$ on $A$.
\end{definition}

The first thing to notice is that this integral is readily seen to be an homomorphism (with the pointwise addition operation) on the space of Lebesgue integrable functions.  Other properties of this integral seem to require special hypothesis, such as the uniform integrability property of the net approximation of functions as in \cite{boccutoxenofon} or some other restrictive property of the approximation procedure. In fact, by imposing some reasonable order-based properties on $X,Y$ and $Z$ we may obtain explicit convergence theorems: in the setting of \cite{lvaluedintegrationvanderwalt}, for example, all their convergence theorems are valid if we suppose $Z$ equipped with a partially ordered convergence structure that is locally solid and order continuous (for the second terminology, see \cite{vanderwaldlocallysolid}).

We also comment that not only the sequential (general) Lebesgue integrals of \cite{boccutoxenofon} and \cite{boccutosambuciniconvergencegroups}, as well as \cite{ballveconvergencespaces}, is contained in our definition, but also most examples of the Lebesgue type integrals used in the literature, which we shall review in the next sections. In particular, taking $L$ to be the space of elementary functions, and approximation procedure of almost uniform convergence, we get the Pavlakos Integral. Also, by analogous constructions, the recently defined integrals of \cite{lvaluedintegrationvanderwalt}, \cite{orderintegrals}, defined on the basis of a sequential (monotone) approximation by simple functions, are also contained in our general definition.

In the Riemmanian context, it is difficult (if even possible) to get the general equivalence of the main definition of this section and the net Riemann integral. General equivalence results in this regard may be obtainable for Lebesgue integration procedures that satisfy one, or more, of the monotone/uniform or/and dominated convergences proved in the previous sections. But, even in this context, a case by case basis by considering each integration procedure individually seems more reasonable. 

With these comments we stop the development of this integral at this point, and refer the reader to the references cited above for more details of this type of Lebesgue integration procedure. We also cite that some general convergence theorems for this integration procedure is expected in future works of the present author.

\subsection{Lebesgue Integrals in Riesz Spaces and Riemmanian Equivalences.} \label{pavlakosintegralsection}

We now expound some aspects of a very specific type of Lebesgue integral for functions and measures with values in Riesz spaces. For a brief and incomplete survey of such constructions, we mention the uniform net limit type integrals of \cite{Berg}, \cite{pavlakosintegration}, \cite{pavlakosrepresentation}, \cite{kapposlagebraicintegral}, the sequential integrals of \cite{wrightintegralmeasure}, \cite{wrightmeasurepartiallyorderedspaces1972}, \cite{siposintegrationpariallyordered},  \cite{potocky1}, \cite{potocky2}, \cite{potocky3}, \cite{malicky1}, \cite{malicky2}, \cite{malicky3}, \cite{sobolevshcherbindifferentiation}, \cite{sobolevschcherbinintegration}, \cite{cristescubook}, \cite{cristescu1}, \cite{cristescu2}, \cite{isakovsubmeasure}, \cite{poroshkin}, \cite{haluskacompletevectorlattices}, \cite{andreuriesz} \cite{boccutorieszabstract}, \cite{integraldconvergenceboccuto}, \cite{boccutodifferential}, \cite{boccutocandeloro2009}, \cite{boccutoxenofon}, \cite{boccutosambucini2023}, \cite{kusraevmalyugin}, \cite{kusraevtasoev1}, \cite{groblermartingale}, \cite{vanrooijzuijlen1}, \cite{vanrooijzuijlen2}, \cite{orderintegrals}, \cite{jeujiang}, \cite{lvaluedintegrationvanderwalt} the theses of \cite{pavlakosthesis}, \cite{jeurnink1982integration}, \cite{groenewegen1982spaces}, \cite{vanzuijlenmaster}, containing abundant information on the structures that will be discussed here, as well as the related Daniell type integrals, not covered here in any detail, of \cite{mchsaneorder}, \cite{matthesdaniell}, \cite{kapposmallios}, \cite{mallios}, \cite{sistodaniell}, \cite{leontiadis1}, \cite{leontiadis2}, \cite{leontiadis3}, \cite{alfsen}, \cite{Higgsintegral}, \cite{wilhelm1}, \cite{wilhelm2}, \cite{anderheiden}, \cite{shamaev}, \cite{vonkomerova}, \cite{congostiglesias}, \cite{riecandaniell1}, \cite{riecandaniell2}, \cite{riecandaniell3}, \cite{groblerdaniell}, \cite{vrabelevrabelova}, \cite{isakov}, \cite{markwardtintegral}, and \cite{westerbaan2019lattice}.

Some of the constructions given above are valid, and in fact presented in the works of the original authors, for lattice ordered groups or partially ordered spaces more generally, but to maintain certain uniformity, and also apply the convergence theorems in the previous sections, we will put some restrictions on the $X,Y,Z$ and their convergence structures. 

More specifically, we will restrict ourselves in this section to the following setting:

\begin{enumerate}
    \item Riesz spaces $X,Y,Z$ such that $Z$ is Dedekind complete and weakly $\sigma$-distributive. In relation to the measure theoretic objects, $(\Omega, \mathcal{H})$ will always denote a measurable space, i.e, $\mathcal{H}$ is a $\sigma$-algebra of subsets of $\Omega$, and $\mu: \mathcal{H} \rightarrow Y^{+}$ will be $\sigma$-additive in the (partially ordered) net convergence structure of $Y$ in the sense of Definition \ref{sigmaadditivemeasurenetconvergence}. We also suppose that $\mu$ is a regular integrator.

    \item In relation to convergence, we shall equip $Z$ with the net convergence structure of $(D)$-convergence or order convergence (each case will be explicitly specified), and $X,Y$ with order convergence in the sense of Definition \ref{usualorderconvergence} (which, in this case, is equivalent to order convergence in the sense of Definition \ref{pavlakosorderconvergence}).
\end{enumerate}

Further hypothesis of any sort will be pointed out when necessary. 

With these observations, we pass to a brief presentation the integration theory of Pavlakos \cite{pavlakosintegration}, which contains the integral of \cite{kapposlagebraicintegral} and \cite{Berg}.

To develop his integral, \cite{pavlakosintegration} supposes that $X$ is of countable type with respect to order convergence, i.e $X$ is fist countable with respect to the net convergence structure given by order convergence (see \cite{convergencestructuresvanderwalt}). Therefore, we will also suppose throughout this (sub)section, besides 1. and 2. above, that:

3. $X$ is first countable with respect to the net convergence structure of order convergence.

Now, we define a list of functions spaces specified by Pavlakos in different sections in \cite{pavlakosintegration}. To obtain clarity and brevity, we will define the function spaces in question already with the hypothesis put onto them by \cite{pavlakosintegration} for the purpose of integration, and with the same notation.

\begin{definition} [Function Spaces]

For $X,Y,Z$ and $(\Omega, \mathcal{H})$ fixed as above, we define:

\begin{enumerate}[(i)]
    \item $F(\Omega,X)$, the space of all functions $f:\Omega \rightarrow X$ equipped with addition and multiplication by (real) scalars pointwise, and with the pointwise partial order inherited from $X$. In this case, this function space  becomes a Riesz Space. 

    \item $E(\Omega,X)$, the (Riesz subspace) of $F(\Omega,X)$ consisting of all $\mathcal{H}$-elementary functions, i.e: the space of all functions $f$ of the form

    \[f(\omega) = \sum_{n=1}^{\infty}a_{i}\mathbbm{1}(\omega)_{A_{i}},\]

    where $\{a_{n}\}_{n=1}^{\infty}$ is a sequence of elements of $X$ and $\{A_{n}\}_{n \in \mathbb{N}}$ is a sequence of disjoint sets from $\mathcal{H}$.

    \item $M(\Omega,X)$, the (Riesz subspace) of $F(\Omega,X)$ consisting of all functions $f$ such that there exists a sequence $\{f_{n}\}_{n \in \mathbb{N}} \in E(\Omega,X)$ such that

    \[f_{n} \rightarrow f,\]

     uniformly almost everywhere in order as defined in the previous section, or $(u)-\lim_{n \rightarrow \infty}f_{n}=f$ a.e.-$\mu$ on $\Omega$ in \cite{pavlakosintegration} notation and terminology. This is called the space of $(\mathcal{H},\mu)$-measurable functions. 
\end{enumerate}

\end{definition}

We now define the integral in question by steps, as is usual for Lebesgue type integration procedures. For a non-negative function $f \in E(\Omega,X)$ of the form 

\[f = \sum_{n=1}^{\infty}a_{i}\mathbbm{1}_{A_{i}},\]

with the same notation as in the previous definition, we say that it is Pavlakos integrable if the following (unconditional series) limit:

\[(o)-lim_{n \rightarrow \infty} \sum_{i=1}^{n}a_{i}\mu(A_{i}),\]

exists in $Z$. Then, in this case we use the notation:

\[\int_{\Omega}^{P}fd\mu = (o)-lim_{n \rightarrow \infty} \sum_{i=1}^{n}a_{i}\mu(A_{i}) = \sum_{n=1}^{\infty}a_{i}\mu(A_{i}).\]

By Lemmas 1.2 and 3.1 of \cite{pavlakosintegration}, this quantity is well defined in the sense of not depending on the rearrangement of the series and the representation of $f$ as an elementary function (which is, in general, non-unique).  

For a general $f \in E(\Omega,X)$, $f$ is Pavlakos integrable if $f^{+}$ and $f^{-}$ are integrable. In this case, we define:

\[\int_{\Omega}^{P}fd\mu = \int_{\Omega}^{P}f^{+}d\mu - \int_{\Omega}^{P}f^{-}d\mu,\]

with the integral in a set $A \in \mathcal{H}$ being defined as:

\[\int_{A}^{P}fd\mu = \int_{\Omega}^{P}f\mathbbm{1}_{A}d\mu.\]

By Theorem 3.5 and Corollary 3.7 of \cite{pavlakosintegration}, the defined integral on $E(\Omega,T)$ is a linear, positive and isotone operator and it is absolute, in the sense that $f$ is Pavlakos integrable if and only if $|f|$ is Pavlakos integrable. 

We now enounce Theorems 3.11 and Corollary 3.12 of \cite{pavlakosintegration}, respectively, for the definition of the integral on $M(\Omega,X)$:

\begin{theorem} \label{existence1}
    Let $\{f_{n}\}_{n \in \mathbb{N}}$ be an increasing sequence of $E(\Omega,T)$ Pavlakos integrable functions.

    Let $f \in M(\Omega,X)$ such that:

    \[(u)-\lim_{n \rightarrow \infty}f_{n} = f \ \mu-\text{a.e. on} \ \Omega.\]

    Then, there exists the order limit $(o)-\lim_{n \rightarrow \infty}\int_{\Omega}f_{n}d\mu$.
\end{theorem}

\begin{lemma} \label{existence2}
    Let $\{f_{n}\}_{n \in \mathbb{N}}$ and $\{b_{n}\}_{n \in \mathbb{N}}$ be increasing sequences of Pavlakos integrable functions on $E(\Omega,X)$ such that:

    \[(u)-\lim_{n\rightarrow \infty}f_{n} = (u)-\lim_{n\rightarrow \infty}g_{n},  \mu-\text{a.e. on} \ \Omega.\]

    Then, 

    \[(o)-\lim_{n \rightarrow \infty}\int^{P}_{\Omega}f_{n}d\mu = (o)-\lim_{n \rightarrow \infty}\int^{P}_{\Omega}g_{n}d\mu.\]
\end{lemma}

We then have the main integral of \cite{pavlakosintegration}:

\begin{definition} [Pavlakos Integral]

A function $f \in M(\Omega,X)$ is said to be $(\mathcal{H},\mu)$-Pavlakos integrable on $\Omega$ if there exists an increasing sequence $\{f_{n}\}_{n \in \mathbb{N}}$ in $E(\Omega,X)$ of Pavlakos integrable functions such that $(u)-\lim_{n \rightarrow \infty}f_{n}=f$ $\mu$-a.e. on $\Omega$.

In this case, we the integral of $f$ on $\Omega$ is the element of $Z$ defined by the equality

\[\int^{P}_{\Omega}fd\mu = (o)-\lim_{n \rightarrow \infty}\int^{P}_{\Omega}f_{n}d\mu\]
\end{definition}

It is a consequence of Theorem \ref{existence1} and Lemma \ref{existence2} that this integral is well defined. The definition of this integral on a set $A \in \mathcal{H}$ and its properties are as same as before. For more basic properties, we refer the reader to the original article of \cite{pavlakosintegration}. 

The last special property of this integral that we will need is the following Theorem 3.8 of \cite{pavlakosintegration}:

\begin{theorem} [$\sigma$-Additivity of the Pavlakos Integral] \label{sigmaadditivtypavlakosintegral}

Let $f$ be a non-negative Pavlakos integrable function, then the set function $\nu: \mathcal{H} \rightarrow Z^{+}$ defined by:

\[\nu(A) = \int_{A}^{P}fd\mu, \]

is $\sigma$-additive on $\mathcal{H}$.
    
\end{theorem}

We now begin to present the connections of this integral with the previous ones of Riemann type. To obtain such a comparison, consider the following result:

\begin{proposition} \label{equivalencefirstelement}
    Let $f: \Omega \rightarrow X^{+}$ be a non-negative elementary function of the form:

\[f = \sum_{n=1}^{\infty}a_{i}\mathbbm{1}_{A_{i}},\]

which is Pavlakos integrable on $\Omega$ with:

\[\int_{\Omega}^{P}fd\mu = (o)-\lim_{n \rightarrow \infty}\sum_{i=1}^{n} a_{i}\mu(A_{i}) = \sum_{n=1}^{\infty}a_{n}\mu(A_{n}).\]

Then, for every $P \in \mathcal{P}(\bigcup_{n=1}^{\infty}A_{n})$ of the form $P = \{B_{m}\}_{m=1}^{\infty}$ such that $P \geq P_{0} = \bigcup_{n=1}^{\infty}\{A_{n}\} $, then:

\[\sum_{m=1}^{\infty}b_{i}\mu(B_{i}) = \sum_{n=1}^{\infty}a_{i}\mu(A_{i}),\]

where $b_{i}$ is the value of $f$ at $B_{i}$, where all the series are unconditionally convergent.
    
\end{proposition}

\begin{proof}
    Let $P,P_{0}$ be as in the hypothesis of the proposition. Then, by definition:

    \[\bigcup_{m=1}^{\infty}B_{m} = \bigcup_{n=1}^{\infty}A_{n}.\]

    Therefore, as $f$ is non-negative and Pavlakos Integrable on $\Omega$, by  Theorem \ref{sigmaadditivtypavlakosintegral}, we have that:

    \[\int_{\bigcup_{m=1}^{\infty}B_{m}}^{P}fd\mu = \sum_{m=1}^{\infty}\int_{B_{m}}^{P}fd\mu =  \sum_{n=1}^{\infty}a_{n}\mu(A_{n}). \]

    But, as $P$ refines $P_{0}$, for each $B_{i} \in P$ there exists (a unique) $A_{j} \in P_{0}$ such that $B_{i} \subseteq A_{j}$. Therefore, $f$ restricted to $B_{i}$, $f|_{B_{i}}$, is such that:

    \[f|_{B_{i}} := b_{i} = a_{j},\]

    which implies that $f$ is constant on each element of $P$ and thus from the definition of the Pavlakos integral we get that:

    \[\sum_{m=1}^{\infty}\int_{B_{m}}^{P}fd\mu = \sum_{m=1}^{\infty}b_{m}\mu(B_{m}) ,\]

    from which we get:

    \[\sum_{m=1}^{\infty}b_{m}\mu(B_{m}) = \sum_{n=1}^{\infty}a_{n}\mu(A_{n}),\]

    which finishes the proof.

\end{proof}

From this proposition, we get:

\begin{corollary} \label{equivalencesecondresult}
    Let $f: \Omega \rightarrow X^{+}$ be a non-negative elementary function of the form:

\[f = \sum_{n=1}^{\infty}a_{i}\mathbbm{1}_{A_{i}},\]

which is Pavlakos integrable on $\Omega$ with:

\[\int_{\Omega}^{P}fd\mu = (o)-\lim_{n \rightarrow \infty}\sum_{i=1}^{\infty} a_{i}\mu(A_{i}) = \sum_{n=1}^{\infty}a_{n}\mu(A_{n}).\]

Then, $f$ is $S^{*}$-partition on $\Omega$ and:

\[\int_{\Omega}^{S^{*}}fd\mu = \sum_{n=1}^{\infty}a_{n}\mu(A_{n}),\]

where all the series are unconditionally convergent.

\end{corollary}

\begin{proof}
    Without loss of generality, suppose that $\Omega = \bigcup_{n=1}^{\infty}A_{n}$, and take $P_{0} = \bigcup_{n=1}^{\infty}\{A_{n}\}$. Let $P \in \mathcal{P}(\Omega)$ such that $P \geq P_{0}$. By construction (and the same argument as in the proof of the previous proposition), for each choice function $\delta$, $f(\delta_{\alpha})$, $\alpha \in P$, is a constant, and:

    \[\sum_{\alpha \in P}f(\delta_{\alpha})\mu(\alpha) := \sum_{m=1}^{\infty}b_{m,P}\mu(B_{m,P}),\]

    where $P = \bigcup_{m=1}^{\infty}\{B_{m,P}\}$, the subindex $m,P$ indicates that the sets corresponds to the partition $P$, and for which the existence of the (unconditional) series in the right-hand side follows from Proposition \ref{equivalencefirstelement}. In fact, from the same proposition, we get that:
    
    \[\sum_{\alpha \in P}f(\delta_{\alpha})\mu(\alpha) = \sum_{m=1}^{\infty}b_{m,P}\mu(B_{m,P}) = \sum_{n=1}^{\infty}a_{n}\mu(A_{n}).\]

    Now, take an arbitrary $(D)$-sequence $\{a_{ij}\}$ in $Z$ and for each $\varphi \in \mathbb{N}^{\mathbb{N}}$ take, independent of $\varphi$, $P_{0}$ as above. By the preceding argument, for each $P \geq P_{0}$, $P \in \mathcal{P}(\Omega)$,

    \[|\sum_{\alpha \in P}f(\delta_{\alpha})\mu(\alpha) - \sum_{n=1}^{\infty}a_{n}\mu(A_{n})| \leq \bigvee_{i=1}^{\infty}a_{i\varphi(i)},\]

    which shows that $f$ is $S^{*}$-partition integral with the same value as the Pavlakos integral of $f$.    
\end{proof}

We now get to the main comparison result:

\begin{theorem} [Pavlakos Integrability implies $S^{*}$-partition Integrability - Non-negative Case] \label{equivalence}

Let $f: \Omega \rightarrow X$ be a non-negative Pavlakos integrable function on $\Omega$. Then, $f$ is $S^{*}$-partition integrable on $\Omega$ and the two integrals agree. 
\end{theorem}

\begin{proof}
    By the definition of the Pavlakos integrability of $f$, let $\{f_{n}\}_{n \in \mathbb{N}}$ be a sequence of non-negative increasing functions, in $E(\Omega,X)$, of Pavlakos integrable functions such that $(u)-\lim_{n \rightarrow \infty}f_{n}=f$ $\mu$-a.e. on $\Omega$.

    By Corollary \ref{equivalencesecondresult}, we get that each $f_{n}$ is $S^{*}$-partition integrable and:

    \[\int_{\Omega}^{P}fd\mu = \int_{\Omega}^{s^{*}}fd\mu.\]

    On the other hand, by the almost everywhere uniform convergence Theorem for the $S^{*}$-partition integral, Theorem \ref{uniformconvergencetheorem2null}, we have that $f$ is $S^{*}$-partition integrable on $\Omega$:

    \[(o)-\lim_{n \rightarrow \infty}\int_{\Omega}^{S^{*}}f_{n}d\mu = \int_{\Omega}^{S^{*}}fd\mu.\]

    As we have that, by definition:

    \[(o)-\lim_{n \rightarrow \infty}\int_{\Omega}^{P}f_{n}d\mu = \int_{\Omega}^{P}f d\mu,\]

    it follows that:

    \[\int_{\Omega}^{P}fd\mu = \int_{\Omega}^{S^{*}}fd\mu,\]

    which finishes the proof.
\end{proof}

For the general case, we simply write the decomposition of a Pavlakos integrable $f$ on its negative and positive parts, both non-negative and Pavlakos integrable, use Theorem \ref{equivalence}, and then we get:

\begin{corollary} [Pavlakos Integrability implies $S^{*}$-partition Integrability - General Case]

Let $f: \Omega \rightarrow X$ be a  Pavlakos integrable function on $\Omega$. Then, $f$ is $S^{*}$-partition integrable on $\Omega$ and the two integrals agree. 
\end{corollary}

Related arguments can be made for the Sion integral in the place of the $S^{*}$-partition integral, but they are similar, so we will not state them. 

One interesting aspect of the discussion from now on is, then, to connect net-approximation (Lebesgue) type integrals to the Pavlakos' one in an attempt to relate them to the Riemmanian approach. As real (Borel) measurable functions $f: \Omega \rightarrow \mathbb{R}$ are strongly measurable in the sense of Böchner (see \cite{dinculeanustochastic}), and thus can be approximated uniformly by a sequence of increasing elementary functions (see again \cite{dinculeanustochastic} and \cite{monteirodissertação} for a study of how uniform convergence in the usual sense, for real functions, is equivalent to ours), we may use the monotone convergence theorem of Maitland-Wright proved in \cite{wrightintegralmeasure} and the device of decomposing a function into negative and positive parts to prove:

\begin{proposition} [Maitland-Wright Integrable Function is Pavlakos Integrable]
    Let $f: \Omega \rightarrow \mathbb{R}$ be a Maitland-Wright integrable function on $\Omega$, then it is Pavlakos integrable on $\Omega$ and the two integrals agree.
\end{proposition}

As a direct corollary, we have:

\begin{proposition} [Maitland-Wright Integrable Function is $S^{*}$ Integrable]
    Let $f: \Omega \rightarrow \mathbb{R}$ be a Maitland-Wright integrable function on $\Omega$, then it is $S^{*}$ integrable on $\Omega$ and the two integrals agree.
\end{proposition}

For more details on the Maitland-Wright integral, see \cite{wrightintegralmeasure} and \cite{wrightmeasurepartiallyorderedspaces1972}, and \cite{pavlakosintegration}. A similar result can be obtained for the integral of Kusraev and Tasoev constructed in \cite{kusraevtasoev1}.

Thus, the results above helps to connect the Riemmanian approaches to integration more closely with the sequential/net-based Lebesgue integrals in the most well known approaches to integration in partially ordered spaces, specially for real functions and Riesz-valued measures and Riesz-valued functions and measures with uniform approximations. A further interesting aspect would be to make analogous Riemmanian constructions of the integral of \cite{jeujiang}, which can take values in the Dedekind completion of a given Riesz space.

\subsection{Lebesgue Integrals in Topological Vector Spaces and Riemmanian Equivalences.} \label{topologicallebesgueintegralssection}

On this section, we will study some aspects of the general Lebesgue integral, previously defined, but now in the setting of, not necessarily partially ordered, topological vector spaces. The literature on the subject contains numerous publications, most of them based on the classical integrals of \cite{birkhofforiginalarticle}, \cite{bochneroriginalarticle} and \cite{pettisoriginalarticle}. A great bibliography in this context is provided in the book \cite{henstockmagnum}.

Nevertheless, there is a less know part of the subject which begins with the construction of the remarkable bilinear integrals of \cite{dobrakov1} and \cite{bartle1956general} (see also \cite{panchapageneralizedpettis},  \cite{panchapagesandistinguishing} and \cite{panchapagesanbook}), and contains reverberations in the recent literature, specially in view of applications in mathematical physics (see \cite{jefferies}). The period after the articles of Dobrakov and Bartle just cited contained different generalizations, and in different directions (again, non-exhaustively, but not covered by \cite{henstockmagnum}):

\begin{enumerate} [(i)]
    \item For (real) scalar functions integrated with respect to (convexly bounded) vector measures with values in a topological vector space, we have the recent works of \cite{labuda}, \cite{labudaconvexly} and \cite{drewlabuda}, which generalized the earlier defined integrals  by \cite{Turpin2}, \cite{role}, \cite{butkovic1979integration}, \cite{butkovic} and \cite{thomas} (based on \cite{thomas1972radon}).

    \item A student of Dobrakov, Jan Haluska, beginning in \cite{haluska1} and \cite{haluska2}, developed a generalization of the Dobrakov procedure for locally convex bornological topological vector spaces. This integral contains structures of product type and various convergence theorems, which can be gauged from \cite{haluskafubini} and \cite{haluskaintegrable}. Connections with the $S^{*}$-partition integral were made in \cite{haluskakolmogorov}, which will be commented below. Related integrals are considered in \cite{dunfordbilinearvectorintegral}, \cite{sivasankara1981vector}, \cite{chakraborty1}, \cite{chakraborty2}, \cite{charkraborty3}, \cite{chakraborty4}, \cite{chivukulasastry}, \cite{debievebilinear} and \cite{kappelerthesis}.

    \item Related bilinear vector integrals, specially based on tensor products, were developed by B. Jefferies, S. Okada, B. Ricker and I. Kluvanek, whose works are summarized in the books \cite{jefferieskac}, \cite{kluvanekintegration} and \cite{jefferies}. An important contribution of the integration procedure of these authors is to have a smaller space of integrable functions, but without supposing finite semivariation of the (operator valued) measures involved. This has important applications in mathematical physics as emphasized by \cite{jefferies}.
    
    \item The "integration school" of Bombal, Rodriguez-Salinas and Jimenez-Guerra in Spain developed similar procedures in bornological (locally convex) spaces with the publications of \cite{bombalintegralbornological}, \cite{bombalintegralradonnikodym}, \cite{rodriguezsalazar}, \cite{salinas1}, \cite{salinas2}, \cite{salinas3}, \cite{salinas4}, \cite{ballvebornological}, \cite{ballvefubini}, \cite{fidel}, \cite{fidelgeneral1}, \cite{fidelgeneral2}, \cite{bravo}, \cite{santiagohidalgo}. Basically, this is a type of integral in locally convex spaces which satisfy an Egorov approximation procedure.
    \item With a view towards stochastic integration, \cite{brooksdinculeanufirst} developed a bilinear integration theory based in a family of Banach-valued set functions, possibly finitely additive, which lead \cite{dinculeanustochastic} to a very general and applicable integration theory. A more recent development is included in the thesis of \cite{kozinski2005abstract}. 
    \item Following the ideas of \cite{drew3} in integration theory, W. Marik (in Germany) and J-C. Massé (in Canada) independently developed a bilinear integration theory for semigroup-valued measures and functions based on convergence in (a family of) submeasure(s). The initial developments are contained in the works \cite{marik1979integration} and \cite{masse}, and later in a reduced version in \cite{marikmonoidwertige} and \cite{massesemigroup}. A special case of this theory is the integration of \cite{heinich} and \cite{ohbaintegral}, already cited above, in groups based on uniform approximations. A very interesting development of this theory for conoids is given in the thesis of \cite{abreu2006integration}, and in Banach spaces by \cite{agafonova}. An independent axiomatic development of these types of submeasure based integrals with values in semigroups has also been developed by \cite{isakov}, as well as in the unpublished integration material in \cite{lechthesis}.
    
    \item Based on initial ideas of \cite{schafke1} and \cite{schafke2}, a very general integration procedure for lattice normed spaces, and even partially ordered semigroup-valued measures without any restriction on the semivariations, have been developed by \cite{schafkefubini}, \cite{schafkequasinormierten}, \cite{schafkeriemann}, \cite{schafkebook}, with ramifications in absolute Riemann integrability in \cite{gunzelrbook}, \cite{gunzler}, \cite{gunzlerimproper}, \cite{diazcarrilloaustralia}, \cite{diazcarrillohandbook}, \cite{diazcarrillocanada}. The most general form of these results are contained in the works of \cite{volkmerweber}, \cite{weberfreie1}, \cite{weberfreie2}, \cite{markwardtintegral}, \cite{markwardt1981teilweise} and \cite{hoffmannsemigroup}.
\end{enumerate}

In view of this great variety, it would take a more voluminously work to deal with all these integration procedures. We will privilege those that are more related to Riemmanian procedures constructed in the previous sections, and in order of complexity we shall construct the following scheme:

\begin{itemize}
    \item First, we will show that the Bartle-Dunford-Schwartz of \cite{drewlabuda} is contained in the net Riemann integral in the form of the $S^{*}$-partition integral in (sequentially complete) Hausdorff topological vector spaces.
    \item Then, we will briefly expose the results of \cite{haluskakolmogorov}, which shows that the bornological integral of Haluska, and therefore the special case of Dobrakov Banach-valued integral, is a special case of the $S^{*}$-partition integral in bornological locally convex topological vector spaces.
    \item Lastly, we construct some elements of the integration theory of \cite{massesemigroup}, showing it contains the related integral of \cite{marikmonoidwertige}, as well as the previous Lebesgue ones cited so far in this list, and prove partial results of the connection of this integration procedure with the net Riemann integral.
\end{itemize}

To not increase the technicalities involved, we will deal with the case of topological vector spaces, leaving (uniform) semigroups for a further work. Also, most of these theories have convergence theorems of dominated and Vitali type, but unless necessary, we will not formulate them explicitly. For every non-defined term concerning topological vector spaces, bornologies and convergences in them, see \cite{beerbornology}, \cite{schaefer}, \cite{radyno}, \cite{jarchow} and \cite{hogbenelend}.

\subsubsection{The Bartle-Dunford-Schwartz (BDS) in Topological Vector Spaces.} \label{bdsintegralsection}

In this section, we follow mainly the ideas and concepts of \cite{labudaconvexly} and \cite{drewlabuda}.

We fix, for this whole section, $(X, \tau)$ a Hausdorff topological vector space with a topology $\tau$ that is sequentially complete and has $\mathcal{U}$ any base of balanced zero-neighborhoods in $X$ and $Fs(X)$ any base of continuous $F$-seminorms which define its topology. Fix also $\mathcal{H}$  a $\sigma$-algebra of subsets of a (non-empty) set $\Omega$.

The main definition for measures we shall use is the following:

\begin{definition} [Convexly Bounded Vector Measures]

We say that a vector measure (that is, a $\sigma$-additive set function) $\mu: \mathcal{H} \rightarrow X$ is convexly bounded if the convex hull of its image is bounded\footnote{In the bornological, or Von Neumann, sense.}.
\end{definition}

An equivalent definition which we shall use employs the continuity properties of simple-integral operators:

\begin{proposition}
    A vector measure $\mu: \mathcal{H} \rightarrow X$ is convexly bounded if and only if the operator:

    \[\int_{\Omega}\cdot \ d\mu: (S_{\infty}, ||\cdot||_{\infty}) \rightarrow (X,\tau),\]

    given by:

    \[\int_{\Omega}sd\mu = \sum_{i=1}^{n}s_{i}\mu(A_{i}),\]

    is continuous, where $(S_{\infty}, ||\cdot||_{\infty})$ is the (Banach) space of simple functions equipped with the supremum norm.
\end{proposition}

The proof follows directly from the relation between continuity of operators and bornologies between two topological vector spaces. 

In this case, we get the connection of finite semivariation defined previously:

\begin{proposition} [Convex Boundedness Implies Finite Semivariation] \label{finitesemivariationconvexlybounded}
    If a vector measure $\mu: \mathcal{H} \rightarrow X$ is convexly bounded, then it has finite semivariation in the sense of Definition \ref{finitesemivariationmillington}.
\end{proposition}

With these properties in mind, we now define a sequence of objects in view of the BDS integral of \cite{Labudaexh}. For now on in this section, $\mu$ will be a convexly bounded vector measure.

In this case, let, for $d(\cdot)$ an $F$-seminorm in $Fs(X)$ and $A \in \mathcal{H}$, 

\[d(\mu)(A) = \sup\{d(\mu(B)): B \in \mathcal{H}, B \subseteq A\},\]

the submeasure-majorant associated to $\mu$ and $d$. It's readily seen that the submeasure-majorant is an order continuous (real) $\sigma$-subadditive set function defined on $\mathcal{H}$.

We then say that a real function $f: \Omega \rightarrow \mathbb{R}$ is $\mu$-measurable if it is equal, except on a $\mu$-null set\footnote{That is, an $A \in \mathcal{H}$ such that $\mu(B) = 0$ in $Y$ for every $B \subseteq A$ such that $B \in \mathcal{H}$.}, to a (Borel) $\mathcal{H}$-measurable function $g$. We denote the (quotient space by $\mu$-null sets of the) space of all such $\mu$-measurable functions by $L^{0}(\mu)$. We equip this space with the Hausdorff vector topology of convergence in (sub)measure $\mu$ (see \cite{labudaconvexly} for details).

In this case, we now have the main definition from \cite{drewlabuda}:

\begin{definition} [Bartle-Dunford-Schwartz Integral]

A function $f \in L^{0}(\mu)$ is said to be Bartle-Dunford-Schwartz (BDS) integrable if there exists a sequence $\{f_{n}\}_{n \in \mathbb{N}}$ of simple functions such that $f_{n} \rightarrow f$ $\mu$-a.e. and, for each $A \in \mathcal{H}$, $\lim_{n \rightarrow \infty}\int_{A}fd\mu$ exists in $X$, where each $\int_{A}f_{n}d\mu$ is defined additively.

Then, by definition, 

\[\int^{(BDS)}_{A}fd\mu = \lim_{n \rightarrow \infty}\int_{A}fd\mu.\]
    
\end{definition}

This integral has the usual additivity properties in the space of integrable functions, and also good continuity properties (see sections $3$ and $4$ of \cite{drewlabuda}). 

\begin{observation} [The BDS integral Ignores $\mu$-null sets] \label{ignoresnullsetsBDS}

By the definition of the $BDS$-integral, we have that, if $A \in \mathcal{H}$ is a null set, then for every $f \in L^{0}(\mu)$,

\[\int_{A}^{(BDS)}fd\mu = \int_{\Omega}^{(BDS)}f\mathbbm{1}_{A}d\mu = 0.\]

In this case, if $f$ is $BDS$-integrable and $A$ is the $\mu$-null set on which the convergence of $\{f_{n}\}_{n \in \mathbb{N}}$ occurs outside it, we have (see also more details in \cite{drewlabuda}):

\[\int_{B}^{(BDS)}fd\mu = \int_{B}^{(BDS)}f\mathbbm{1}_{A^{c}}d\mu,\]

for every $B \in \mathcal{H}$. As the same property of integration on null sets is valid for the $S^{*}$-integral by the results already stated and used in Section \ref{convergenceoutsidenullset}, we may suppose without loss of generality that $f$ is, itself, Borel-measurable to prove the equivalence of the two integrals. 
    
\end{observation}

Considering the above observation, to prove that the BDS integral is a special case of a type of the net Riemann integral, we shall use the following simplification of the decomposition of (strongly) measurable functions given in \cite[Lemma 5.1, p. 201]{musialtopics}:

\begin{proposition} \label{musialdecomposition}
    If $f: \Omega \rightarrow \mathbb{R}$ is a Borel measurable function, then there exits a bounded (Borel) measurable function $g: \Omega \rightarrow \mathbb{R}$ and a(n) (elementary and Borel) measurable function $h: \Omega \rightarrow \mathbb{R}$ of the form $\sum_{i=1}^{\infty}x_{n}\mathbbm{1}_{A_{i}}$, with pairwise disjoint $A_{i} \in \mathcal{H}$ such that:

    \[f = g + h.\]
\end{proposition}

As is well known that each bounded real $\mu$-measurable function is a uniform limit of simple functions defined on $\mathcal{H}$ (see chapter 1 of \cite{dinculeanustochastic}) and that $\mu$ is of finite semivariation by Proposition \ref{finitesemivariationconvexlybounded}, we get immediately by Theorem \ref{quasiuniformconvergencetheorem}, Observation \ref{ignoresnullsetsBDS} and the fact that the $S^{*}$-partition integral equals de BDS-integral on simple functions that:

\begin{proposition} \label{bdsbounded}
    Let $g: \Omega \rightarrow \mathbb{R}$ a bounded $\mu$-measurable function. Then, $g$ is $S^{*}$-partition integrable and its integral equals the $BDS$-integral of $g$\footnote{Which is always BDS-integrable by construction.}.
\end{proposition}

Now, we need the next result which is contained in \cite[Remarks 3.3. (3), p. 627]{drewlabuda}:

\begin{proposition} \label{bdselementary}
    Let $h: \Omega \rightarrow \mathbb{R}$ be a(n) (elementary) $\mu$-measurable function  of the form $\sum_{i=1}^{\infty}x_{n}\mathbbm{1}_{A_{i}}$, with pairwise disjoint $A_{i} \in \mathcal{H}$ which is BDS-integrable. Then, 

    \[\int_{A}^{(BDS)}hd\mu= \sum_{i=1}^{\infty}x_{i}\mu(A \cap A_{i}),\]

    for each $A \in \mathcal{H}$ where the series is unconditionally summable in $X$.
\end{proposition}

Combining these results, we get:

\begin{theorem} [BDS integrability implies $S^{*}$-partition integrability] \label{labudadrewriemann}
    Let $f \in L^{0}(\mu)$ be a BDS-integrable function. Then, $f$ is $S^{*}$-partition integrable and:

    \[\int_{A}^{S^{*}}fd\mu = \int_{A}^{(BDS)}fd\mu, \]

    for each $A \in \mathcal{H}$.
\end{theorem}

\begin{proof}
    As $f$ is $\mu$-measurable, by Observation \ref{ignoresnullsetsBDS}, we may assume that $f$ is (Borel) measurable. Therefore, by Proposition \ref{musialdecomposition}, there exits a bounded measurable function $g: \Omega \rightarrow \mathbb{R}$ and a(n) (elementary) measurable function $h: \Omega \rightarrow \mathbb{R}$ of the form $\sum_{i=1}^{\infty}x_{n}\mathbbm{1}_{A_{i}}$, with pairwise disjoint $A_{i} \in \mathcal{H}$ such that:

    \[f = g + h,\]

   As $f$ is BDS-integrable and $g$ is bounded, and therefore $BDS$-integrable by Proposition \ref{bdsbounded} with BDS-integral equal to its $S^{*}$-partition integral, it follows that $h$ is BDS-integrable and, by Proposition \ref{bdselementary},

    \[\int_{A}^{(BDS)}hd\mu= \sum_{i=1}^{\infty}x_{i}\mu(A \cap A_{i}),\]

    for $A \in \mathcal{H}$ and with the right-hand series unconditionally summable in $X$.

    By the same techniques as in the Pavlakos integral of elementary functions in Proposition \ref{equivalencefirstelement}, we get, using the unconditional summability of the above series, that $h$ is $S^{*}$-partition integrable and:

    \[\int_{A}^{S^{*}}hd\mu= \sum_{i=1}^{\infty}x_{i}\mu(A \cap A_{i}).\]

    Combining with the results on $g$ and using the additivity of each integral, we get the result. 
\end{proof}

Therefore, by Proposition \ref{sandsionintegrals}, we get that the BDS-integral is a special case of the $S^{*}$-partition and Sion integrals, and therefore a special case of the Net Riemann integral. All the results above can be developed for $\delta$-rings in the place of $\sigma$-algebras, as detailed in \cite{thomas} for the BDS-integral with the use of Theorem \ref{uniformconvergencetheoremtopological} accounting for the $\delta$-ring. One consequence is that the (stochastic) integral of \cite{schwartz2006semi}, as well as the integrals of \cite{masaniniemi1}, \cite{masaniniemi2} and \cite{masaniniemi3}, are contained as special cases of the Riemmanian definition of integration given by the net Riemann integral.

We now study the next point in relation to \cite{haluska1} bornological integral.

\subsubsection{The Haluska and Rodriguez-Salazar integrals in Bornological Locally Convex Spaces.}

For this section, we mainly follow the article \cite{haluskakolmogorov} for the basic elements in establishing the bornological integral. 

In this case, we let $X,Y$ be two complete bornological locally convex topological vector spaces with respect to the field of real numbers, with bornologies $\mathcal{B}_{X}$ and $\mathcal{B}_{Y}$. Let also be $\mathcal{U}$ and $\mathcal{W}$ the set of all Banach disks in $X$ and $Y$ respectively, in such a way that $X$ and $Y$ are inductive limits of (separable) Banach spaces $X_{U}$ and $Y_{W}$,

\[X = inj\lim_{U \in \mathcal{U}}X_{U},\]

\[Y = inj\lim_{W \in \mathcal{W}}Y_{W},\]

We equip both spaces with the net convergence structure of Mackey convergence on each Banach disk (see \cite{convergencestructuresvanderwalt} for details of the net convergence structures associated to bornologies). 

We also denote $p_{U}: X \rightarrow [0, \infty]$ the Minkowski functional associated to each member $U \in \mathcal{U}$ (similarly for $\mathcal{W}$), and by $L(X,Y)$ the space of all continuous linear operators between $X$ and $Y$.

Further, for this section, let $\mathcal{H}$ be a $\delta$-ring of subsets of $\Omega$ and $\sigma(\mathcal{H})$ the $\sigma$-algebra generated by $\mathcal{H}$.

In this case, we may define a central notion for the theory of bilinear integral of Dobrakov type:

\begin{definition} [Bilinear Semivariation and Sets of Finite Semivariation]

Let $\mu: \mathcal{H} \rightarrow L(X,Y)$ be a finitely additive operator valued measure. For $(U,W) \in \mathcal{U} \times \mathcal{W}$, we denote by $\hat{\mu}_{U,W}$ the $(U,W)$-semivariation of $\mu$, that is, the order continuous (extended real) submeasure (defined on $\mathcal{H}$) given by:

\[\hat{\mu}_{U,W}(A) = \sup\{p_{W}(\sum_{i=1}^{n}\mu(A \cap A_{i})x_{i}): \{x_{i}\}_{i=1}^{n} \in U; \{A_{i}\}_{i=1}^{n} \ \text{disjoint in} \ \mathcal{H}, n \in \mathbb{N}\}.\]

We then denote by $\mathcal{H}_{U,W}$ the largest $\delta$-ring of sets from $\mathcal{H}$ which are of finite $\hat{\mu}_{U,W}$ semivariation. Also, $\hat{\mu}_{\mathcal{U},\mathcal{W}}$ will be the set of all semivariations $\hat{\mu}_{U,W}$.

\end{definition}

We now need some analogous constructions ,in the bornological case, of operator valued measures that are $\sigma$-additive in the strong operator topology of a Banach space. For that, we need some further notions related to operators between bornological spaces. 

In this context, let $\Phi$ be the class of all functions defined on $\mathcal{U}$, and with values in $\mathcal{W}$, equipped with the partial order given by: for $\phi,\zeta \in \Phi$, we say that $\phi \leq \zeta$ if and only if $\phi(U) \subseteq \zeta(U)$ for each $U \in \mathcal{U}$. We also suppose that $\Phi \subseteq L(X,Y)$.

We then have the next measure related structures from \cite{haluskakolmogorov}:

\begin{definition} [Countable Additivity and Bornologies]
    For $\mu: \mathcal{H} \rightarrow L(X,Y)$  a finitely additive operator valued measure, we say that:

    \begin{enumerate} [(i)]
        \item  For $(U,W) \in \mathcal{U} \times \mathcal{W}$, $\mu$ is of $\sigma$-finite $(U,W)$-semivariation if there exists a sequence $\{A_{n}\}_{n \in \mathbb{N}}$ in $\mathcal{H}_{(U,W)}$ such that $\Omega = \bigcup_{n=1}^{\infty}A_{n}$.
        \item For $\phi \in \Phi$, $\mu$ is of $\sigma_{\phi}$-finite $(\mathcal{U},\mathcal{W})$-semivariation if for every $U \in \mathcal{U}$, $\mu$ is of $\sigma$-finite $(\mathcal{U}, \phi(\mathcal{U}))$-semivariation.
        \item $\mu$ is of finite $\sigma_{\phi}$-finite $(\mathcal{U},\mathcal{W})$-semivariation if there exists a $\phi \in \Phi$ such that for every $U \in \mathcal{U}$ the operator valued measure $\mu$ is of $\sigma_{\phi}$-finite $(\mathcal{U},\mathcal{W})$-semivariation.
        \item $\nu: \sigma(\mathcal{H}) \rightarrow Y$ finite additive is $(W,\sigma)$-additive, for $W \in \mathcal{W}$, if $\nu$ is countable additive in $Y_{W}$. We then say that $\nu$ is $(\mathcal{W},\sigma)$-additive if there exists $W \in \mathcal{W}$ such that $\nu$ is $(W,\sigma)$-additive. 
        \item Let $\phi \in \Phi$. We say that $\mu$ is $\sigma_{\phi}$-additive if $m$ is of finite $\sigma_{\phi}$-finite $(\mathcal{U}, \mathcal{W})$-semivariation and, for every $A \in \mathcal{H}_{U,\phi(U)}$, the finite additive set function $\mu(A \cap \cdot)x: \sigma(\mathcal{H}) \rightarrow Y$ is a $(\phi(U),\sigma)$-additive for every $x \in X_{U}$, $U \in \mathcal{U}$. We then say that $\mu$ is $\sigma_{\Phi}$-additive if there exists $\phi \in \Phi$ such that $\mu$ is $\sigma_{\phi}$-additive.
        \item Let $W \in \mathcal{W}$. A sequence of $(W, \sigma)$-additive set functions $\nu_{n}: \sigma(\mathcal{H}) \rightarrow Y$, $n \in \mathbb{N}$, is said to be uniformly $(W,\sigma)$-additive on $\sigma(\mathcal{H})$ if for every $\varepsilon > 0$, $A \in \sigma(\mathcal{H})$ such that $p_{W}(\nu_{n}(A)) < \infty$, $A_{i} \in \sigma(\mathcal{H})$ disjoint, there exists $k_{0} \in \mathbb{N}$ such that for every $k \geq k_{0}$, 

        \[p_{W}(\nu_{n}(\bigcup_{i=k+1}^{\infty}(A_{i} \cap E))) \leq \varepsilon,\]

       uniformly for every $n \in \mathbb{N}$.
    \end{enumerate}
\end{definition}

The last step to define the integral is to define some structures related to a class of "measurable" functions in this abstract (operator valued) setting. We again collect the relevant definitions from \cite{haluskakolmogorov}:

\begin{definition} [Measurable Functions and Bornologies]

For $A \in \mathcal{H}$ and $\mu: \mathcal{H} \rightarrow L(X,Y)$  a finitely additive operator valued measure, we say that:

\begin{enumerate} [(i)]
    \item For $R \in \mathcal{U}$, a sequence of functions $f_{n}: \Omega \rightarrow X$, $n \in \mathbb{N}$, $(R,A)$-converges $\hat{\mu}_{U,W}$-a.e. to a function $f: \Omega \rightarrow X$ if:

    \[\lim_{n \rightarrow \infty}p_{R}(f_{n}(\omega)-f(\omega)) = 0,\]

    for every $\omega \in A \setminus N$, where $N$ is a $\hat{\mu}_{U,W}$-null set. 
    \item  A sequence of functions $f_{n}: \Omega \rightarrow X$, $n \in \mathbb{N}$, $(\mathcal{U},A)$-converges $\hat{\mu}_{\mathcal{U},\mathcal{W}}$-a.e. to a function $f: \Omega \rightarrow X$ if there exists $R \in \mathcal{U}$ and $(U,W) \in \mathcal{U} \times \mathcal{W}$ such that the sequence $f_{n}$ $(R,A)$-converges $\hat{\mu}_{U,W}$-a.e. to f. In this case, we write $f = \mathcal{U}-\lim_{n \rightarrow \infty}f_{n}$ $\hat{\mu}_{U,W}$-a.e.
    \item We say that a function $f: \Omega \rightarrow X$ is $\mathcal{U}$-measurable function if there exists a $R \in \mathcal{U}$ such that for every $R \subseteq U \in \mathcal{U}$ and $\delta > 0$

    \[\{\omega \in \Omega: p_{U}(f(\omega)) \geq \delta\} \in \sigma(\mathcal{H}).\]

    We denote by $\mathbb{M}_{\mathcal{U}}$ the set of all $\mathcal{U}$-measurable functions.
    \item A function $f: \Omega \rightarrow X$ is $\mathcal{H}$-simple if it has a finite range and $f^{-1}(x) \in \mathcal{H}$ for each $x \in X \setminus \{0\}$. We denote the space of all $\mathcal{H}$-simple functions by $\mathcal{S}$. On the other hand, we say that a function $f: \Omega \rightarrow X$ is $\mathcal{H}_{U,W}$ simple, for $(U,W) \in \mathcal{U} \times \mathcal{W}$ if:

    \[f = \sum_{i=1}^{n}x_{i}\mathbbm{1}_{A_{i}},\]

    where $\{x_{i}\}_{i=1}^{n} \in X_{U}$, $\{A_{i}\}_{i=1}^{n}$ is a disjoint sequence of sets in $\mathcal{H}_{U,W}$ and $n \in \mathbb{N}$. We denote the space of these functions by $\mathcal{S}_{U,W}$. Then, a function is said to be $\mathcal{H}_{\mathcal{U},\mathcal{W}}$ simple if there exists $(U, W) \in \mathcal{U} \times \mathcal{W}$ such that $f \in \mathcal{S}_{U,W}$. The space of all $\mathcal{H}_{\mathcal{U},\mathcal{W}}$ simple functions will be denoted by $\mathcal{S}_{\mathcal{U}, \mathcal{W}}$.
\end{enumerate}
    
\end{definition}

We now give the main definition of the integral first presented in \cite{haluska1}, which is a generalization of the Banach space operator-valued measure of \cite{dobrakov1}:

\begin{definition} [Haluska Integral]

Let $\mu: \mathcal{H} \rightarrow L(X,Y)$ be a $\sigma_{\phi}$-additive operator valued measure. A function $f \in \mathbb{M}_{\mathcal{U}}$ is said to be $\mathcal{H}_{\mathcal{U},\mathcal{W}}$-integrable, or Haluska integrable, and we write $f \in \mathcal{I}_{\mathcal{U},\mathcal{W}}$, if:

\begin{enumerate} [(i)]
    \item There exists a sequence $f_{n} \in \mathcal{S}_{\mathcal{U},\mathcal{W}}$, $n \in \mathbb{N}$, of (simple) functions such that:

    \[\mathcal{U}-\lim_{n \rightarrow \infty}f_{n} =f \ \hat{\mu}_{\mathcal{U},\mathcal{W}}-a.e.,\]

    \item The integrals (defined, as usually, additively) $\int_{A}f_{n}d\mu$, $n \in \mathbb{N}$, are uniformly $(\mathcal{W},\sigma)$-additive measures on $\sigma(\mathcal{H})$.
\end{enumerate}

Then, the (Haluska) integral of the function $f \in \mathcal{I}_{\mathcal{U},\mathcal{W}}$ on a set $A \in \sigma(\mathcal{H})$ is defined by the equality:

\[\int_{A}^{(H)}fd\mu = \mathcal{W}-\lim_{n \rightarrow \infty}\int_{A}f_{n}d\mu.\]
    
\end{definition}

The $S^{*}$-partition integral of Definition \ref{kolmogorovS*integralconvergencestructure} takes the following form in this context:

\begin{definition} [$S^{*}$-Partition Integral and Bornologies]

Let $(U,W) \in \mathcal{U} \times \mathcal{W}$. A function $f: \Omega \rightarrow X$ is $S^{*}_{U,W}$-partition integrable on $A \in \mathcal{H}$ if, for $f_{U}: \Omega \rightarrow X_{U}$, there exists $y \in Y_{W}$ such that for every $\varepsilon > 0$ there exists a $P_{0} \in \mathcal{P}(A)$ consisting of $\mathcal{H}_{U,W}$ sets such that, for every $P \in \mathcal{P}(A)$ consisting of $\mathcal{H}_{U,W}$ sets and which satisfies $P \geq P_{0}$, the series $\sum_{\alpha \in P}\mu(\alpha)f_{U}(\delta_{\alpha})$ is unconditionally convergent in $Y_{W}$ for every choice function $\delta$ and:

\[p_{W}(\sum_{\alpha \in P}\mu(\alpha)f_{U}(\delta_{\alpha}) - y) \leq \varepsilon.\]
    
\end{definition}

We then have the main result of the article \cite{haluskakolmogorov}, which gives the equivalence between the $S^{*}$-partition integral above and the Haluska integral, and which generalizes the corresponding result of \cite{dobrakov7} for the Banach-valued case:

\begin{theorem} [$S^{*}$-Partition Integral is Equivalent to the Haluska Integral]

A $\mathcal{H}_{\mathcal{U},\mathcal{W}}$ measurable function $f: \Omega \rightarrow X$ is Haluska integrable on $\sigma(\mathcal{H})$ if and only if it is $S^{*}$-partition integrable on $\sigma(\mathcal{H})$ and, in this case,

\[\int_{A}^{(H)}fd\mu = \int_{A}^{S^{*}}fd\mu.\]
\end{theorem}

Using the class of Dobrakov measurable functions, which is a special class of the class of measurable functions defined above, we get that, from the last theorem, $X,Y$ being Banach spaces,

\begin{corollary} [$S^{*}$-Partition Integral is Equivalent to the Dobrakov Integral] \label{dobrakovandsion}

A Dobrakov measurable function $f: \Omega \rightarrow X$ is Dobrakov integrable on $\sigma(\mathcal{H})$ if and only if it is $S^{*}$-partition integrable on $\sigma(\mathcal{H})$ and, in this case,

\[\int_{A}^{(D)}fd\mu = \int_{A}^{S^{*}}fd\mu,\]

where $\int_{A}^{(D)}fd\mu$ is de Dobrakov integral of $f$ on $A$.

\end{corollary}

These two equivalences results gives interesting further properties of the net Riemann integral, and its special (topological) case of the $S^{*}$-partition integral, in the context of bilinear integration in locally convex spaces, specially Banach and bornological ones. For further interesting convergence theorems results, we refer the reader to the articles cited above. 

We should also comment that, as noticed by \cite{haluska2}, his integral contains the Egorov type integral of the Spanish school constructed by \cite{rodriguezsalazar}, thus giving another interesting integration procedure related to the net Riemann integral. Interestingly, we also may have proved this result directly by looking at the constructions in \cite{rodriguezsalazar} or \cite{fidelgeneral2} and noticing that it reduces to approximating functions by countable valued functions in the regime of quasi-uniform convergence (always supposing finite semivariation of the operator valued-measures involved), thus the equivalence can be proven using Theorem \ref{quasiuniformconvergencetheorem}. We don't pursue this construction for brevity.

\subsubsection{The Submeasure Integral of Massé.}

We now study the last point in the integration equivalences. This is the integral of \cite{massesemigroup}, which contains as special case the integral of \cite{marikmonoidwertige}. We mainly follow the first cited article, and we also use the submeasure terminology originally due to \cite{drew1}.

For this section, unless stated to the contrary, we consider $X,Y,Z$ Hausdorff topological vector spaces each of which has the net convergence structure induced by its topology, which we assume are complete and generated by a base of balanced zero-neighborhoods, denoted in each case by $\mathcal{U}_{X},\mathcal{U}_{Y}, \mathcal{U}_{Z}$. Also, $Fs(X)$ well denote any base of continuous $F$-seminorms which define the topology of $X$. In this case, with a bilinear product, we fix an integration structure given by Definition \ref{integrationstructure}.

Also, in this section, $\mathcal{H}$ will be an algebra of subsets of $\Omega$, and we denote by $\Delta$ the operation of symmetric differences on sets. In this case, we define:

\[\mathbb{S} = \{\xi_{i}\}_{i \in I},\]

to be a non-empty collection of (extended) real valued submeasures defined on $\mathcal{H}$. In this case, we equip $(2^{\Omega}, \Delta, \cap)$ with the topology induced by $\mathbb{S}$ for which it is a topological ring (with operations of symmetric difference and intersection). We call this topology $\tau_{\mathbb{S}}$. Then, $\mathcal{H}(\mathbb{S})$ will denote the sub-ring of this last topological ring with the subspace topology.

We now define an essential concept due to \cite{massesemigroup}, which generalizes the idea of convergence in measure:

\begin{definition} [Convergence in a Family of Submeasures]

We say that a net of functions $\{f_{l}\}_{l \in L}$, each defined on $\Omega$ and with values in $X$, converges in submeasure to $f: \Omega \rightarrow X$ if and only if:

\[\xi_{i}(\{\omega \in \Omega: ||f_{l}(\omega)-f(\omega)||_{d}\}) \rightarrow 0,\]

for each $d \in Fs(X)$ and $i \in I$, where $||\cdot||_{d}$ is the $F$-norm induced on $X$ by $d$. The (vector\footnote{See section 1 of \cite{massesemigroup}.}) topology induced by this convergence in $X^{\Omega}$, which has the base of neighborhoods of the form $\{W_{K}(U,\varepsilon)\}_{U \in \mathcal{U}_{X}, K \subseteq \mathcal{F}(\mathbb{N}), \varepsilon >0}$, where:

\[W_{K}(U,\varepsilon) =\{g \in X^{\Omega}: \xi_{k}(\{\omega \in \Omega: f - g \notin U\}) \leq \varepsilon,\ k \in K\}, \]

is called the topology of convergence in $\mathbb{S}$-submeasure. 
    
\end{definition}

We now define a structure that consists, in a generalized sense, on the sets for which we may apply the family $\mathbb{S}$ (one example would be the sets of finite semivariation in the previous section):

\begin{definition} [$\mathbb{S}_{\mu}$-Integrable Sets]

For a finitely additive set function $\mu: \mathcal{H} \rightarrow Y$ and the family of submeasures $\mathbb{S}$, we suppose the existence of a sub-ring of $\mathcal{H}$, which we denote by $\mathcal{H}^{\mu}_{\mathbb{S}}$, and call the family of $\mathbb{S}_{\mu}$-integrable sets.  
\end{definition}

Associated to these structures, we have the main two hypothesis of \cite{massesemigroup}, which are imposed in order to give some regularity to the integration procedure, and we assume for the rest of the section. Basically, they ensure the to-be-defined integration procedure is continuous in the space of simple functions with uniform convergence:

\begin{definition} [Massé's Hypothesis]

We suppose that the integration structure of this section satisfies the following two hypothesis:

($M_{1}$) The integration operator acting on simple functions, based on $\mathcal{H}^{\mu}_{\mathbb{S}}$, 

\[\int_{\Omega}\cdot \ d\mu: (S^{\mathcal{H}^{\mu}_{\mathbb{S}}}_{\infty}, \tau_{\infty}) \rightarrow (Z,\tau),\]

    given by:

    \[\int_{\Omega}sd\mu = \sum_{i=1}^{n}s_{i}\mu(A_{i}),\]

    is continuous, where $(S^{\mathcal{H}^{\mu}_{\mathbb{S}}}_{\infty}, \tau_{\infty})$ is the space of simple functions, based on $\mathcal{H}^{\mu}_{\mathbb{S}}$, with the topology of uniform convergence induced on $X^{\Omega}$ denoted by $\tau_{\infty}$.

($M_{2}$) For every $x \in \mathcal{X}$, 

\[\lim_{\mathcal{H}^{\mu}_{\mathbb{S}}, E \rightarrow \emptyset}x\mu(E) = 0,\]

where the limit is taken in  $\mathcal{H}^{\mu}_{\mathbb{S}}$ equipped with the convergence in $\mathbb{S}$-submeasure topology induced by indicator function convergence. 
\end{definition}

We now give the main integral of \cite{massesemigroup}:

\begin{definition} [Massé Integral]

A function $f: \Omega \rightarrow X$ is said to be Massé-integrable with respect to a finitely additive set function $\mu: \Omega \rightarrow Y$ if there exists a net of $X$-valued simple functions based on $\mathcal{H}^{\mu}_{\mathbb{S}}$, denoted by $\{f_{l}\}_{l \in L}$, such that $f_{l} \rightarrow f$ in $\mathbb{S}$-submeasure and such that the net of additively-defined integrals $\{\int_{A}^{(M)}f_{l}d\mu\}_{l \in L}$ is Cauchy in $Z$ uniformly for $E \in \mathcal{H}$ . 

In this case, we define the Massé-integral of $f$ on $A \in \mathcal{H}$ by:

\[\int_{A}^{(M)}fd\mu = \lim_{l \in L}\int_{A}^{(M)}f_{l}d\mu.\]
    
\end{definition}

We refer the reader to section 2 of \cite{massesemigroup} for the facts that this integral is well-defined, is an homomorphism in the space of integrable functions, and also for Vitali type convergence theorems.

For the equivalence with some version of the Net Riemann integral, matters are more complicated than the results above. The first thing to notice in this regard is that the hypothesis $(M1)$ and $(M2)$ are not necessary to define the net Riemann integral, or the $S^{*}$-integral, which is contained in it, in this case. In fact, we may find Sion/$S^{*}$-integrable functions, with respect to a finitely additive measure of non-finite semivariation, which are not Massé integrable (see \cite{marikmonoidwertige}). On the other hand, both may exist and not be equal, as the following example indicate:

\begin{example} [Sion Integrable, but not Massé Integrable, Function]

Let $\Omega = \mathbb{N}$, $\mathcal{H}$ the algebra of cofinite sets in $\Omega$, and take $\mu(A)$ (the finitely, but not $\sigma$,additive set function) to be $0$ when $A$ is finite, and $1$ when $A^{c}$ is finite. Then, it is straightforward to see that the function $f = \mathbbm{1}_{\mathbb{N}}$ is Sion integrable, with integral $0$, and Massé integrable, with integral $1$.
\end{example}

In fact, this example gives an indication of a much wider phenomena, as can be seen from the next proposition, which is a consequence of Theorems \ref{sandsionintegrals} and \ref{sigmaadditives*integral}:

\begin{proposition}
    Let $f$ be a Massé and $S^{*}$/Sion-integrable function with respect to a finitely additive measure $\mu: \mathcal{H} \rightarrow Y$ such that $\nu_{f}: \mathcal{H} \rightarrow Y$, the set function induced by the Massé integral of $\mathcal{H}$, is not $\sigma$-additive on $\mathcal{H}$. Then, the Massé and $S^{*}$/Sion-integrals of $f$ are not equal.
\end{proposition}

In other words, from the fact that the $S^{*}$/Sion-partition integrals always project a measure intro its $\sigma$-additive part we may, for example, picking a purely finitely-additive set function, always produce functions and set functions for which these integrals disagree. 

Therefore, we need further hypothesis, as in the sections above, to prove some equivalence results concerning these integrals, or some other interesting versions of the net Riemann integral. Nevertheless, we couldn't find a single general condition on the Sion or $S^{*}$-integrals, for example, that gives equivalence with Massé's one,  something that is difficult given the radically different "natures" of these two integrals, as argued by Massé in the epilogue of his thesis (see \cite{masse}), and in the penultimate paragraph of Section 2 of \cite{massesemigroup}, as well as by the example and result above. 

One possible approach for the resolution of this problem  may be to reduce it to a question of which type of convergence in submeasure theorem for the Net Riemann integral, or its special cases, is valid. As said in our section about convergence theorems, this is left for a future work, and we therefore leave the general version of the equivalence as an open problem. 

The second approach, without asking the $\sigma$-additivity of $\mu$, may be to look for integrals in which the projection property into the $\sigma$-additive part is not automatic, and restrict the class of functions for which some type of convergence theorem is valid. In fact, by $(M1)$ we know that every function which is a uniform limit of a net of simple functions based $\mathcal{H}^{\mu}_{\mathbb{S}}$ is Massé integrable. To follow this approach, we need the next result, which we state for uniform spaces (see \cite{pageuniform}, or the other topology references above, for non-defined terms):

\begin{lemma} [Uniform Approximation in Totally Bounded Uniform Spaces]
    Let $f: \Omega \rightarrow X$ be a map from a paved algebra $(\Omega, \mathcal{H})$ to a totally bounded uniform space $(X, \mathcal{U})$. Then, $f$ is the uniform limit of a net $\{f_{i}\}_{i \in I}$ of simple functions $f_{i}: \Omega \rightarrow X$. If $f^{-1}(U[\cdot]) \in \mathcal{H}$ for every $U \in \mathcal{U}$ \footnote{Or a (sub)base for this uniformity, which is equivalent.}, then we can choose the net $\{f_{i}\}_{i \in I}$ to be measurable with respect to $\mathcal{H}$.
\end{lemma}

\begin{proof}
    As $X$ is totally bounded, for each $U \in \mathcal{U}$ there exists a finite collection of points $\{x_{k}\}$, $k = 1, \dots, n_{U}$, $n_{U} \in \mathbb{N}$ - which depends on the entourage - such that:

    \[X = \bigcup_{k=1}^{n_{U}}U(x_{k}).\]

    Now, let $A_{k}^{U} = f^{-1}(U(x_{k}) - \bigcup_{j=1}^{k-1}U(x_{j}))$, for $k=2, \dots, n_{U}$, and $A_{1}^{U} = f^{-1}(U(x_{1}))$. Then, for each (fixed) $U \in \mathcal{U}$, the sets $A_{k}^{U}$ are disjoint and $\Omega = \bigcup_{k=1}^{n_{U}}A_{k}^{U}$. Therefore, define the simple function given by:

    \[f_{U}(\omega) = \sum_{k=1}^{n_{U}}x_{k} \mathbbm{1}_{A_{k}^{U}}(\omega),\]

    which, by definition, have $f(\omega) \in U(f_{U}(\omega))$, for each $U \in \mathcal{U}$, and $\omega \in \Omega$. Therefore, by definition, we have the uniform convergence $f_{U} \rightarrow f$, when $\mathcal{U}$ with the subset (downwards!) directed order, and therefore $\{f_{U}\}_{U \in \mathcal{U}}$ being a well defined net. Also, when $f$ satisfies the preimage condition $f^{-1}(U) \in \mathcal{H}$ for every $U \in \mathcal{U}$, then the support sets of each $f_{U}$ are in $\mathcal{H}$, and therefore all of the approximation maps are measurable with respect to the algebra $\mathcal{H}$. 
\end{proof}

Therefore, from this last result, and the fact that the $S$-Kolmogorov integral (see Example \ref{KolmogorovSintegralconvergencestructures}) agrees with the Massé integral on simple functions and satisfy the hypothesis of the (uniform convergence) Theorem \ref{uniformconvergencetheoremtopological}, we have:

\begin{corollary} [Massé and $S$-Kolmogorov Integrals] 
\label{masseandskolmogorov}

Suppose that $f: \Omega \rightarrow X$ is a totally bounded function such that $f^{-1}(O) \in \mathcal{H}^{\mu}_{\mathbb{S}}$ for each open set $O \subseteq X$. Then, $f$ is $S$-Kolmogorov and Massé integrable, and the two integrals agree. 
    
\end{corollary}

For a set of minimal, but strong, hypothesis, this seems the most general result we have without supposing further hypothesis on the additivity of $\mu$.

If we allow $\mu$ to be $\sigma$-additive, then is is readily seen that we can change the integral in Corollary \ref{masseandskolmogorov} to be the $S^{*}$/Sion integral and get the same result. Also, in this case, we now study a different set of restrictive hypothesis that make the referenced integrals equal . For that, we first need \cite[Lemma 6.11, p. 30]{masse}, and then \cite[Theorem 9.7, p. 60]{masse}, combined with \cite[Theorem 6.4-6.6, p. 27]{masse}:

\begin{theorem} \label{massetheoremthesis}
    Let $\mathbb{S}$ be a non-empty family of (extended real) submeasures in a $\sigma$-algebra $\mathcal{H}$, such that it possesses a control submeasure\footnote{For this definition, see \cite{drewcontrol}.}. Suppose also that $X$ and $Z$ are metrizable. Then, a function $f:\Omega \rightarrow X$ is Massé integrable, with respect to the family $\mathbb{S}$, if and only if there is a sequence of simple functions based on $\mathcal{H}^{\mu}_{\mathbb{S}}$ such that:

    \begin{enumerate} [(i)]
        \item There exists a countable collection of subsets of $\Omega$, all elements of $\mathcal{H}$, $\{F_{k}\}_{k \in \mathbb{N}}$, such that $\Omega \setminus \bigcup_{k \in \mathbb{N}}F_{k}$ is a $\mu$-null set and $f_{n} \rightarrow f$ uniformly on each $F_{k}$;
        \item $\{\int_{E}^{(M)}f_{n}d\mu\}_{n \in \mathbb{N}}$ is Cauchy in $Z$ uniformly for $E \in \mathcal{H}$.
    \end{enumerate}
\end{theorem}

In this case, we shall modify the proof of the main theorems of \cite{haluskakolmogorov} and \cite{dobrakov7} to prove the following:

\begin{theorem} \label{sionandmasseintegrals}
    Assume the hypothesis of Theorem \ref{massetheoremthesis}, and let $f: \Omega \rightarrow X$ be Massé integrable on each $A \in \mathcal{H}$ with respect to $\mu: \mathcal{H} \rightarrow Y$, which is $\sigma$-additive. Suppose also $Y$ is metrizable and the semivariation of $\mu$ is finite on each $\mathcal{H}$. Then, $f$ is $S^{*}$-integrable on each $A \in \mathcal{H}$, and the two integrals agree. 
\end{theorem}

\begin{proof}

We begin by noting that, by hypothesis, we may, using Theorem \ref{massetheoremthesis}, find a countable collection of subsets of partition of $\Omega$, based on sets from $\mathcal{H}$, $\{F_{k}\}_{k \in \mathbb{N}}$, such that $\Omega \setminus \bigcup_{k \in \mathbb{N}}F_{k}$ is a $\mu$-null set and $f_{n} \rightarrow f$ uniformly on each $F_{k}$, which are each of finite semivariation and may be assumed, without loss of generality, to be increasing. Also, by the fact that the Massé integral ignores null sets in the sense of Observation \ref{ignoresnullsetsBDS} (see \cite{massesemigroup}), we may reduce the problem to the integrability and equality of integrals of $f\mathbbm{1}_{ \Omega \setminus (\bigcup_{k \in \mathbb{N}}F_{k})^{c}}$, which we shall denote by $g$. 

By the uniform convergence theorem given in Theorem \ref{uniformconvergencetheoremtopological}, which is applicable as we assume $(M1)$, $f$ is $S^{*}$ and $S$-integrable on each $F_{k+1} \setminus F_{k}$, as they are uniform limits of simple functions on these sets and the set function $\mu$ is $\sigma$-additive (therefore, each simple function is $S^{*}$ integrable), and the integral equals the Massé integral by the same argument before Corollary \ref{masseandskolmogorov}. Now using these two facts we may, for $\varepsilon > 0$ and $k \in \mathbb{N}$ both fixed, find a finite partition $P_{0} \in \mathcal{P}^{f}(F_{k+1} \setminus F_{k})$ such that for each countable partition $P_{0}^{(k)} \in \mathcal{P}(F_{k+1} \setminus F_{k})$ that refines $P_{0}$, and any choice function $\delta$, the series $\sum_{\alpha \in P_{0}^{(k)}}g(\delta_{\alpha})\mu(\alpha)$ is unconditionally summable and:

\begin{equation} \label{eqestrela}
 |\sum_{\alpha \in P_{0}^{(k)}}g(\delta_{\alpha})\mu(\alpha) - \int^{(M)}_{A \cap F_{k+1}\setminus F_{k}}gd\mu| \leq \frac{\varepsilon}{2^{k+1}},
\end{equation}

for each $A \in \mathcal{H}$ and $|\cdot|$ a fixed element of $Fs(Z)$.

Now, put $P_{0}^{F} = \bigcup_{k \in \mathbb{N}}P_{0}^{(k)}$, which is an element of $\mathcal{P}(F)$. Take now $P \geq P_{0}^{F}$ in $\mathcal{P}(F)$, and notice that $P$ is of the form:

\[P = \bigcup_{k \in \mathbb{N}}P^{(k)},\]

for $P^{(k)} \in \mathcal{P}(F_{k+1}\setminus F_{k})$ and $P^{(k)} \geq P_{0}^{(k)}$. We now show that:

\[\sum_{\alpha \in P}g(\delta_{\alpha})\mu(\alpha) = \sum_{k \in \mathbb{N}, \alpha^{(k)} \in P^{(k)}} g(\delta_{\alpha^{(k)}})\mu(\alpha^{(k)}),\]

is unconditionally summable. In fact, it is sufficient to stablish that, by Antosik interchange theorem (see \cite[Theorem 1, Section 8.5, p. 98]{matricesglidinghumpswartz}):

\begin{enumerate} [(i)]
    \item For each $k \in \mathbb{N}$, the series $\sum_{\alpha \in P^{(k)}}g(\delta_{\alpha})\mu(\alpha)$ is unconditionally summable in $Z$. This holds because $g$ is $S^{*}$ integrable on each $F_{k+1} \setminus F_{k}$ by the argument above.
    \item For any sequence of subsets $I^{(k)} \subseteq \mathbb{N}$, $k \in \mathbb{N}$, the series:

    \[\sum_{k=1}^{\infty}\sum_{j \in I^{(k)}}g(\alpha^{(k)}_{j})\mu(\alpha^{(k)}_{j}),\]

    is unconditionally summable in $Z$.
\end{enumerate}

For the second point, note that, for each fixed $I^{(k)}$ (as above) and by \ref{eqestrela},

\[|\sum_{j \in I^{(k)}}g(\alpha^{(k)}_{j})\mu(\alpha^{(k)}_{j}) - \int^{(M)}_{\bigcup_{j \in I^{(k)}}\alpha^{(k)}_{j} \cap F_{k+1}\setminus F_{k}}gd\mu| \leq \frac{\varepsilon}{2^{k+1}},\]

for each $k \in \mathbb{N}$ and $|\cdot| \in Fs(Z)$ fixed as before. Also, by using this inequality and the countable additivity of the Massé integral on $\mathcal{H}$ (as $\mu$ is countably additive - see Proposition $2.6$ of \cite{massesemigroup}), we get that

\[\sum_{k \in \mathbb{N}}\int^{(M)}_{\bigcup_{j \in I^{(k)}}\alpha^{(k)}_{j} \cap F_{k+1}\setminus F_{k}}gd\mu,\]

is unconditionally summable in $Z$. By this result and using again inequality \ref{eqestrela}, we get that:

\[|\sum_{\alpha \in P}g(\delta_{\alpha})\mu(\alpha) - \int_{A}gd\mu| \leq \varepsilon,\]

with the same notation as before. Therefore, by definition, as $|\cdot| \in Fs(Z)$ was arbitrary, $g$ is $S^{*}$ integrable on $A \in \mathcal{H}$ and this integral equals the Massé integral of $f$ on $A$.
\end{proof}

With this result, we finish the main points of equivalence regarding Lebesgue type integration and the net Riemann integral in topological vector spaces. We now pass to a discussion of another (the last) type of integration.

\section{A Saks Type Integration Procedure in Net Convergence Structures.} \label{saksintegralsection}

Now, we pass to, as indicated in the previous section, a brief discussion of the last type of integral to be defined in this work, a kind of Saks type integral as defined originally by \cite{saks} and \cite{gould}, but also more recently by \cite{fleischerchange} and  \cite{schcherbinsakstyunp}, all in topological or partially ordered cases. Basically, the idea is to integrate functions by first integrating in an upwards directed class of sets, and then take the limit of integration in sets in this class as it "expands" (as it is (upwards) directed by inclusion, therefore forming a net). This concept is a classical procedure in classical Lebesgue integration, where we may define the integral in an unbounded set as the limit of integrals on approximating (upwards) bounded sets. 

An interesting aspect of this integration procedure is that, by a Riemmanian definition, we may integrate with respect to non-additive set functions and get, by the set function induced by this process, a non-additive set function. As seen in Section \ref{sectionnetriemannintegral}, this is not the case for the net Riemann integral, which causes problems when studying the relation between it and integration procedures of Lebesgue type against non-additive set functions, such as the ones contained in the classical reference by \cite{papnull}. 

Considering the same integration structures and notation as in Section \ref{sectionnetriemannintegral}, we may directly define the main integral of this section:

\begin{definition} [Saks Integral] \label{saksintegral}

Let $f: \Omega \rightarrow X$  and  $\mu: \mathcal{H} \rightarrow Y$ be arbitrary function and set function respectively. Also, suppose that, for each $A \in \mathcal{H}$, it is given an non-empty upwards directed by inclusion class of sets $\mathcal{C}_{f,A}$, which also depends on $f$, such that:

\[\mathcal{C}_{A} \subseteq \mathcal{H},\]

and $f$ is net Riemann integrable on each $B \in \mathcal{C}_{f,A}$. Then, we say that $f$ is Saks integrable on $A$ if the net of net Riemann integrals $\{\int_{B}fd\mu\}_{B \in \mathcal{C}_{f,A}}$ has a limit in $Z$ and, in this case, we define:

\[\int_{A}^{Sa}fd\mu = \lim_{(\mathcal{C}_{f,A}, \eta_{3})}\int_{B}fd\mu,\]

and call this the Saks integral of $f$ on $A$.
    
\end{definition}

As we suppose that the space $Z$ is Hausdorff in its net convergence structure, the integral just presented is well defined (taking into account that $\mathcal{G}_{f,A}$ is directed).

Contrary to Corollary \ref{linearityofnetriemannintegral} and Theorem \ref{propertiesnetriemannintegral}, the Saks integral will not be, in general, linear. Nevertheless, some general properties can be established in this general setting. 

Denote by $L_{S}(\mu)$ the space of all functions $f: \Omega \rightarrow X$ Saks integrable on each $A \in \mathcal{H}$, with respect to the (fixed) set function $\mu: \Omega \rightarrow Y$ in the context of a (fixed) integration structure. 

\begin{theorem} [Properties of the Saks integral] \label{propertiessaksintegral}

The Saks integral has the following properties compatible with a partially ordered integral structure in Riesz spaces $X, Y, Z$:

\begin{enumerate} [(i)]
    \item Let $f \in L_{S}(\mu)$ such that $f \geq 0$ and $\mu \geq 0$. Then, 

    \[\int^{Sa}_{A}f d\mu \geq 0,\]

    for each $A \in \mathcal{H}$, i.e, the Saks integral is a positive map on the space of integrable functions.
    \item Let $\mu \geq 0 $ and $f,g \in L_{NR}(\mu)$ such that:

    \[g \geq f.\]

    Then,

    \[\int^{Sa}_{A}gd\mu \geq \int^{Sa}_{A}fd\mu,\]

     for each $A \in \mathcal{H}$, i.e, the Saks integral is an isotone map on $L_{NR}(\mu)$ equipped with the partial ordering inherited from $X$.

    \item Let $\mu \geq 0$ and $f \in L_{Sa}(\mu)$ such that $|f| \in L_{NR}(\mu)$, then:

    \[|\int^{Sa}_{A}fd\mu| \leq \int^{Sa}_{A}|f|d\mu,\]

     for each $A \in \mathcal{H}$.
\end{enumerate}
    
\end{theorem}

Its also easy to see that the Saks integral is more general than the net Riemann integral: if $f$ is net Riemann integrable on $A \in \mathcal{H}$, taking $\mathcal{G}_{f,A} = \{A\}$ produces the following:

\begin{proposition} 
    If $f: \Omega \rightarrow X$ is net Riemann integrable on $A \in \mathcal{H}$, it is Saks integrable on $A$ and the two integrals agree.
\end{proposition}
\begin{proof}
    By the choice indicated above, i.e  $\mathcal{G}_{f,A} = \{A\}$ (which is trivially upwards directed by inclusion), it follows that the net $\{\int_{B}fd\mu\}_{B \in \mathcal{G}_{f,A}}$ is constant and, by Definition \ref{netconvergencestructure}, its limit exists and equals the net Riemann integral $\int_{A}fd\mu$, which finishes the proof.
\end{proof}

Also, by ways of comparison, we point out that \cite{fleischerchange} defines a similar integral (which we call the Fleischer integral) based on the Kolmogorov S-integral with $\mathcal{H}$ being a ring and $\mathcal{G}_{f,A}$ being an ideal of sets of this ring and independent of $A$ (i.e, he takes it dependent only on $f$, and the only domain of integration being $\Omega$). In this case, we have:

\begin{proposition}
    Suppose that $\mathcal{H}$ is a ring of subsets of $\Omega$. For each $f: \Omega \rightarrow X$ and $A \in \mathcal{H}$, take $\mathcal{G}_{f,A}$ to be an ideal of sets. Then, if for each $B \in \mathcal{G}_{f,A}$, $f$ is S-integrable (See \ref{KolmogorovSintegralconvergencestructures}) on $B$, the Saks integral exists and coincides with the Fleischer integral.  
\end{proposition}

As \cite{shcherbinsakstype} integral is the same as Fleischer's but in the context of Riesz spaces and order convergence, the result above also covers his integration procedure.

As indicated in the beginning of the section, non-additive integration procedures are of special interest in the context of Saks integration. In this regard, by using some elements of the work of \cite{boccutoriecanconcave}, we may obtain that the Saks integral contains also the Choquet integral of Riesz space-valued functions (\cite{boccutoriecanchoquetsymmetric}), and the Sipos integral for the case of non-negative functions. In fact, let $\mathcal{H}$ be a $\sigma$-algebra and $C: \mathcal{H} \rightarrow Y$ be a capacity (for this notion, see \cite{boccutoriecanchoquetsymmetric}). Fix $f: \Omega \rightarrow X = \overline{\mathbb{R}}^{+}$ a function such that:

\[\Sigma^{f}_{t} = \{\omega \in \Omega: f(\omega) \geq t\} \in \mathcal{H}, \forall \ t \in \mathbb{R},\]

and define $u_{f}(t) = u(t) = C(\Sigma^{f}_{t})$. Note that $u: \mathbb{R} \rightarrow Y$ is a monotone function. In this context, we shall need the next result (see Section 3 of \cite{boccutoriecanchoquetsymmetric}):

\begin{proposition}
    Let $g: [a,b] \rightarrow Y$ be a monotone function with values in $Y$ that is supposed to be Dedekind complete and $[a,b] \subseteq \mathbb{R}$ a finite interval. Then, $u$ is Kolmogorov $S$- integrable with respect to the Lebesgue measure. 
\end{proposition}

Now, for each $f$ take, in the notation above, $\Omega = \overline{\mathbb{R}}^{+}$ with $\mathcal{H}$ the $\sigma$-algebra of Lebesgue measurable sets on $[a,b]$, and

\[\mathcal{C}_{f,\overline{\mathbb{R}}^{+}} = \mathcal{C}_{\overline{\mathbb{R}}^{+}}= \{[0,a] \subseteq \mathbb{R}^{+}: a \in \mathbb{R}\}.\]

Then, this collection is upwards directed by inclusion.

Considering the Net Riemann integral to be the Kolmogorov S-integral with respect to the Lebesgue measure as in the last proposition, we get that the net of integrals $\{\int_{0}^{a}u(t)dt\}_{[0,a] \in \mathcal{C}_{\overline{\mathbb{R}}^{+}}}$ is well defined an increasing. In case the limit exists, i.e the Saks Integral of $u$ exists on the positive extended real line, it is given by:

\[\int_{\overline{\mathbb{R}}^{+}}^{Sa}u(t)dt = \sup_{[0,a] \in \mathcal{C}_{\overline{\mathbb{R}}^{+}}}\int_{0}^{a}u(t)dt= \sup_{a>0} \int_{0}^{a}u(t)dt,\]

where the two last integrals are in the sense of the Kolmogorov S-integral, and which produced exactly the Choquet integral of $f$ (fixed above) in Definition 3.3 of \cite{boccutoriecanchoquetsymmetric} with respect to $C$. Therefore, we get, by using also Theorem 3.6 of \cite{boccutoriecanchoquetsymmetric}:

\begin{proposition} [Saks includes Choquet and (Non-Negative) Sipos]
    Let $f: \Omega \rightarrow X = \overline{\mathbb{R}}^{+}$ be a function with the properties above. Then, $f$ is Choquet integrable with respect to $C$ if and only if it is Sipos integrable with respect to $C$ and both integrals are equal to:

    \[\int_{\overline{\mathbb{R}}^{+}}^{Sa}u(t)dt,\]

    the Saks integral of $u$ over the extended real line with respect to the Lebesgue measure, formed by nets of Kolmogorov S-integrable functions. 
\end{proposition}

Some convergence theorems for this integral in a more restricted setting can be found in the article \cite{shcherbinsakstype} and \cite{gould}.

With these comments, we stop the development of this last integral, and consider some conclusions and open problems based on the integration theories so far defined. 

\section{Conclusion: Classification Schemes and (other) Open Problems.}

The various integrals presented above all possesses strengths and weaknesses, specially with respect to which classes of functions and measures are admitted as possible ingredients of the procedure. As Kolmogorov in \cite{kolmogorovuntersuchungen} argued, a general definition of an integration procedure is unlikely to contain all existing ones, for various reasons. Nonetheless, the development above provides a reason to expect that each integration procedure we can find falls within one of the following four categories:

\begin{enumerate} 
    \item Saks type-Integration,
    \item Abstract Lebesgue Integration,
    \item Monotone Type, modeled on Choquet and Sipos, Integration,
    \item Weak Type Integration, after  \cite{pettisoriginalarticle} famous model. 
\end{enumerate}

Given the variety and diversity of integration procedures exhibited in the present work, such a classification may be premature, but can generate a way of organizing all the massive existing literature about integration theories. In this respect, an open problem, which is more of a research program carried forward already by such works as \cite{sambuciniboccutocomparison}, \cite{henstockmagnum}, \cite{sionsemigroup} and others, is:

\begin{itemize}
\item Is there a general relationship between the general procedures above ? More generally, can a "taxonomy" of integration procedures be produced ?
\end{itemize}

An affirmative response to both problems seems unlikely, specially as some integration procedures are not even comparable in principle, but even this is a useful response to the general question: is there a "most general" procedure in each class ? And, if so, what are its limits ? Possibly, the main classes of integration procedures presented in this work are the most general in some sense. 

A related issue is the one of convergence: general (net or filter) convergence structures provides integration theory wide generality, but further restrictions need to be made to obtain stronger convergence and other type of theorems for this integrals. 

This provides the second (in our view, interesting) problem:

\begin{itemize}
    \item Is it possible to obtain (uniform) convergence theorems for the integrals above without requiring some specific convergence and lattice norms, only working with general properties of the filters/net involved, maybe with some minor additional condition to guarantee good continuity behaviors of the convergence ?
\end{itemize}

This is basically a question of the continuity of a (linear and isotone, sometimes) functional on spaces with general convergences and with specific properties generated by integration. That seems to require additional study of general convergence structures and many-valued topologies as is \cite{hohle}, \cite{butzmann}, \cite{gahler1}, \cite{gahler2} and references and further works therein. This seems to be related to the connection between integration in general convergence structures and Daniell integration in this context - continuity properties here are the key to obtain possible connections.  

With respect to measures and functions, the Riemmanian constructions above, not so much the Lebesgue ones, admit interval functions integration in the sense of  \cite{henstockmagnum}. This provides, in particular, a non-linear version of the integration procedures above. The domain of the measures produces another challenge:

\begin{itemize}
    \item Considering more general domains, such as orthoalgebras, difference posets, (D)-lattices and other structures presented in \cite{pulmannova}, which results above are preserved and, more interestingly, which convergence properties of measures and integrals are possible to prove ? That is, developing a non-commutative integration in the Riemmanian sense is an open problem in general.
\end{itemize}

More can be said about a host of other problems (product measures and Fubini theorems, for example) relating to general integration theory and its applications other areas, but we finish the article here, noting that much more can be said about this subject: both in its own right and regarding its relation to other areas of modern mathematics.

\section{Appendix: Uniform Approximation of Measurable Functions with Values in Riesz Spaces.}

Using the results of Section \ref{pavlakosintegralsection}, some interesting integrability results can be achieved for classes of (Riesz space-valued) functions that are uniform order limits of sequences of simple and/or countably valued functions and measurable in some way. There appear to be only limited results in this area, with main references being Section 4 of \cite{pavlakosintegration} and Proposition 2.1 of \cite{potocky2}. To not interrupt the main flux of the text, we prove new results in a wider setting in this appendix. 

In fact, we propose a generalization of the results of these two references to specific types of locally solid topological Riesz spaces. In this case, let $X$ be Riesz space with a topology induced by a (possibly uncountable) collection of Riesz $F$-norms $\{p_{l}\}_{l \in L}$ in such a way that the sets:

\[\bigcap_{i=1}^{n}\{y: p_{l_{i}}(x-y) \leq \varepsilon\},\]

which depends on $x \in X$, $n \in \mathbb{N}$, $\varepsilon > 0$ and $\{l_{i}\}_{i=1}^{n}$, forms a base of neighborhoods of a (Hausdorff) topology of $X$. Its easy to show that such sets are solid, in such a way that the induced topology is indeed a Riesz space locally solid topology (see \cite[p.~296]{zaanen2}).

Now, we define the concept of a strong order unit for the previously defined class of Riesz space (see \cite{pavlakosintegration}):

\begin{definition} [Strong Order Unit]

We say that a Riesz $F$-norm (on a Riesz space $X$) $p: X \rightarrow [0, \infty)$ has a strong order unit $u > 0$ if $p(x) \leq k$ implies $|x| \leq ku$, for each $k \in \mathbb{R}_{+}$ and $x \in X$. We then say that a collection of Riesz $F$-norms $\{p_{l}\}_{l \in L}$ has a strong order unit if every element has it. 
\end{definition}

With the help of these concepts we can now state and prove the following result (compare with the two references above):

\begin{proposition} [Approximation by Countably-Valued Functions]

Let $(X, \leq, \tau)$ be a locally solid Riesz space with topology $\tau$ induced by a collection $\{p_{l}\}_{l \in L}$ of Riesz $F$-norms. Suppose that $X$ is first countable and has a strong order unit with respect to its collection of Riesz $F$-norms. Then, a Borel-measurable (with respect to the class of open sets in the order topology of $X$) function $f: \Omega \rightarrow X$ is a order uniform limit of countably valued functions.
\end{proposition}

\begin{proof}
    As $X$ is first countable, it is separable and we can find a countable collection $C = \{c_{k}\}_{k \in \mathbb{N}} \subseteq X$ such that:

    \[f(\Omega) \subseteq \overline{C} = X.\]

Then, denote by $C_{n,k}$ the subbasis element of the metrizable topology of $X$ consisting of the closed ball with radius $\frac{1}{n}$ of the element $c_{k}$ of $C$, with respect to the Riesz $F$-seminorm that induces its metrizable topology (by the Birkhoff-Kakutani theorem), which we denote by $p^{*}$. In this case, we make a disjoint sequence of elements by taking:

\[B_{n,1} = C_{n,1}, \ B_{n,k} = C_{n,k} - \bigcup_{i < k}C_{n,i}, k \geq 2;\]

in such a way that $\Omega = \bigcup_{k \in \mathbb{N}}B_{n,k}$ for every $n \in \mathbb{N}$. 

Now, define the following sequence of countably valued functions defined on $\Omega$ with values in $X$:

\[f_{n}(t) = \sum_{k=1}^{\infty}c_{k}\mathbbm{1}(t)_{C_{n,k}},\]

for which is valid that:

\[p^{*}(f(\omega)-f_{n}(\omega)) \leq \frac{1}{n},\]

for any $(\omega, n) \in \Omega \times \mathbb{N}$.

Hence, by the existence of a strong order unit $u$,

\[|f(\omega)-f_{n}(\omega)| \leq \frac{1}{n}u,\]

for any $(\omega, n) \in \Omega \times \mathbb{N}$. This shows that $f_{n}$ converges order uniformly to $f$ on $\Omega$.
\end{proof}

With the same proof, we may stablish the metrizable case:

\begin{proposition} [Approximation by Countably-Valued Functions - II]

Let $(X, \leq, \tau)$ be a metrizable locally solid Riesz space with topology $\tau$. Suppose that $X$ is separable  and has a strong order unit. Then, a Borel-measurable function $f: \Omega \rightarrow X$ is a order uniform limit of countably valued functions.
\end{proposition}

As a final note, it would be interesting to know the maximal (i.e, sufficient and necessary) conditions imposed into the Riesz space $X$ such that interesting classes of (measurable) functions are uniform limits, or even quasi-uniform limits, of nets of elementary/simple functions.

\section*{Acknowledgement}

This study was funded in part by the Coordination for the Improvement of Higher Education Personnel - CAPES - Funding Code 001. The author would like to thank Professors Francesco Russo (ENSTA - Paris), Christophe Gallesco (UNICAMP), David Fremlin (Exeter), Lech Drewnowski (Poznan) and Alberto Castejón (Vigo) for interesting conversations and remarks regarding integration theory.

\bibliography{biblio}

\end{document}